\documentclass[12pt]{article}
\usepackage[utf8]{inputenc}
\usepackage{amsmath, amsthm, amssymb,fullpage, xcolor}
\usepackage{bbm}
\usepackage{boxedminipage}
\usepackage{libertine}
\usepackage{libertinust1math}
\usepackage[T1]{fontenc}
\usepackage[font=small,labelfont=bf]{caption}
\usepackage{hyperref}
\usepackage{microtype}
\usepackage[chicago]{xellipsis}
\usepackage{pgf,tikz,pgfplots}
\pgfplotsset{compat=1.14}
\usepackage{mathrsfs}
\usepackage{enumitem}
\usepackage{thmtools}
\usepackage{thm-restate}
\usepackage{hyperref}
\usepackage{cleveref}
\usepackage{booktabs}

\usepackage{marginnote}
\usepackage{threeparttable, tablefootnote}
\usepackage{float}

\usetikzlibrary{arrows}
\hypersetup{
    colorlinks=true, 
    citecolor=blue, 
    linkcolor=blue, 
    urlcolor=red, 
    linktoc=all 
}

\newtheorem{theorem}{Theorem}[section]
\newtheorem{lemma}[theorem]{Lemma}
\newtheorem{corollary}[theorem]{Corollary}
\newtheorem{proposition}[theorem]{Proposition}
\theoremstyle{definition}
\newtheorem{definition}[theorem]{Definition}
\newtheorem{observation}[theorem]{Observation}
\newtheorem{remark}[theorem]{Remark}

\newtheorem{conjecture}[theorem]{Conjecture}

\newcommand{\poly}{\mathrm{poly}}
\newcommand{\polylog}{\mathrm{polylog}}
\newcommand{\fl}{\mathsf{fl}}

\newcommand{\outdef}{\eta}
\newcommand{\indef}{\beta}
\newcommand{\sgnerr}{\beta}
\newcommand{\eigerr}{\delta}
\newcommand{\eigprob}{\theta}
\newcommand{\spnerr}{\eta}
\newcommand{\spr}{\mathsf{spr}}
\newcommand{\re}{\operatorname{Re}}
\renewcommand{\Re}{\operatorname{Re}}

\renewcommand{\u}{\mach}
\newcommand{\R}{\mathbb{R}}
\newcommand{\C}{\mathbb{C}}
\newcommand{\dist}{\mathsf{dist}}

\renewcommand{\P}{\mathbb{P}}

\newcommand{\bm}[1]{\begin{bmatrix} #1\end{bmatrix}}
\newcommand{\vol}{\mathrm{vol}}
\newcommand{\dE}{\mathbb{E}}

\newcommand{\eps}{\epsilon}

\newcommand{\Tr}{\mathrm{Tr}}

\newcommand{\gap}{\mathrm{gap}}

\newcommand{\cp}[1]{\mathsf{C}^{+}_{#1}}
\newcommand{\cm}[1]{\mathsf{C}^{-}_{#1}}
\newcommand{\cpm}[1]{\mathsf{C}_{#1}}
\newcommand{\ap}[2]{\mathsf{A}^{+}_{#1,#2}}
\newcommand{\am}[2]{\mathsf{A}^{-}_{#1,#2}}
\newcommand{\apm}[2]{\mathsf{A}_{#1,#2}}
\newcommand{\newton}{g}
\newcommand{\G}{\mathsf{G}}
\newcommand{\dee}{\textup{d}}
\newcommand{\mach}{\textbf{\textup{u}}}

\newcommand{\sgn}{\mathrm{sgn}}

\newcommand{\SGN}{\mathsf{SGN}}

\newcommand{\SPAN}{\mathsf{DEFLATE}}
\newcommand{\EIG}{\mathsf{EIG}}

\newcommand{\RURV}{\mathsf{RURV}}
\newcommand{\N}{\mathsf{N}}
\newcommand{\cn}{c_{\mathsf{N}}}

\newcommand{\grid}{\mathsf{grid}}
\renewcommand{\gets}{\leftarrow}

\newcommand{\spa}{\mathrm{deflate}}
\newcommand{\rurv}{\mathrm{rurv}}
\newcommand{\qr}{\mathrm{QR}}
\newcommand{\rank}{\mathrm{rank}}
\newcommand{\range}{\mathrm{range}}

\newcommand{\muav}{\mu_{F, \mathsf{av}}}
\newcommand{\calV}{\mathcal{V}}

\renewcommand{\u}{\mach}

\newcommand{\MM}{\mathsf{MM}}
\newcommand{\pmult}{\mu_{\MM}}
\newcommand{\INV}{\mathsf{INV}}
\newcommand{\pinv}{\mu_{\INV}}
\newcommand{\cinv}{c_\INV}
\newcommand{\pqr}{\mu_\QR}
\newcommand{\QR}{\mathsf{QR}}

\newcommand{\SPLIT}{\mathsf{SPLIT}}

\newcommand{\g}{\mathsf{g}}
\newcommand{\expbound}{2e^{-2 n}}
\newcommand{\expboundd}{2e^{-t^2 n}}

\DeclareMathOperator{\diag}{diam}

\allowdisplaybreaks

\newcommand{\SHATTER}{\mathsf{SHATTER}}
\title{ Pseudospectral Shattering, the Sign Function, and Diagonalization in Nearly Matrix Multiplication Time}
\author{ Jess Banks\thanks{Department of Mathematics, UC Berkeley, Berkeley, CA 94720.}  \thanks{Supported by  the NSF Graduate Research Fellowship Program under Grant DGE-1752814.} \\ \normalsize jess.m.banks@berkeley.edu \and  Jorge Garza-Vargas \footnotemark[1] \\\normalsize jgarzavargas@berkeley.edu \and   Archit Kulkarni \footnotemark[1]  \\ \normalsize akulkarni@berkeley.edu \and  Nikhil Srivastava \footnotemark[1]  \thanks{Supported by NSF Grant CCF-1553751.} \footnote{Corresponding Author.}\\ \normalsize nikhil@math.berkeley.edu }

\date{\normalsize \today}
\begin{document}

	\maketitle
	\begin{abstract}
{\footnotesize	We exhibit a randomized algorithm which given a square matrix $A\in \C^{n\times n}$ with $\|A\|\le 1$ and $\delta>0$, computes with high probability an invertible $V$  and diagonal $D$  such that
	$	    \|A-VDV^{-1}\|\le \delta$	    
	 using $O(T_\MM(n)\log^2(n/\delta))$ arithmetic operations, in finite arithmetic with $O(\log^4(n/\delta)\log n)$ bits of precision. The computed similarity $V$ additionally satisfies $\|V\|\|V^{-1}\|\le O(n^{2.5}/\delta)$.  Here $T_\MM(n)$ is the number of arithmetic operations required to multiply two $n\times n$ complex matrices numerically stably, known to satisfy $T_\MM(n)=O(n^{\omega+\eta})$ for every $\eta>0$ where $\omega$ is the exponent of matrix multiplication (Demmel et al., Numer. Math., 2007). The algorithm is a variant of the spectral bisection algorithm in numerical linear algebra (Beavers Jr. and Denman, Numer. Math., 1974) with a crucial Gaussian perturbation preprocessing step.  Our result significantly improves the previously best known provable running times of $O(n^{10}/\delta^2)$ arithmetic operations  for diagonalization of general matrices (Armentano et al., J. Eur. Math. Soc., 2018), and (with regards to the dependence on $n$) $O(n^3)$ arithmetic operations for Hermitian matrices (Dekker and Traub, Lin. Alg. Appl., 1971). It is the first algorithm to achieve nearly matrix multiplication time for diagonalization in any model of computation (real arithmetic, rational arithmetic, or finite arithmetic), thereby matching the complexity of other dense linear algebra operations such as inversion and $QR$ factorization upto polylogarithmic factors.
	
    The proof rests on two new ingredients. (1) We show that adding a small complex Gaussian perturbation
    to {\em any} matrix splits its pseudospectrum into $n$ small well-separated components. In particular, this implies that the eigenvalues of the perturbed matrix have a large minimum gap, a property of independent interest in random matrix theory. (2) We give a rigorous analysis
    of Roberts' Newton iteration method (Roberts, Int. J. Control, 1980) for computing the sign function of a matrix in finite arithmetic, itself an open problem
    in numerical analysis since at least 1986. \\
    {\em Keywords: Linear Algebra, Random Matrix Theory, Numerical Analysis, Computational Complexity.} \\ { \em AMS MSC 2020 Codes: 60B20, 65F15, 68Q25.} \\
    Communicated by Peter B\"urgisser.
    \par }
    
	\end{abstract}
\newpage
	\tableofcontents
	\newpage
	\section{Introduction}\label{sec:intro}


We study the algorithmic problem of approximately finding all of
the eigenvalues {and eigenvectors} of a given arbitrary $n\times n$
complex matrix.  While this problem is quite well-understood in the special case of
Hermitian matrices (see, e.g., \cite{parlett1998symmetric}), the general
non-Hermitian case has remained mysterious from a theoretical standpoint even
after several decades of research. In particular, the currently best known {\em
provable} algorithms for this problem run in time $O(n^{10}/\delta^2)$
\cite{abbcs} or $O(n^c\log(1/\delta))$ \cite{cai1994computing}
with $c\ge 12$ where $\delta>0$ is the desired accuracy, depending on the model
of computation and notion of approximation considered.\footnote{A detailed
discussion of these and other related results appears in Section
\ref{sec:related}.} To be sure, the non-Hermitian case is well-motivated:
coupled systems of differential equations, linear dynamical systems in control
theory, transfer operators in mathematical physics, and the nonbacktracking
matrix in spectral graph theory are but a few situations where finding the
eigenvalues {\em and eigenvectors} of a non-Hermitian matrix is important. 

The key difficulties in dealing with non-normal matrices are the interrelated
phenomena of {\em non-orthogonal eigenvectors} and {\em spectral
instability}, the latter referring to extreme sensitivity of the eigenvalues and invariant subspaces to perturbations of the matrix. Non-orthogonality slows down convergence of
standard algorithms such as the power method, and spectral instability can
force the use of very high precision arithmetic, also leading to slower
algorithms. Both phenomena together make it difficult to  reduce the
eigenproblem to a subproblem by ``removing'' an eigenvector or invariant
subspace, since this can only be done approximately and one must control the
spectral stability of the subproblem in order to be able to rigorously reason about it. 

In this paper, we overcome these difficulties by identifying and leveraging a
phenomenon we refer to as {\em pseudospectral shattering}: adding a small
complex Gaussian perturbation to any matrix typically yields a matrix with
well-conditioned eigenvectors and a large minimum gap between the eigenvalues,
implying spectral stability. Previously, even the existence of such a regularizing perturbation with favorable parameters was not known \cite{davies2007approximate}. This result builds on the recent solution of
Davies' conjecture \cite{banks2019gaussian}, and is of independent interest in random
matrix theory, where minimum eigenvalue gap bounds in the non-Hermitian case were previously only known for i.i.d. models \cite{shi2012smallest,ge2017eigenvalue}. 

We complement the above by proving that a variant of the well-known spectral
bisection algorithm in numerical linear algebra \cite{beavers1974new} is both fast and numerically
stable when run on a pseudospectrally shattered matrix --- we call an iterative algorithm {\em numerically stable} if it can be implemented using finite precision arithmetic with polylogarithmically many bits, corresponding to a
dynamical system whose trajectory to the approximate solution is robust to
adversarial noise (see, e.g. \cite{smale1997complexity}). The key step in the
bisection algorithm is computing the {\em sign function} of a matrix, a problem
of independent interest in many areas such including control theory and approximation theory \cite{kenney1995matrix}. Our main algorithmic contribution is a rigorous
analysis of the well-known Newton iteration method \cite{roberts1980linear}
for computing the sign function {\em in finite arithmetic}, showing that it
converges quickly and numerically stably on matrices for which the sign
function is well-conditioned, in particular on pseudospectrally shattered ones.

The end result is an algorithm which reduces the general diagonalization
problem to a {polylogarithmic} (in the desired accuracy and dimension $n$) number of invocations of standard numerical
linear algebra routines (multiplication, inversion, and QR factorization), each of which is reducible to matrix multiplication
\cite{demmel2007fast}, yielding a nearly matrix multiplication runtime for the whole algorithm. This improves on the previously best known running time of $O(n^3+n^2\log(1/\delta))$ arithmetic operations even in the Hermitian case (\cite{dekker1971shifted}, see also \cite{ hoffmann1978new, parlett1998symmetric}), and yields the same improvement for the related problem of computing the singular value decomposition of a matrix.

We now proceed to give precise mathematical formulations of the eigenproblem and
computational model, followed by statements of our results and a detailed
discussion of related work.
\subsection{Problem Statement}\label{sec:models}
An {\em eigenpair} of a matrix $A\in \C^{n\times n}$ is a tuple $(\lambda, v)\in \C\times \C^n$ such that 
$$Av=\lambda v,$$
and $v$ is normalized to be a unit vector.  The {\em eigenproblem} is the
problem of finding a maximal set of linearly independent eigenpairs $(\lambda_i,v_i)$ of a given matrix $A$; note that an eigenvalue may appear more than
once if it has geometric multiplicity greater than one.  In the case when $A$
is diagonalizable, the solution consists of exactly $n$ eigenpairs, and if $A$
has distinct eigenvalues then the solution is unique, up to the phases of the $v_i$. 

\subsubsection{Accuracy and Conditioning}
Due to the Abel-Ruffini theorem, it is impossible to have a finite-time
algorithm which solves the eigenproblem exactly using arithmetic operations and
radicals.  Thus, all we can hope for is {\em approximate} eigenvalues and
eigenvectors, up to a desired accuracy $\delta>0$. There are two standard notions of approximation. We assume $\|A\|\le 1$ for normalization, where throughout this work, $\Vert \cdot \Vert$ denotes the spectral norm (the $\ell^2 \to \ell^2$ operator norm).\\

\noindent {\bf Forward Approximation.} Compute pairs $(\lambda_i',v_i')$ such that
$$|\lambda_i-\lambda_i'|\le\delta\quad\textrm{and}\quad \|v_i-v_i'\|\le\delta$$
for the true eigenpairs $(\lambda_i,v_i)$, i.e., find a solution close to the exact solution. 
This makes sense in contexts where the exact solution is meaningful; e.g. the
matrix is of theoretical/mathematical origin, and unstable (in the entries) quantities such
as eigenvalue multiplicity can have a significant meaning.\\

\noindent {\bf Backward Approximation.} Compute $(\lambda_i',v_i')$ which are the exact eigenpairs of a matrix $A'$ satisfying 
$$\|A'-A\|\le \delta,$$
i.e., find the exact solution to a nearby problem. This is the appropriate and standard notion in scientific computing, where the matrix is of physical or empirical origin and is not assumed to be known exactly (and even if it were, roundoff error would destroy this exactness). Note that since diagonalizable matrices are dense in $\C^{n\times n}$, one can hope to always find a complete set of eigenpairs for some nearby $A'=VDV^{-1}$, yielding an {\em approximate diagonalization} of $A$:\marginnote{\emph{approximate diagonalization}}
\begin{equation}\label{eqn:approxdiag} \|A-VDV^{-1}\|\le \delta.\end{equation}

Note that the eigenproblem in either of the above formulations is {\em not} easily reducible to the problem of
computing eigenvalues, since they can only be computed approximately and it is
not clear how to obtain approximate eigenvectors from approximate eigenvalues. 
We now introduce a condition number for the eigenproblem, which measures the sensitivity of the eigenpairs of a matrix to perturbations and allows us to relate its forward and backward approximate solutions.\\

\noindent {\bf Condition Numbers.}
For diagonalizable $A$, the {\em eigenvector condition number} of $A$, denoted $\kappa_V(A)$, is defined as:
\begin{equation}\label{eqn:kappavdef}
    \kappa_V(A):=\inf_{V} \|V\|\|V^{-1}\|, \marginnote{$\kappa_V(A)$}
\end{equation}
where the infimum is over all invertible $V$ such that $A=VDV^{-1}$ for some diagonal $D$, and its {\em minimum eigenvalue gap} is defined as:
$$
    \gap(A):=\min_{i\neq j}|\lambda_i(A)-\lambda_j(A)|,
    \marginnote{$\gap(A)$}
$$
where $\lambda_i$ are the eigenvalues of $A$ (with multiplicity).
\newcommand{\kappaeig}{\kappa_{\mathrm{eig}}}
We define the {\em condition number of the eigenproblem} to be\footnote{This quantity is inspired by but not identical to the ``reciprocal of the distance to ill-posedness'' for the eigenproblem considered by Demmel \cite{demmel1987condition}, to which it is polynomially related. See also \cite{van1987estimating} for another natural  definition of eigenvector condition number similar in spirit to that of Demmel.}: 
\begin{equation}\label{eqn:kappaeig} 
\kappaeig(A):=\frac{\kappa_V(A)}{\gap(A)}\in [0,\infty].\marginnote{$\kappaeig$}
\end{equation}
It follows from the proposition below (whose proof appears in Section \ref{sec:prelimspectral}) that a $\delta$-backward approximate solution of the eigenproblem is a $6n\kappaeig(A)\delta$-forward approximate solution.\footnote{ In fact, it can be shown that $\kappaeig(A)$ is related by a $\poly(n)$ factor to the smallest constant for which \eqref{eqn:backwardforward} holds for all sufficiently small $\delta>0$.}
\begin{proposition} \label{prop:eigperturb}
If $\|A\|,\|A'\|\le 1$, $\|A-A'\|\le \delta$, and $\{(v_i,\lambda_i)\}_{i\le n}$, $\{(v_i',\lambda_i')\}_{i\le n}$ are eigenpairs of $A,A'$ with distinct eigenvalues, and $\delta < \frac{\gap(A)}{8 \kappa_V(A)}$, then
\begin{equation}\label{eqn:backwardforward}
    \|v_i'-v_i\|\le 2n\kappaeig(A) \delta \quad \text{ and }\quad  \|\lambda_i'-\lambda_i\|\le \kappa_V(A) \delta \le 2\kappaeig(A) \delta\quad \forall i=1,\ldots,n,
\end{equation} after possibly multiplying the $v_i$ by phases.
\end{proposition}

Note that $\kappaeig=\infty$ if and only if $A$ has a double eigenvalue; in this case, a relation like \eqref{eqn:backwardforward} is not possible since different infinitesimal changes to $A$ can produce macroscopically different eigenpairs.


In this paper we will present a backward approximation approximation for the eigenproblem with running time scaling polynomially in $\log(1/\delta)$, which by \eqref{eqn:backwardforward} yields a forward approximation algorithm with running time scaling polynomially in $\log(1/\kappaeig\delta)$. \\

\begin{remark}[Multiple Eigenvalues] A backward approximation algorithm for the eigenproblem can be used to accurately find bases for the eigenspaces of matrices with multiple eigenvalues, but quantifying the forward error requires introducing condition numbers for invariant subspaces rather than eigenpairs. A standard treatment of this can be found in any numerical linear algebra textbook, e.g. \cite{demmel1997applied}, and we do not discuss it further in this paper for simplicity of exposition.
\end{remark}

\subsubsection{Models of Computation}
These questions may be studied in various computational models: exact {\em real arithmetic} (i.e., infinite precision), {\em variable precision
rational arithmetic} (rationals are stored exactly as numerators and denominators), and {\em finite precision arithmetic} (real numbers are rounded to a fixed number of bits which may depend on the input size and accuracy). Only the last two models yield
actual Boolean complexity bounds, but introduce a second source of error stemming from the fact that 
computers cannot exactly represent real numbers. 

We study the third model in this paper, axiomatized as follows.\\

\noindent {\bf Finite Precision Arithmetic.} We use the standard floating point axioms from
\cite{higham2002accuracy}. Numbers are stored and manipulated approximately up
to some machine precision $\mach:=\mach(\delta,n)>0$, which for us will depend on the instance size $n$ and desired accuracy $\delta$. This means every number $x\in\C$ is stored as $\fl(x)=(1+\Delta)x$ for some adversarially chosen $\Delta\in \C$ satisfying $|\Delta|\le \mach$, and each arithmetic operation $\circ\in \{+,-,\times,\div\}$ is guaranteed to yield an output satisfying 
$$ 
    \fl(x\circ y) = (x\circ y)(1+\Delta)\quad |\Delta|\le \mach.
$$ 
It is also standard and convenient to assume that we can evaluate $\sqrt{x}$ for any $x\in \R$, where again $\fl(\sqrt x) = \sqrt x (1 + \Delta)$ for $|\Delta| \le \mach$. 

Thus, the outcomes of all
operations are adversarially noisy due to roundoff. The bit lengths of numbers stored in this form remain fixed at $\lg(1/\mach)$\marginnote{$\lg$}, where $\lg$ denotes the logarithm base 2. The {\em bit complexity} of an algorithm is therefore the number of arithmetic operations times $O^*(\log(1/\mach))$, the running time of standard floating point arithmetic, where the $*$ suppresses $\log\log(1/\mach)$ factors. We will state all running times in terms of arithmetic operations accompanied by the required number of bits of precision, which thereby immediately imply bit complexity bounds.\\

\begin{remark}[Overflow, Underflow, and Additive Error] Using $p$ bits for the exponent in the floating-point representation allows one to represent numbers with magnitude in the range $[2^{-2^p},2^{2^p}]$. It can be easily checked that all of
	the nonzero numbers, norms, and condition numbers appearing during the
	execution of our algorithms lie in the range
	$[2^{-\lg^c(n/\delta)},2^{\lg^c(n/\delta)}]$ for some small $c$, so 
	overflow and underflow do not occur. In fact, we could have
	analyzed our algorithm in a computational model where every number is
	simply rounded to the nearest rational with denominator $2^{\lg^c(n/\delta)}$---corresponding to {\em additive} arithmetic errors. We have chosen
	to use the multiplicative error floating point model since it is the
	standard in numerical analysis, but our algorithms do not exploit any
	subtleties arising from the difference between the two models.
\end{remark}
The advantages of the floating point model are that it is realistic and potentially yields
very fast algorithms by using a small number of bits of precision (polylogarithmic in $n$ and $1/\delta$), in contrast to rational arithmetic,
where even a simple operation such as inverting an $n\times n$ integer matrix
requires $n$ extra bits of precision (see, e.g., Chapter 1 of
\cite{grotschel2012geometric}).  An iterative algorithm that can be implemented
in finite precision (typically, polylogarithmic in the input size and desired accuracy) is called {\em numerically stable}.

The disadvantage of the model is that it is only possible to compute forward
approximations of quantities which are {\em well-conditioned} in the input --- in particular, discontinuous quantities such as eigenvalue multiplicity cannot be computed in the floating point model, since it is not even assumed that the input is stored exactly. 

\subsection{Results and Techniques}
In addition to $\kappaeig$, we will need some more refined quantities to measure the stability of the eigenvalues and eigenvectors of a matrix to perturbations, and to state our results regarding it. The most important of these is the $\epsilon$-pseudospectrum, defined for any $\epsilon>0$ and $M\in \C^{n\times n}$ as:
\begin{align}
   \label{eqn:pseudodef1} \Lambda_\epsilon(M) &:= \left\{ \lambda\in \C : \lambda \in \Lambda(M + E) \text{ for some }\|E\| < \epsilon\right\} \marginnote{$\Lambda_\eps$} \\
   \label{eqn:pseudodef2}
    &= \left\{\lambda \in \C : \left\|(\lambda - M)^{-1}\right\| > 1/\epsilon \right\}
\end{align}
where $\Lambda(\cdot)$ denotes the spectrum of a matrix. The equivalence of (\ref{eqn:pseudodef1}) and (\ref{eqn:pseudodef2}) is simple and can be found in the excellent book \cite{trefethen2005spectra}.\\ 

\noindent {\bf Eigenvalue Gaps, $\kappa_V$, and Pseudospectral Shattering.} 
The key probabilistic result of the paper is that a random {\em complex} Gaussian perturbation of any matrix yields a nearby matrix with 
large minimum eigenvalue gap and small $\kappa_V$.
\begin{theorem}[Smoothed Analysis of $\gap$ and $\kappa_V$]\label{thm:smoothed} Suppose $A\in \C^{n\times n}$ with $\|A\|\le 1$, and $\gamma\in (0,1/2)$. Let $G_n$ be 
an $n\times n$ matrix with i.i.d. complex Gaussian $N(0,1_\C/n)$ entries, and let $X:=A+\gamma G_n$. Then $$\kappa_V(X)\le \frac{n^2}{\gamma}, \quad  \gap(X)\ge \frac{\gamma^4}{n^5}, \quad \text{and}\quad  \|G_n\|\leq 4,$$ with probability at least $1-12/n$.\end{theorem}

The proof of Theorem \ref{thm:smoothed} appears in Section \ref{sec:smoothed}. The key idea is to first control $\kappa_V(X)$ using \cite{banks2019gaussian}, and then observe that for a matrix with small $\kappa_V$, two eigenvalues of $X$ near a complex number $z$ imply a small {\em second-least} singular value of $z-X$, which we are able to control.

In Section \ref{sec:shatter} we develop the notion of {\em pseudospectral shattering}, which is implied by Theorem \ref{thm:smoothed} and says roughly that the pseudospectrum consists of $n$ components that lie in separate squares of an appropriately coarse grid in the complex plane. This is useful in the analysis of the spectral bisection algorithm in Section \ref{sec:spectralbisec}.\\

\noindent {\bf Matrix Sign Function.} The sign function of a number $z\in \C$ with $\Re(z)\neq 0$ is defined as $+1$ if $\Re(z)>0$ and $-1$ if $\Re(z)<0$. The {\em matrix sign function} \marginnote{$\sgn(\cdot)$} of a matrix $A$ with Jordan normal form
$$ 
    A = V\bm{N & \\ & P}V^{-1},
$$
where $N$ (resp. $P$) has eigenvalues with strictly negative (resp. positive) real part, is defined as
$$ 
    \sgn(A) = V\bm{-I_N & \\ & I_P} V^{-1},
$$
where $I_P$ denotes the identity of the same size as $P$. The sign function is undefined for matrices with eigenvalues on the imaginary axis. Quantifying this discontinuity, Bai and Demmel \cite{bai1998using} defined the following condition number for the sign function: 
\newcommand{\kappasign}{\kappa_{\mathrm{sign}}}
\begin{equation}
    \kappasign(M) := \inf\left\{  1/\epsilon^2 : \Lambda_\epsilon(M) \textrm{ does not intersect the imaginary axis} \right\},
    \marginnote{$\kappasign$}
\end{equation}
and gave perturbation bounds for $\sgn(M)$ depending on $\kappasign$.

Roberts \cite{roberts1980linear} showed that the simple iteration 
\begin{equation}\label{eqn:robertsnewton} A_{k+1}=\frac{A_k+A_k^{-1}}{2}\end{equation}
converges  globally and  quadratically to $\sgn(A)$ in exact arithmetic,
but his proof relies on the fact that all iterates of the algorithm are
simultaneously diagonalizable, a property which is destroyed in finite
arithmetic since inversions can only be done approximately.\footnote{Doing the inversions exactly in rational arithmetic could require numbers of bit length $n^k$ for $k$ iterations, which will typically not even be polynomial.} In Section \ref{sec:matrix-sign} we show that this iteration is indeed convergent when implemented in finite arithmetic for matrices with small $\kappasign$, given a numerically stable matrix inversion algorithm.  This leads to the following result:
\begin{theorem}[Sign Function Algorithm] \label{thm:signintro} There is a deterministic algorithm $\SGN$ which on input an $n \times n$ matrix $A$ with $\|A\|\le 1$, a number $K$ with $K \ge \kappasign(A)$, and a desired accuracy $\beta \in (0, 1/12)$,  outputs an approximation $\SGN(A)$ with 
$$\|\SGN(A)-\sgn(A)\|\le \beta,$$
in 
\begin{equation}
O( (\log K + \log \log (1/\beta)) T_{\INV}(n) )
\end{equation}
 arithmetic operations on a floating point machine with 
 $$ \lg(1/\mach) = O(\log n \log^3 K (\log K + \log(1/\beta)))$$  
 bits of precision, where $T_\INV(n)$ denotes the number of arithmetic operations used by a numerically stable matrix inversion algorithm (satisfying Definition \ref{def:inv}).
\end{theorem}

The main new idea in the proof of Theorem \ref{thm:signintro} is to control the evolution of the pseudospectra $\Lambda_{\epsilon_k}(A_k)$ of the iterates with appropriately decreasing (in $k$) parameters $\epsilon_k$, using a sequence of carefully chosen shrinking contour integrals in the complex plane. The pseudospectrum provides a richer induction hypothesis than scalar quantities such as condition numbers, and allows one to control all quantities of interest using the holomorphic functional calculus. This technique is introduced in Sections \ref{sec:circlesappolonius} and \ref{sec:exactarithnewton}, and carried out in finite arithmetic in Section \ref{sec:finitearithnewton}, yielding Theorem \ref{thm:signintro}. \\

\noindent {\bf Diagonalization by Spectral Bisection.} Given an algorithm for computing the sign function, there is a natural and well-known approach to the eigenproblem pioneered in \cite{beavers1974new}. The idea is that the matrices $(I\pm\sgn(A))/2$ are spectral projectors onto the invariant subspaces corresponding to the eigenvalues of $A$ in the left and right open half planes, so if some shifted matrix $z + A$ or $z + iA$ has roughly half its eigenvalues in each half plane, the problem can be reduced to smaller subproblems appropriate for recursion. 

The two difficulties in carrying out the above approach are: (a) efficiently computing the sign function (b) finding a balanced splitting along an axis that is well-separated from the spectrum. These are nontrivial even in exact arithmetic, since the iteration \eqref{eqn:robertsnewton} converges slowly if (b) is not satisfied, even without roundoff error. We use Theorem \ref{thm:smoothed} to ensure that a good splitting always exists after a small Gaussian perturbation of order $\delta$, and Theorem \ref{thm:signintro} to compute splittings efficiently in finite precision. Combining this with well-understood techniques such as rank-revealing QR factorization, we obtain the following theorem, whose proof appears in Section \ref{sec:eigproof}.

\begin{theorem}[Backward Approximation Algorithm] \label{thm:bkwd} There is a randomized algorithm $\EIG$ which on input any matrix $A\in \C^{n\times n}$ with $\|A\|\le 1$ and a desired accuracy parameter $\delta>0$ outputs a diagonal $D$ and invertible $V$ such that 
$$ \quad \|A-VDV^{-1}\|\le \delta \quad\mathrm{and}\quad \kappa(V) \le 32n^{2.5}/\delta$$
in 
$$O\left(T_\MM(n)\log^2\frac{n}{\delta}\right)$$
arithmetic operations on a floating point machine with $$O(\log^4(n/\delta)\log n)$$ bits of precision, with probability at least $1-14/n$. Here $T_\MM(n)$ refers to the running time of a numerically stable matrix multiplication algorithm (detailed in Section \ref{sec:logstable}). 
\end{theorem}

Since there is a correspondence in terms of the condition number between backward and forward approximations, and as it is customary in numerical analysis, our discussion revolves around backward approximation guarantees. For convenience of the reader we write down below the explicit guarantees that one gets by using   \eqref{eqn:backwardforward} and invoking $\EIG$ with accuracy $\frac{\delta}{6n \kappaeig}$.
\begin{corollary}[Forward Approximation Algorithm] \label{thm:fwd} There is a randomized algorithm which on input any matrix $A\in \C^{n\times n}$ with $\|A\|\le 1$,  a desired accuracy parameter $\delta>0$, and an estimate $K\ge \kappaeig(A)$ outputs a $\delta-$forward approximate solution to the eigenproblem for $A$ in 
$$O\left(T_\MM(n)\log^2\frac{n K}{\delta}\right)$$
arithmetic operations on a floating point machine with $$O(\log^4(nK/\delta)\log n)$$ bits of precision, with probability at least $1-1/n-12/n^2$. Here $T_\MM(n)$ refers to the running time of a numerically stable matrix multiplication algorithm (detailed in Section \ref{sec:logstable}). 
\end{corollary}

\begin{remark}[Accuracy vs. Precision] The gold standard of ``backward stability'' in numerical analysis postulates that
$$ \log(1/\mach) = \log(1/\delta) + \log(n),$$
i.e., the number of bits of precision is linear in the number of bits of accuracy. The relaxed notion of ``logarithmic stability'' introduced in \cite{DDHK} requires
$$ \log(1/\mach) = \log(1/\delta)+O(\log^c(n)\log(\kappa))$$
for some constant $c$, where $\kappa$ is an appropriate condition number. In comparison, Theorem \ref{thm:bkwd} obtains the weaker relationship
$$ \log(1/\mach) = O(\log^4(1/\delta)\log(n) + \log^5(n)),$$
which is still polylogarithmic in $n$ in the regime $\delta=1/\poly(n)$.\end{remark}

\subsection{Related Work} \label{sec:related}

\noindent {\bf Minimum Eigenvalue Gap.}
The minimum eigenvalue gap of random matrices has been studied in the case of
Hermitian and unitary matrices, beginning with the work of Vinson \cite{vinson2011closest},
who proved an $\Omega(n^{-4/3})$ lower bound on this gap in the case of the Gaussian Unitary Ensemble (GUE) and the Circular Unitary Ensemble (CUE). Bourgade and Ben Arous
\cite{arous2013extreme} derived exact limiting formulas for the distributions
of all the gaps for the same ensembles. Nguyen, Tao, and Vu
\cite{nguyen2017random} obtained non-asymptotic inverse polynomial bounds for a large class of
non-integrable Hermitian models with i.i.d. entries (including Bernoulli matrices). 

In a different direction, Aizenman et al. proved an inverse-polynomial bound
\cite{aizenman2017matrix} in the case of an arbitrary Hermitian matrix plus a
GUE matrix or a Gaussian Orthogonal Ensemble (GOE) matrix, which may be viewed as a smoothed analysis of the minimum
gap.  Theorem \ref{thm:polygap} may be viewed as a non-Hermitian analogue of
the last result. 

In the non-Hermitian case, Ge \cite{ge2017eigenvalue} obtained an inverse polynomial bound for i.i.d. matrices with real entries satisfying some mild moment conditions, and  \cite{shi2012smallest}\footnote{At the time of writing, the work \cite{shi2012smallest} is still an unpublished arXiv preprint.} proved an inverse polynomial lower bound for the complex Ginibre ensemble. Theorem \ref{thm:polygap} may be seen as a generalization of these results to non-centered complex Gaussian matrices.\\

\noindent {\bf Smoothed Analysis and Free Probability.} 
The study of numerical algorithms on Gaussian random matrices (i.e., the case $A=0$ of smoothed analysis) dates back to \cite{von1947numerical,smale1985efficiency,demmel1988probability,edelman1988eigenvalues}.  The powerful idea of improving the conditioning of a numerical computation by adding a small amount of Gaussian noise was introduced by Spielman and Teng in \cite{spielman2004smoothed}, in the context of the simplex algorithm.
Sankar, Spielman, and Teng \cite{sankar2006smoothed} showed that adding real
Gaussian noise to any matrix yields a matrix with polynomially-bounded
condition number; \cite{banks2019gaussian} can be seen as an extension of
this result to the condition number of the eigenvector matrix, where the proof crucially requires that the Gaussian perturbation is complex rather than real.  The main difference between our results and most of the results on smoothed analysis (including \cite{abbcs}) is that our running time depends logarithmically rather than polynomially on the size of the perturbation.


The broad idea of regularizing the spectral instability of a nonnormal matrix by
adding a random matrix can be traced back to the work of \'Sniady
\cite{sniady2002random} and Haagerup and Larsen \cite{haagerup2000brown} in the context of Free Probability theory.\\

\noindent{\bf  Matrix Sign Function.}
The matrix sign function was introduced by Zolotarev in 1877. 
It became a popular topic in numerical analysis following
the work of Beavers and Denman \cite{beavers1973computational, beavers1974new, denman1976matrix}
and Roberts \cite{roberts1980linear}, who used it first to solve the algebraic
Ricatti and Lyapunov equations and then as an approach to the eigenproblem; see
\cite{kenney1995matrix} for a broad survey of its early history. The numerical
stability of Roberts' Newton iteration was investigated by Byers
\cite{byers1986numerical}, who identified some cases where it is and isn't
stable.  Malyshev \cite{malyshev1993parallel}, Byers, He, and Mehrmann \cite{byers1997matrix}, Bai, Demmel, and Gu \cite{bai1997inverse}, and Bai and Demmel
\cite{bai1998using} studied the condition number of the matrix sign function, and showed that {if} the
Newton iteration converges then it can be used to obtain a high-quality
invariant subspace\footnote{This is called an {\em a fortiriori bound} in
numerical analysis.}, but did not prove convergence in finite
arithmetic and left this as an open question.\footnote{\cite{byers1997matrix} states: ``A
priori backward and forward error bounds for evaluation of the matrix sign
function remain elusive.''} The key issue in analyzing the convergence of the
iteration is to bound the condition numbers of the intermediate matrices
that appear, as N. Higham remarks in his 2008 textbook:
\begin{quote}
Of course, to obtain a complete picture, we also need to understand the effect
of rounding errors on the iteration prior to convergence. This effect is surprisingly
difficult to analyze.\ \xelip Since errors will
in general occur on each iteration, the overall error will be a complicated function of
$\kappa_{\text{\emph{sign}}}(X_k)$ and $E_k$ for all $k$.\ \xelip 
We are not aware of any published rounding error analysis for the computation of $\text{\emph{sign}}(A)$ via the Newton
iteration.
--\cite[Section 5.7]{higham2008functions}
\end{quote}
This is precisely the problem solved by Theorem \ref{thm:signintro}, which is as far as we know the first provable algorithm for computing the sign function of an arbitrary matrix which does not require computing the Jordan form.

In the special case of Hermitian matrices, Higham \cite{higham1994matrix} established
efficient reductions between the sign function and the polar decomposition. Byers and Xu \cite{byers2008new}
proved backward stability of a certain scaled version of the Newton iteration for Hermitian matrices,
in the context of computing the polar decomposition. Higham and Nakatsukasa \cite{nakatsukasa2012backward} (see also the improvement \cite{nakatsukasa2016computing}) proved backward stability of a different iterative scheme for
computing the polar decomposition, and used it to give backward stable spectral bisection algorithms
for the Hermitian eigenproblem with $O(n^3)$-type complexity.\\

\noindent {\bf Non-Hermitian Eigenproblem.} 
{\em Floating Point Arithmetic.} The eigenproblem has been thoroughly studied in the numerical analysis community,
in the floating point model of computation. While there are provably fast and accurate algorithms in the Hermitian case (see the next subsection) and a large body of work for various structured matrices (see, e.g., \cite{bindel2005fast}), the general case is not nearly as well-understood. As recently as 1997, J. Demmel remarked in his well-known textbook \cite{demmel1997applied}: ``\xelip the problem of devising an algorithm [for the non-Hermitian eigenproblem] that is numerically stable and globally (and quickly!) convergent remains open." 

Demmel's question remained entirely open until 2015, when it was answered in the following sense by Armentano, Beltr\'an, B\"urgisser, Cucker, and Shub in the remarkable paper \cite{abbcs}. They exhibited an algorithm (see their Theorem 2.28) which given any $A\in
		\C^{n\times n}$ with $\|A\|\le 1$ and $\sigma>0$ produces in $O(n^{9}/\sigma^2)$
		expected arithmetic operations the diagonalization of 
		the nearby random perturbation
		$A+\sigma G$ where $G$ is a matrix with standard complex Gaussian
		entries.
		By setting $\sigma$ sufficiently small, this may be viewed as
		a {backward approximation} algorithm for diagonalization, in
		that it solves a nearby problem essentially
		exactly\footnote{The output of their algorithm is $n$ vectors
		on each of which Newton's method converges quadratically to an
		eigenvector, which they refer to as ``approximation \`a la Smale''.} -- in particular, by setting $\sigma=\delta/\sqrt{n}$ and noting that $\|G\|=O(\sqrt{n})$ with very high probability, their result implies a running time of $O(n^{10}/\delta^2)$ in our setting.
		Their algorithm is based on homotopy continuation methods, which they argue
		informally are numerically stable and can be implemented in finite precision arithmetic. Our algorithm is similar on a high level in that it adds a Gaussian perturbation to the input and then obtains a high accuracy forward approximate solution to the perturbed problem. The difference is that their overall running time depends polynomially rather than logarithmically on the accuracy $\delta$ desired with respect to the original unperturbed problem.
		
\begin{table}
\centering
\begin{threeparttable}[] 

\begin{tabular}{@{}lllll@{}} 
\toprule
 Result    & Error &  Arithmetic Ops & Boolean Ops & Restrictions  \\ \midrule

\cite{parlett1998symmetric} & Backward & $n^3+n^2\log(1/\delta)$ & $n^3\log(n/\delta)+n^2\log(1/\delta)\log(n/\delta)$ & Hermitian\\

\cite{abbcs}      &   Backward  & $n^{10}/\delta^2$    & $n^{10}/\delta^2\cdot \polylog(n/\delta)$\tnote{a} \\

\cite{ben2018quasi}  & Backward & $n^{\omega+1}\polylog(n)\log(1/\delta)$ & $n^{\omega+1}\polylog(n)\log(1/\delta)$ & Hermitian\\

Theorem \ref{thm:bkwd} \tnote{b}  & Backward & $T_\MM(n)\log^2(n/\delta)$ & $T_\MM(n)\log^6(n/\delta)\log(n)$ &\\

Corollary \ref{thm:fwd}  & Forward & $T_\MM(n)\log^2(n\kappaeig/\delta)$ & $T_\MM(n)\log^6(n\kappaeig/\delta)\log(n)$ & \\

\bottomrule
\end{tabular}

\begin{tablenotes}\footnotesize
\item [a] Does not specify a particular bound on precision.
\item [b] $T_\MM(n)=O(n^{\omega+\eta})$ for every $\eta>0$, see Definition \ref{def:MM} for details.
\end{tablenotes}

\caption{Results for finite-precision floating-point arithmetic} \label{tab:runtimes}
\end{threeparttable}
\end{table}
{\em Other Models of Computation.}
If we relax the requirements further and ask for {any} provable algorithm in {any} model of Boolean computation, 
there is only one more positive result with a polynomial bound on the number of bit operations: Jin Yi Cai showed in 1994 \cite{cai1994computing} that given a rational
		$n\times n$ matrix $A$ with integer entries of bit length $a$, one
		can find an $\delta$-forward approximation to its Jordan Normal Form $A=VJV^{-1}$
		in time $\poly(n,a,\log(1/\delta))$, where the degree of the
		polynomial is at least 12.
		This algorithm works in the {rational arithmetic} model of computation, so it does not quite answer
		Demmel's question since it is not a numerically stable algorithm. However, it enjoys the significant advantage of being able to compute forward approximations to discontinuous quantities such as the Jordan structure. 

\begin{table}[] 
\centering
\begin{threeparttable}

\begin{tabular}{@{}llllll@{}}
\toprule
Result  &  Model   & Error &  Arithmetic Ops & Boolean Ops & Restrictions\\ \midrule

\cite{cai1994computing} & Rational & Forward\tnote{a} & $\poly(a,n,\log(1/\delta))\tnote{b}$ & $\poly(a,n,\log(1/\delta))$ & \\

\cite{pan1999complexity} & Rational & Forward & $n^{\omega}+n\log\log(1/\delta)$ & $n^{\omega+1}a+n^2\log(1/\delta)\log\log(1/\delta)$ & Eigenvalues only\tnote{c}\\

\cite{louis2016accelerated} & Finite\tnote{c} & Forward & $n^\omega\log(n)\log(1/\delta)$ & $n^\omega\log^4(n)\log^2(n/\delta)$ & \mbox{Hermitian, $\lambda_1$ only}\\

\bottomrule
\end{tabular}
\begin{tablenotes}\footnotesize
\item [a] Actually computes the Jordan Normal Form. The degree of the polynomial is not specified, but is at least $12$ in $n$.
\item [b] In the bit operations, $a$ denotes the bit length of the input entries.
\item [c] Uses a custom bit representation of intermediate quantities.
\end{tablenotes}
\caption{Results for other models of arithmetic} \label{tab:runtimes2}
\end{threeparttable}
\end{table}

As far as we are aware, there are no other published provably  polynomial-time algorithms for the general eigenproblem.
The two standard references for diagonalization appearing most often in theoretical computer science papers do not meet this criterion.
	In particular, the widely cited work by Pan and Chen
	\cite{pan1999complexity} proves that one can compute the {\em
	eigenvalues} of $A$ in $O(n^\omega + n\log\log(1/\delta))$ (suppressing logarithmic factors) {\em arithmetic} operations by
	finding the roots of its characteristic polynomial, which becomes a
	bound of $O(n^{\omega+1}a+n^2\log(1/\delta)\log\log(1/\delta))$ bit operations if the characteristic
	polynomial is computed exactly in rational arithmetic and the matrix has entries of bit length $a$. However that paper
	does {not} give any bound for the amount of time taken to find
	approximate eigenvectors from approximate eigenvalues, and states this as an open problem.\footnote{	``The remaining nontrivial problems are, of course, the estimation of the above output precision $p$ [sufficient for finding an approximate eigenvector
	from an approximate eigenvalue], \xelip .  We leave these open problems as a challenge for the reader.'' -- \cite[Section 12]{pan1999complexity}.}

	Finally, the important work of Demmel, Dumitriu, and Holtz \cite{demmel2007fast} (see also the followup \cite{ballard2010minimizing}), which we rely on heavily,
	does not claim to provably solve the eigenproblem either---it bounds the running
	time of {one iteration} of a specific algorithm, and shows that
	such an iteration can be implemented numerically stably, without
	proving any bound on the number of iterations required in general.\\

\noindent {\bf Hermitian Eigenproblem}. For comparison, the eigenproblem for
Hermitian matrices is much better understood.  We cannot give a complete
bibliography of this huge area, but mention one relevant landmark result: the work of
Wilkinson \cite{wilkinson1968global}, who exhibited a globally convergent diagonalization algorithm, and the work of  Dekker and Traub \cite{dekker1971shifted} who quantified the rate of convergence of Wilkinson's algorithm and from which it follows that the Hermitian eigenproblem can be solved with
	backward error $\delta$ in $O(n^3+n^2\log(1/\delta))$ arithmetic operations in exact arithmetic.\footnote{We are not aware a published analysis of this algorithm in finite arithmetic, but believe that it can be carried out with $O(\log(n/\delta))$ bits of precision. The only issue that needs to be handled is forward instability of the QR step when the Wilkinson shift is very close to an eigenvalue of the matrix, which can be resolved e.g. by a small random perturbation of the Wilkinson shift.}. We refer the reader to \cite[\S 8.10]{parlett1998symmetric} for the simplest and most insightful proof of this result, due to Hoffman and Parlett \cite{hoffmann1978new} 

There has also recently been renewed interest in this problem in the theoretical computer science community, with
the goal of bringing the runtime close to $O(n^\omega)$: Louis and Vempala \cite{louis2016accelerated}
show how to find a $\delta-$approximation of just the largest eigenvalue in $O(n^\omega\log^4(n)\log^2(1/\delta))$
bit operations, and Ben-Or and Eldar \cite{ben2018quasi} give an $O(n^{\omega+1}\polylog(n))$-bit-operation algorithm for finding
a $1/\poly(n)$-approximate diagonalization of an $n\times n$ Hermitian matrix normalized to have $\|A\|\le 1$.

\begin{remark}[Davies' Conjecture]\label{rem:davies} The beautiful paper \cite{davies2007approximate} introduced the idea of approximating a matrix function $f(A)$ for nonnormal $A$ by $f(A+E)$ for some well-chosen $E$ regularizing the eigenvectors of $A$. This directly inspired our approach to solving the eigenproblem via regularization. 

The {existence} of an approximate diagonalization \eqref{eqn:approxdiag} for every $A$
with a {\em well-conditioned similarity} $V$ (i.e, $\kappa(V)$ depending
polynomially on $\delta$ and $n$) was precisely the content of Davies'
conjecture \cite{davies2007approximate}, which was recently solved by some of the authors and Mukherjee in \cite{banks2019gaussian}. The existence
	of such a $V$ is a pre-requisite for proving that one can always efficiently find an approximate diagonalization in
finite arithmetic, since if $\|V\|\|V^{-1}\|$ is very large it may require many bits of precision to represent. Thus, Theorem \ref{thm:bkwd} can be viewed as an efficient algorithmic answer to Davies' question.
\end{remark}

\begin{remark}[Subsequent work in Random Matrix Theory]\label{rem:realdavies}
Since the first version of the present paper was made public there have been some advances  in random matrix theory \cite{banks2020overlaps, jain2020real} that prove analogues of  Theorem \ref{thm:smoothed}  in the case where $G_n$ is replaced by a  perturbation with random real independent entries.  These results formally articulate that, in the context of this paper,  there is nothing special about complex Ginibre matrices, and that the same  regularization effect can be achieved using a broader class of perturbations. Bounding the eigenvector condition number and the eigenvalue gap when the random perturbation has  real  entries poses interesting technical challenges that were tackled in different ways in the aformentioned papers. We also refer the reader to \cite{cipolloni2021condition} where optimal results were obtained in the case where $A=0$ and $G_n$ has real Gaussian entries. 
\end{remark}

\begin{remark}[Alternate Proofs using \cite{abbcs}]
\label{rem:comparisontoabbcs}
In October 2021 (about two years after the first appearance of this paper), we noticed that a version of Theorem \ref{thm:smoothed} (with a worse $\kappa_V$ bound but a better eigenvalue gap bound) as well as the main theorem of \cite{banks2019gaussian} (with a slightly worse dependence on $n$) can be easily derived from some auxiliary results shown in \cite{abbcs} (specifically  Proposition 2.7 and Theorem 2.14 of that paper), which we were not previously aware of. We present these short alternate proofs in Appendix \ref{sec:proofsforABBCS}. We remark that our original proofs are essentially different from those appearing in \cite{abbcs} --- in particular, they rely on studying the area of pseudospectra, whereas the proof of Theorem 2.14 of \cite{abbcs} relies on geometric concepts and the coarea formula for Gaussian integrals of certain determinantal quantities on Riemannian manifolds. The proofs based on pseudospectra are arguably more flexible; as mentioned in Remark \ref{rem:realdavies}, they have been recently generalized to ensembles besides the complex Ginibre ensemble, which seems difficult to do for the more algebraic proofs of \cite{abbcs}.
\end{remark}

\noindent {\bf Reader Guide.} This paper contains a lot of parameters and constants. On first reading, it may be good to largely ignore the constants not appearing in exponents, and to keep in mind the typical setting $\delta=1/\poly(n)$ for the accuracy, in which case the important auxiliary parameters $\omega, 1-\alpha, \epsilon, \beta, \eta$ are all $1/\poly(n)$, and the machine precision is $\log(1/\mach)=\polylog(n)$.

	\section{Preliminaries}


Let $M \in \C^{n\times n}$ be a complex matrix, not necessarily normal. We will write matrices and vectors with uppercase and lowercase letters, respectively. Let us denote by $\Lambda(M)$ \marginnote{$\Lambda(M),\lambda_i(M)$} the spectrum of $M$ and by $\lambda_i(M)$ its individual eigenvalues. In the same way we denote the singular values of $M$ by $\sigma_i(M)$ \marginnote{$\sigma_i(M)$} and we adopt the convention $\sigma_1(M) \ge \sigma_2(M) \ge \cdots \ge \sigma_n(M)$. When $M$ is clear from the context we will simplify notation and just write $\Lambda, \lambda_i$ or $\sigma_i$ respectively. 

Recall that the \emph{operator norm} of $M$ is
$$
    \|M\| = \sigma_1(M) = \sup_{\|x \| = 1}\|Mx\|.
$$
As usual, we will say that $M$ is \textit{diagonalizable} if it can be written as $M = VDV^{-1}$ for some diagonal matrix $D$ whose nonzero entries contain the eigenvalues of $M$. In this case we have the spectral expansion
\begin{equation} \label{eq:spectral-expansion}
    M = \sum_{i=1}^n \lambda_i v_i w_i^\ast,
\end{equation}
where the right and left eigenvectors $v_i$ and $w_j^\ast$ are the columns and rows of $V$ and $V^{-1}$ respectively, normalized so that $w^\ast_i v_i = 1$. 

\subsection{Spectral Projectors and Holomorphic Functional Calculus} \label{sec:holofuncalc}
Let $M \in \C^{n\times n}$, with eigenvalues $\lambda_1,...,\lambda_n$.
We say that a matrix $P$ is a \emph{spectral projector} \marginnote{\emph{spectral\\ projector}} for $M$ if $MP = PM$
and  $P^2 = P$. For instance, each of the terms $v_i w_i^\ast$ appearing in the
spectral expansion \eqref{eq:spectral-expansion} is a spectral projector, as
$Av_iw_i^\ast = \lambda_i v_i w_i^\ast = v_i w_i^\ast A$ and $w_i^\ast v_i =
1$. If $\Gamma_i$ is a simple closed positively oriented rectifiable curve in the complex plane
separating $\lambda_i$ from the rest of the spectrum, then it is well-known
that
$$  
    v_i w_i^\ast = \frac{1}{2\pi i}\oint_{\Gamma_i} (z - M)^{-1}\dee z,
$$
by taking the Jordan normal form of the the \emph{resolvent} \marginnote{\emph{resolvent}} $(z - M)^{-1}$ and applying Cauchy's integral formula.\footnote{In this manuscript we will use $z-M$ as a shorthand notation for $zI-M$ where $I$ denotes the identity matrix.}

Since every spectral projector $P$ commutes with $M$, its range agrees exactly
with an invariant subspace of $M$. We will often find it useful to choose some
region of the complex plane bounded by a simple closed positively oriented rectifiable curve $\Gamma$,
and compute the spectral projector onto the invariant subspace spanned by those
eigenvectors whose eigenvalues lie inside $\Gamma$. Such a projector can be
computed by a contour integral analogous to the above.

Recall that if $f$ is any function, and $M$ is diagonalizable, then we can
meaningfully define $f(M) := V f(D) V^{-1}$, where $f(D)$ is simply the result of applying $f$ to
each element of the diagonal matrix $D$. The \emph{holomorphic functional calculus} gives an
equivalent definition that extends to the case when $M$ is non-diagonalizable.
As we will see, it has the added benefit that bounds on the norm of the
resolvent of $M$ can be converted into bounds on the norm of $f(M)$.

\begin{proposition}[Holomorphic Functional Calculus]
    Let $M$ be any matrix, $B \supset \Lambda(M)$ be an open neighborhood of its spectrum (not necessarily connected), and $\Gamma_1,...,\Gamma_k$ be simple closed positively oriented rectifiable curves in $B$ whose interiors together contain all of $\Lambda(M)$. Then if $f$ is holomorphic on $B$, the definition
    $$
        f(M) := \frac{1}{2\pi i}\sum_{j=1}^k \oint_{\Gamma_j} f(z)(z - M)^{-1}\dee z
        \marginnote{holomorphic\\ functional\\ calculus}
    $$
    is an \emph{algebra homomorphism} in the sense that $(fg)(M) = f(M)g(M)$ for any $f$ and $g$ holomorphic on $B$.
\end{proposition}

Finally, we will frequently use the 
\emph{resolvent identity} 
$$(z - M)^{-1} - (z - M')^{-1} = (z-M)^{-1}(M - M')(z
- M')^{-1}$$
to analyze perturbations of contour integrals.

\subsection{Pseudospectrum and Spectral Stability} \label{sec:prelimspectral}



The $\epsilon-$pseudospectrum of a matrix is defined in 
\eqref{eqn:pseudodef1}. Directly from this definition, we
can relate the pseudospectra of a matrix and a perturbation of it.
\begin{proposition}[\cite{trefethen2005spectra}, Theorem 52.4] 
\label{prop:decrementeps} 
    For any $n \times n$ matrices $M$ and $E$ and any $\eps > 0$, $\Lambda_{\eps - \Vert E\Vert}(M) \subseteq \Lambda_\eps(M+E)$.
\end{proposition}
\noindent It is also immediate that $\Lambda(M) \subset \Lambda_\epsilon(M)$, and in fact a stronger relationship holds as well:
\begin{proposition}[\cite{trefethen2005spectra}, Theorem 4.3]
\label{thm:componentsofpseudoespec}
    For any $n \times n$ matrix $M$, any bounded connected component of $\Lambda_\eps(M)$ must contain an eigenvalue of $M$.
\end{proposition}

Several other notions of stability will be useful to us as well. If $M$ has distinct eigenvalues $\lambda_1,\ldots,\lambda_n$, and spectral expansion as in \eqref{eq:spectral-expansion}, we define the \emph{eigenvalue condition number of $\lambda_i$} to be
	$$
	    \kappa(\lambda_i) := \left\| {v_i w_i^\ast}\right\| =\|v_i\|\|w_i\|.
	    \marginnote{$\kappa(\lambda_i)$}
	$$
By considering the scaling of $V$ in \eqref{eqn:kappavdef} in which its columns $v_i$ have unit length, so that $\kappa(\lambda_i) = \Vert w_i \Vert$, we obtain the useful relationship
\begin{equation}\label{eqn:kappav} 
    \kappa_V(M)\le \Vert V \Vert \Vert V^{-1} \Vert \le  \|V\|_F\|V^{-1}\|_F\le \sqrt{n\cdot\sum_{i\le n} \kappa(\lambda_i)^2}.
\end{equation}
Note also that the eigenvector condition number and pseudospectrum are related as follows: 
\begin{lemma}[\cite{trefethen2005spectra}]
\label{lem:pseudospectralbauerfike}
    Let $D(z,r)$ denote the open disk of radius $r$ centered at $z \in \C$.  For every $M \in \C^{n\times n}$,
    \begin{equation} \label{eqn:lambdakappa}
        \bigcup_i D(\lambda_i,\eps)\subset \Lambda_\eps(M)\subset \bigcup_i D(\lambda_i, \eps \kappa_V(M)).
    \end{equation}
\end{lemma}

In this paper we will repeatedly use that assumptions about the pseudospectrum of a matrix can be turned into stability statements about functions applied to  the matrix via the holomorphic functional calculus. Here we describe an instance of particular importance. 

Let $\lambda_i$ be a simple eigenvalue of $M$ and let $\Gamma_i$ be a contour in the complex plane, as in Section \ref{sec:holofuncalc}, separating $\lambda_i$ from the rest of the spectrum of $M$, and assume $\Lambda_\epsilon(M)\cap \Gamma =\emptyset$. Then, for any  $\|M-M'\|< \eta <\epsilon$, a combination of Proposition \ref{prop:decrementeps} and Proposition \ref{thm:componentsofpseudoespec} implies that there is a unique eigenvalue $\lambda_i'$ of $M'$ in the region enclosed by $\Gamma$, and furthermore $\Lambda_{\epsilon-\eta}(M')\cap \Gamma = \emptyset$. If $v_i'$ and $w_i'$ are the right and left eigenvectors of $M'$ corresponding to $\lambda_i'$ we have 
\begin{align}
        \| v_i'w_i'^\ast - v_iw_i^\ast \| &= \frac{1}{2\pi} \left\|\oint_\Gamma (z - M)^{-1} - (z-M')^{-1} \dee z \right\| \nonumber \\
        &= \frac{1}{2\pi}\left\|\oint_\Gamma (z - M)^{-1}(M - M')(z-M')^{-1} \dee z \right\| \nonumber \\ \label{eq:projectorstability}
        &\le \frac{\ell(\Gamma)}{2\pi}\frac{\eta}{\epsilon(\epsilon - \eta)}.
    \end{align}
We have introduced enough tools to prove Proposition \ref{prop:eigperturb}.

\begin{proof}[Proof of Proposition \ref{prop:eigperturb}]
    For $t\in [0, 1]$ define  $A(t) = (1-t)A+ tA' $. Since $\delta <\frac{\gap(A)}{8\kappa_V(A)}$ the Bauer-Fike theorem implies  that $A(t)$ has distinct eigenvalues for all $t$, and in fact $\gap(A(t))\geq \frac{3\gap(A)}{4}$. Standard results in perturbation theory (for instance \cite[Theorem 1]{greenbaum2020first} or any of the references therein) imply that for every $i=1, \dots, n$, $A(t)$ has a unique eigenvalue $\lambda_i(t)$ such that $\lambda_i(t)$ is a differentiable trajectory, $\lambda_i(0) =\lambda_i$ and $\lambda_i(1)=\lambda_i'$. Let $P_i(t)$ be the associated spectral projector of $\lambda_i(t)$, which is uniquely defined via a contour integral, and write $P_i = P_i(0)$.
    
    Let $\Gamma_i$ be the positively oriented contour forming the boundary of the closed disk centered at $\lambda_i$ with radius $\gap(A)/2$, and define $\epsilon =\frac{\gap(A)}{2\kappa_V(A)}$. Lemma \ref{lem:pseudospectralbauerfike} implies $\Lambda_{\epsilon}(A)$ is contained in the union of these disks over all $i \in [n]$, and for fixed $t\in [0, 1]$, since $\|A-A(t)\|< t \delta \le \epsilon/4$, Proposition \ref{prop:decrementeps} gives the same containment for $\Lambda_{3\epsilon/4}(A(t))$. Since these disks intersect only in their boundaries (if they do at all), $\|(z - A)^{-1}\| \le 1/\epsilon$ and $\|(z - A(t))^{-1}\| \le 4/3\epsilon$ for $z \in \Gamma_i$. By the derivation of \eqref{eq:projectorstability} above,
    $$
        |\kappa(\lambda_i)-\kappa(\lambda_i(t))|\leq \|P_i(t) - P_i\| \leq \frac{\ell(\Gamma_i)}{2\pi} \cdot \frac{1}{\epsilon}\cdot \frac{4}{3\epsilon}\cdot \frac{\epsilon}{4} = \frac{\gap(A)}{2}\frac{2\kappa_V(A)}{3\gap(A)} = \frac{\kappa_V(A)}{3}
    $$
    and hence $\kappa(\lambda_i(t))\leq \kappa(\lambda_i)+\kappa_V(A)/3 \leq 4\kappa_V(A)/3$. Combining this with (\ref{eqn:kappav}) we obtain 
    $$\kappa_V(A(t)) \leq 2 \sqrt{n \cdot \sum_i \kappa(\lambda_i)^2}< 4n \kappa_V(A)/3. $$
    
    From Theorem 2 of \cite{greenbaum2020first} and the subsequent discussion on p. 468, there exist analytic functions $v_i(t)$ satisfying $v_i(0) = v_i$ and $A(t)v_i(t) = \lambda_i(t)v_i(t)$ for all $i \in [n]$ and $t \in [0,1]$, which furthermore admit the bound
    $$
        \|\dot{v_i}(t)\| \le \frac{\kappa_V(A(t))}{\gap(A(t))}\|\dot A(t)\|\|v_i(t)\| \le \frac{\delta \kappa_V(A(t))}{\gap(A(t)}\|v_i(t)\|.
    $$
    However, these $v_i(t)$ need not in general be unit vectors (see \cite[Section 3.4]{greenbaum2020first} and references for discussion of various normalizations). Therefore set $\hat v_i(t) = \|v_i(t)\|^{-1} v_i(t)$, and note that by an application of the chain rule, 
    $$
        \|\dot{\hat{v_i}}(t)\| \le \frac{\delta \kappa_V(A(t))}{\gap(A(t))}.
    $$
    It then follows that the vectors $v_i' = \hat{v_i}(1)$ for $i \in [n]$ satisfy the conclusion of the theorem, by bounding $\kappa_V(A(t))\leq 4n\kappa_V(A)/3$ and $\gap(A(t))\geq \frac{3\gap(A)}{4}$, and integrating the resulting upper bound $\|\dot{\hat{v_i}}(t)\| \le \frac{16n\delta \kappa_V(A)}{9\gap(A)}$ from $t = 0$ to $t= 1$.
    
\end{proof}


\subsection{Finite-Precision Arithmetic}
We briefly elaborate on the axioms for floating-point arithmetic given in Section \ref{sec:models}. Similar guarantees to the ones appearing in that section for scalar-scalar operations also hold for operations such as matrix-matrix addition and matrix-scalar multiplication. In particular, if $A$ is an $n\times n$ complex matrix,
$$
    \fl(A) = A + A \circ \Delta \qquad |\Delta_{i,j}| < \mach.
$$
It will be convenient for us to write such errors in additive, as opposed to multiplicative form. We can convert the above to additive error as follows. Recall that for any $n\times n$ matrix, the spectral norm (the $\ell^2 \to \ell^2$ operator norm) is at most $\sqrt n$ times the $\ell^2 \to \ell^1$ operator norm, i.e. the maximal norm of a column. Thus we have
\begin{equation} \label{eqn:maxnorm}
    \|A \circ\Delta\| \le \sqrt{n} \max_i \|(A\circ \Delta) e_i\| \le \sqrt{n}\max_{i,j}|\Delta_{i,j}| \max_i \|A e_i\| \le \mach \sqrt{n}\|A\|.
\end{equation}
For more complicated
operations such as matrix-matrix multiplication and matrix inversion, we use
existing error guarantees from the literature. This is the subject of Section
\ref{sec:logstable}. 

We will also need to compute the trace of a matrix $A \in \C^{n\times n}$, and normalize a vector $x \in \C^n$. Error analysis of these is standard (see for instance the discussion in \cite[Chapters 3-4]{higham2002accuracy}) and the results in this paper are highly insensitive to the details. For simplicity, calling $\hat x := x/\|x\|$, we will assume that
\begin{align}
    |\fl\left(\Tr A\right) - \Tr A| &\le n\|A\| \mach \label{eqn:fltrace}\\
    \| \fl(\hat x) - \hat x\| &\le n\mach. \label{eqn:flnorm}
\end{align}
Each of these can be achieved by assuming that $\mach n \le \epsilon$ for some suitably chosen $\epsilon$, independent of $n$, a requirement which will be depreciated shortly by several tighter assumptions on the machine precision.

Throughout the paper, we will take the pedagogical perspective that our
algorithms are games played between the practitioner and an adversary who may
additively corrupt each operation. In particular, we will include explicit
error terms (always denoted by $E_{(\cdot)}$) in each appropriate step of every
algorithm. In many cases we will first analyze a routine in exact
arithmetic---in which case the error terms will all be set to zero---and
subsequently determine the machine precision $\mach$ necessary so that the errors are 
small enough to guarantee convergence.

\subsection{Sampling Gaussians in Finite Precision}
\label{sec:gaussians}
For various parts of the algorithm, we will need to sample from normal distributions.  
For our model of arithmetic, we assume that the complex normal distribution can be sampled up to machine precision in $O(1)$ arithmetic operations.  To be precise, we assume the existence of the following sampler:  

\begin{definition}[Complex Gaussian Sampling]\label{def:gaussian}

A $\cn$-stable Gaussian sampler $\N(\sigma)$ takes as input $\sigma \in \mathbb{R}_{\ge 0}$ and outputs a sample of a random variable $\widetilde{G} = \N(\sigma)$ with the property that there exists $G \sim N_{\mathbb{C}}(0, \sigma^2)$ satisfying 
\[ 
    |\widetilde{G} - G|  \le \cn \sigma \cdot \mach 
    \marginnote{$\cn$}
\]
with probability one, in at most $T_\N$ arithmetic operations for some universal constant $T_\N>0$. \marginnote{$T_\N$}
\end{definition}

Note that, since the Gaussian distribution has unbounded support, one should only expect the sampler $\N(\sigma)$ to have a relative error guarantee of the sort $|\widetilde{G} - G|  \le \cn \sigma |G| \cdot \mach$. However, as it will become clear below, we only care about realizations of Gaussians satisfying $|G|<R$, for a certain prespecified $R>0$, and the rare event $|G|>R$ will be accounted for in the failure probability of the algorithm. So, for the sake of exposition we decided to omit the $|G|$ in the bound on $|\widetilde{G}-G|$.  

We will only sample $O(n^2)$ Gaussians during the algorithm, so this sampling will not contribute significantly to the runtime.  Here as everywhere in the paper, we will omit issues of underflow or overflow.  Throughout this paper, to simplify some of our bounds, we will also assume that $\cn \geq 1$.

\subsection{Black-box Error Assumptions for  Multiplication, Inversion, and QR} \label{sec:logstable}
Our algorithm uses matrix-matrix multiplication, matrix inversion, and QR
factorization as primitives. For our analysis, we must therefore assume some
bounds on the error and runtime costs incurred by these subroutines.  In this
section, we first formally state the kind of error and runtime bounds we require, and
then discuss some implementations known in the literature that satisfy each of
our requirements with modest constants. 

Our definitions are inspired by the
definition of {\em logarithmic stability} introduced in \cite{demmel2007fast}.  Roughly speaking, they say that implementing the algorithm with floating point
precision $\u$ yields an accuracy which is at most polynomially or quasipolynomially in $n$
worse than $\u$ (possibly also depending on the condition number in the case of
inversion).  Their definition has the property that while a logarithmically stable algorithm is not strictly-speaking backward stable, it can attain the same forward error bound as a backward stable algorithm at the cost of increasing the bit length by a polylogarithmic factor.  See Section 3 of their paper for a precise definition and a more detailed discussion of how their definition relates to standard numerical stability notions.

\begin{definition}
\label{def:MM}
	A {\em $\pmult(n)$-stable multiplication algorithm} \marginnote{$\MM$} $\MM(\cdot, \cdot)$ takes as input $A,B\in \C^{n\times n}$ and a precision $\u>0$ and outputs $C=\MM(A, B)$ satisfying
    $$ 
        \|C-AB\|\le \pmult(n) \cdot \u \|A\|\|B\|,
        \marginnote{$\pmult $}
    $$
    on a floating point machine with precision $\u$, in  $T_\MM(n)$ arithmetic operations. \marginnote{$T_{\MM}(n)$}
\end{definition}

\begin{definition} \label{def:inv} A {\em $(\pinv(n), \cinv)-$stable inversion algorithm} $\INV(\cdot)$ \marginnote{$\INV$} takes as input $A\in\C^{n\times n}$ and a precision $\u$ and outputs $C=\INV(A)$ satisfying
$$ 
    \|C-A^{-1}\|\le \pinv(n)\cdot \u\cdot \kappa(A)^{\cinv \log n}\|A^{-1}\|,
    \marginnote{$\pinv,\cinv$}
$$
on a floating point machine with precision $\u$, in $T_\INV(n)$ arithmetic operations. \marginnote{$T_{\INV}$}
\end{definition}

\begin{definition}
\label{def:qr}
	A {\em $\pqr(n)$-stable QR factorization algorithm} $\QR(\cdot)$ \marginnote{$\QR$} takes as input $A\in \C^{n\times n}$ and a precision $\u$, and outputs $[Q,R]=\QR(A)$ such that \begin{enumerate}
	\item $R$ is exactly upper triangular.
	\item There is a unitary $Q'$ and a matrix $A'$ such that \begin{equation}\label{eqn:qr1} Q' A'= R, \end{equation}
		and
    $$
        \|Q' - Q\|\le \pqr(n)\u, \quad \textrm{and} \quad \| A'-A\|\le \pqr(n)\u\|A\|,
        \marginnote{$\pqr$}
    $$
	\end{enumerate}
on a floating point machine with precision $\u$. Its running time is $T_\QR(n)$ arithmetic operations. \marginnote{$T_{\QR}$}
\end{definition}
\begin{remark} 
    Throughout this paper, to simplify some of our bounds, we will assume that $$1 \leq \mu_\MM(n), \mu_{\INV}(n) , \mu_\QR (n), c_{\INV}\log n.$$
\end{remark}

The above definitions can be instantiated with traditional $O(n^3)$-complexity algorithms for which $\pmult, \pqr, \pinv$ are all $O(n)$ and $\cinv=1$
\cite{higham2002accuracy}. This yields easily-implementable practical algorithms with 
running times depending cubically on $n$.

In order to achieve $O(n^\omega)$-type efficiency, we instantiate them with
fast-matrix-multiplication-based algorithms and with $\mu(n)$ taken to be a low-degree
polynomial \cite{demmel2007fast}. Specifically, the following parameters are known to be achievable.

\begin{theorem}[Fast and Stable Instantiations of $\MM,\INV, \QR$] \label{thm:instantiate}
~\\
\begin{enumerate}
	\item If $\omega$ is the exponent of matrix multiplication, then for every $\eta>0$ there is a $\pmult(n)-$stable multiplication algorithm with $\pmult(n)=n^{c_\eta}$ and $T_\MM(n)=O(n^{\omega+\eta})$, where $c_\eta$ does not depend on $n$.
	\item Given an algorithm for matrix multiplication satisfying (1), there is a ($\pinv(n),\cinv)$-stable inversion algorithm with
		$$\pinv(n)\le O(\pmult(n)n^{\lg(10)}),\quad \quad \cinv\le 8,$$
		and $T_\INV(n)\le T_\MM(3n)=O(T_\MM(n))$.
	\item Given an algorithm for matrix multiplication satisfying (1), there is a $\pqr(n)-$stable QR factorization algorithm with
		$$\pqr(n)=O(n^{c_\QR} \pmult(n)),$$
		where $c_\QR$ is an absolute constant, and $T_\QR(n)=O(T_\MM(n))$.
\end{enumerate}
In particular, all of the running times above are bounded by $T_\MM(n)$ for an $n\times n$ matrix.
\end{theorem}
	\begin{proof}
	(1) is Theorem 3.3 of \cite{DDHK}. (2) is Theorem 3.3 (see also equation (9) above its statement) of \cite{demmel2007fast}. The final claim follows by noting that $T_\MM(3n)=O(T_\MM(n))$ by dividing a $3n\times 3n$ matrix into nine $n\times n$ blocks and proceeding blockwise, at the cost of a factor of $9$ in $\pinv(n)$. (3) appears in Section 4.1 of \cite{demmel2007fast}.

	\end{proof}
We remark that for specific existing fast matrix multiplication algorithms such as Strassen's
algorithm, specific small values of $\mu_\MM(n)$ are known (see \cite{DDHK} and its references for details), so these may also be used as a black box, though we will not do this in this paper.

	\section{Pseudospectral Shattering} \label{sec:gap}

This section is devoted to our central probabilistic result, Theorem \ref{thm:smoothed}, and the accompanying notion of \textit{pseudospectral shattering} which will be used extensively in our analysis of the spectral bisection algorithm in Section \ref{sec:spectralbisec}.

\subsection{Smoothed Analysis of Gap and Eigenvector Condition Number}\label{sec:smoothed}

As is customary in the literature, we will refer to an $n\times n$ random matrix $G_n$ \marginpar{$G_n$} whose entries are independent complex Gaussians drawn from $\mathcal{N}(0,1_\C/n)$ as a \textit{normalized complex Ginibre random matrix}. To be absolutely clear, and because other choices of scaling are quite common, we mean that $\dE G_{i,j} = 0$ and $\dE |G_{i,j}|^2 = 1/n$. 

In the course of proving Theorem \ref{thm:smoothed}, we will need to bound the probability that the second-smallest singular value of an arbitrary matrix with small Ginibre perturbation is atypically small. We begin with a well-known lower tail bound on the singular values of a Ginibre matrix alone.

\begin{theorem}[{\cite[Theorem 1.2]{szarek1991condition}}]\label{thm:szarek}
For $G_n$ an $n\times n$ normalized complex Ginibre matrix and for any $\alpha \geq 0$ it holds that
$$
    \P\left[\sigma_j(G_n) <  \frac{\alpha(n-j+1)}{n} \right] \leq \left(\sqrt{2e} \, \alpha\right)^{2(n-j+1)^2}.
$$
\end{theorem}

\noindent As in several of the authors' earlier work \cite{banks2019gaussian}, we can transfer this result to case of a Ginibre perturbation via a remarkable coupling result of P. \'{S}niady.

\begin{theorem}[\'Sniady \cite{sniady2002random}]\label{thm:sniady}
	Let $A_1$ and $A_2$ be $n \times n$ complex matrices such that $\sigma_i(A_1) \le \sigma_i(A_2)$ for all $1 \le i \le n$.  Assume further that $\sigma_i(A_1) \ne \sigma_j(A_1)$ and $\sigma_i(A_2) \ne \sigma_j(A_2)$ for all $i \ne j$.  Then for every $t \ge 0$, there exists a joint distribution on pairs of $n \times n$ complex matrices $(G_1, G_2)$ such that 
	\begin{enumerate}
    	\item the marginals $G_1$ and $G_2$ are distributed as normalized complex Ginibre matrices, and
    	\item almost surely $\sigma_i(A_1 + \sqrt{t} G_1) \le \sigma_i(A_2 + \sqrt{t} G_2)$ for every $i$.
 	\end{enumerate} 
\end{theorem}

\begin{corollary} \label{corr:sigma2szarek} For any fixed matrix $M$ and parameters $\gamma, t>0$
$$ \P[\sigma_{n-1}(M+\gamma G_n)<t]\le  (e/2)^4 (tn/\gamma)^8 \le 4(tn/\gamma)^8.$$
\end{corollary}
\begin{proof} We would like to apply Theorem \ref{thm:sniady} to $A_1=0$ and $A_2=M$, but the theorem has the technical condition that $A_1$ and $A_2$ have distinct singular values.  Taking vanishingly small perturbations of $0$ and $M$ satisfying this condition and taking the size of the perturbation to zero, we obtain
$$ \P[\sigma_{n-1}(M+\gamma G_n)<t] \le \P[\sigma_{n-1}(\gamma G_n)<t] = \P[\sigma_{n-1}(G_n)<t/\gamma].$$
Invoking Theorem \ref{thm:szarek} with $j=n-1$ and $\alpha$ replaced by $tn/2\gamma$ yields the claim.\end{proof}

We will need as well the main theorem of \cite{banks2019gaussian}, which shows that the addition of a small complex Ginibre to an arbitrary matrix tames its eigenvalue condition numbers.

\begin{theorem}[{\cite[Theorem 1.5]{banks2019gaussian}}] \label{thm:davies} 
	Suppose $A\in\C^{n\times n}$ with $\|A\|\le 1$ and $\delta\in (0,1)$. Let $G_n$ be a complex Ginibre matrix, and let $\lambda_1,\ldots,\lambda_n\in \C$ be the (random) eigenvalues of $A+\delta G_n$. Then for every measurable open set $B\subset \C,$
	$$ 
		\dE \sum_{\lambda_i \in B} \kappa(\lambda_i)^2 \le \frac{n^2}{\pi \delta^2}\vol(B).
	$$
\end{theorem}

\noindent Our final lemma before embarking on the proof in earnest shows that bounds on the $j$-th smallest singular value and eigenvector condition number are sufficient to rule out the presence of $j$ eigenvalues in a small region. For our particular application we will take $j=2$.  

\begin{lemma}\label{lem:sigma2} Let $D(z_0,r) :=\{z \in \mathbb{C} :|z-z_0|<r\}$. \marginpar{$D(z_0, r)$}  If $M\in \mathbb{C}^{n\times n}$ is a diagonalizable matrix with at least $j$ eigenvalues in $D(z_0,r)$ then 
$$\sigma_{n-j+1}(z_0-M) \leq r\kappa_V(M).$$
\end{lemma}
\begin{proof} Write $M=VDV^{-1}$. By Courant-Fischer:
\begin{align*}
    \sigma_{n-j+1}(z_0-M) &= \min_{S:\dim(S)=j}\max_{x\in S\setminus\{0\}} \frac{\|V(z_0-D)V^{-1}x\|}{\|x\|} & \\
    &=\min_{S:\dim(S)=j}\max_{y\in V(S)\setminus\{0\}} \frac{\|V(z_0-D)y\|}{\|Vy\|}&\textrm{setting $y=Vx$}\\
    &=\min_{S:\dim(S)=j}\max_{y\in S\setminus\{0\}} \frac{\|V(z_0-D)y\|}{\|Vy\|}&\textrm{since $V$ is invertible}\\
    &\le \min_{S:\dim(S)=j}\max_{y\in S\setminus\{0\}} \frac{\|V\|\|(z_0-D)y\|}{\sigma_n(V)\|y\|} & \\
    &\le \kappa_V(M)\sigma_{n-j+1}(z_0-D). &
\end{align*}
Since $z_0-D$ is diagonal its singular values are just $|z_0-\lambda_i|$, so the $j$-th smallest is at most $r$, finishing the proof.
\end{proof}

We now present the main tail bound that we use to control the minimum gap and eigenvector condition number.
\begin{theorem}[Multiparameter Tail Bound]\label{thm:polygap} 
    Let $A\in \mathbb{C}^{n\times n}$. Assume $\|A\|\le 1$ and $\gamma<1/2$, and let $X:=A+\gamma G_n$ where $G_n$ is a complex Ginibre matrix. For every $t,r>0$: 
        \begin{equation}\label{eqn:smoothed}
            \P\big[\kappa_V(X)<t,\, \gap(X)>r, \, \|G_n\|<4\big]\ge 1-\left(\frac{144}{r^2}\cdot 4(trn/\gamma)^8 + (9n^3/\gamma^2t^2)+\expbound\right).
        \end{equation}
\end{theorem}

\begin{proof} 
    Write $\Lambda(X):=\{\lambda_1,\ldots,\lambda_n\}$ for the (random) eigenvalues of $X:=A+\gamma G_n$, in increasing order of magnitude (there are no ties almost surely). Let $\mathcal{N}\subset \C$ be a minimal $r/2$-net of $B:=D(0,3)$, recalling the standard fact that one exists of size no more than $(3\cdot 4/r)^2=144/r^2$. The most useful feature of such a net is that, by the triangle inequality, for any $a,b \in D(0,3)$ with distance at most $r$, there is a point $y\in \mathcal{N}$ with $|y-(a+b)/2|<r/2$ satisfying $a,b\in D(y,r)$. In particular, if $\gap(X) < r$, then there are two eigenvalues in the disk of radius $r$ centered at some point $y \in \mathcal{N}$.
    
    Therefore, consider the events
    \begin{align*}
        E_\gap &:= \{\gap(X)<r\} \subset\{\exists y\in \mathcal{N}: |D(y,r)\cap \Lambda(X)|\ge 2\}
        \marginnote{$E_\gap$} \\
        E_D &:= \{\Lambda(X)\not\subseteq D(0,3)\}\subset \{\|G_n\|\ge 4\}:=E_G 
        \marginnote{$E_D,E_G$} \\
        E_\kappa &:= \{\kappa_V(X) > t\}
        \marginnote{$E_\kappa$} \\
        E_y &:=\{\sigma_{n-1}(y-X) < rt\},\qquad y\in \mathcal{N}.
        \marginnote{$E_y$}
    \end{align*}
    Lemma \ref{lem:sigma2} applied to each $y\in \mathcal{N}$ with $j=2$ reveals that 
    $$ 
        E_\gap \subseteq E_D\cup E_\kappa\cup \bigcup_{y\in \mathcal{N}} E_y,
    $$
    whence
    $$ 
        E_\gap \cup E_\kappa \subseteq E_D\cup E_\kappa\cup \bigcup_{y\in \mathcal{N}} E_y.
    $$
    By a union bound, we have
    \begin{equation}\label{eqn:gapunion}
        \P[E_\gap \cup E_\kappa] \le \P[E_D\cup E_\kappa]+|\mathcal{N}|\max_{y\in \mathcal{N}} \P[E_y].
    \end{equation}
    From the tail bound on the operator norm of a Ginibre matrix in \cite[Lemma 2.2]{banks2019gaussian}, 
    \begin{equation}\label{eqn:ginibrenorm}
        \P[E_D] \le \P[E_G]\le  2e^{-(4-2\sqrt{2})^2n}\le \expbound.
    \end{equation}
    Observe that by \eqref{eqn:kappav},
    $$ 
        \left\{\kappa_V(X) > \sqrt{n\sum_{\lambda_i\in D(0,3)}\kappa(\lambda_i)^2}\right\} \subset E_D,
    $$
    since the inequality in the left hand event must reverse when we sum over all $\lambda_i \in \Lambda(X)$; thus
    $$ 
        E_\kappa \subset E_D \cup \left\{\sum_{\lambda_i\in D(0,3)}\kappa(\lambda_i)^2 > t^2/n\right\}.
    $$
    Theorem \ref{thm:davies} and Markov's inequality yields
    $$ 
        \P \left[\sum_{\lambda_i\in D(0,3)}\kappa(\lambda_i)^2 > t^2/n\right] \le \dE \sum_{\lambda_i\in D(0,3)}\kappa(\lambda_i)^2 \frac{n}{t^2} \le \frac{9\pi n^2}{\pi \gamma^2} \frac{n}{t^2} = \frac{9n^3}{t^2\gamma^2}.
    $$
    Thus, we have
    $$
        \P[E_\kappa \cup E_D]\le \frac{9n^3}{t^2\gamma^2} + \expbound.
    $$
    
    Corollary \ref{corr:sigma2szarek} applied to $M=-y+A$ gives the bound
    $$ 
        \P[E_y]\le 4\left(\frac{trn}{\gamma}\right)^8,
    $$
    for each $y\in \mathcal{N}$, and plugging these estimates back into \eqref{eqn:gapunion} we have
    $$ 
        \P[E_\gap \cup E_\kappa \cup E_D] \le \P[E_\gap \cup E_\kappa\cup E_G]\le \frac{144}{r^2}\cdot 4\left(\frac{trn}{\gamma}\right)^8 + \frac{9n^3}{\gamma^2t^2}+\expbound,
    $$
    as desired.
    \end{proof}
    
A specific setting of parameters in Theorem \ref{thm:polygap} immediately yields Theorem \ref{thm:smoothed}.

\begin{proof}[Proof of Theorem \ref{thm:smoothed}]
Applying Theorem \ref{thm:polygap} with parameters $ t:=\frac{n^2}{\gamma}$ and $r := \frac{\gamma^4}{n^5}$, we have 
\begin{equation}\label{event:goodkappa} 
    \P\big[\gap(X)>r,\, \kappa_V(X)<t,\, \|G\| \le 4\big]
    \ge 1-{600} \frac{n^{10}}{\gamma^8}\left(\frac{\gamma^2}{n^2}\right)^8 - \frac{9}{n^2}-\expbound\ge 1-12/n,
\end{equation}
as desired, where in the last step we use the assumption $\gamma < 1/2$.
\end{proof}
    Since it is of independent interest in random matrix theory, we record the best bound on the gap alone that is possible to extract from the theorem above.
    \begin{corollary}[Minimum Gap Bound] 
    \label{cor:mingapbound}
    
    For $X$ as in Theorem \ref{thm:polygap},
    $$
        \P[\gap(X)<r]\le 2\cdot 9^{4/5}(144\cdot 4)^{1/5}(n/\gamma)^{2+6/5}r^{6/5}\le 42(n/\gamma)^{3.2}r^{1.2}+ \expbound. 
    $$
    In particular, the probability is $o(1)$ if $r=o((\gamma/n)^{8/3})$.
    \end{corollary}
    \begin{proof}
    Setting
    $$t^{10} = \frac{9}{144\cdot 4}(\gamma/nr)^6$$
    in Theorem \ref{thm:polygap} balances the first two terms and yields the advertised bound.
\end{proof}

\renewcommand{\c}{\psi}
\subsection{Shattering}\label{sec:shatter}
Propositions \ref{prop:decrementeps} and \ref{thm:componentsofpseudoespec} in the preliminaries together tell us that if the $\epsilon$-pseudospectrum of an $n\times n$ matrix $A$ has $n$ connected components, then each eigenvalue of any size-$\epsilon$ perturbation $\widetilde{A}$ will lie in its own connected component of $\Lambda_\epsilon(A)$. The following key definitions  make this phenomenon quantitative in a sense which is useful for our analysis of spectral bisection.

\begin{definition}[Grid] 
    A \textit{grid} \marginpar{$\grid$} in the complex plane consists of the boundaries of a lattice of squares with lower edges parallel to the real axis. We will write
    $$
        \grid(z_0,\omega,s_1,s_2) \subset \C
    $$
    to denote an $s_1\times s_2$ grid of $\omega\times \omega$-sized squares and lower left corner at $z_0 \in \C$. Write $\diag(\g) := \omega\sqrt{s_1^2 + s_2^2}$\marginpar{$\diag(\g)$} for the diameter of the grid.
\end{definition}
\begin{definition}[Shattering] 
    A pseudospectrum $\Lambda_\epsilon(A)$ is \emph{shattered} \marginpar{\emph{shattered}} with respect to a grid $\g$ if:
    \begin{enumerate}
        \item Every square of $\g$ has at most one eigenvalue of $A$. 
        \item $\Lambda_\eps(A)\cap \g=\emptyset$.
    \end{enumerate}
\end{definition}

\begin{observation}
    \label{obs:eps-guarantee}
    As $\Lambda_\epsilon(A)$ contains a ball of radius $\epsilon$ about each eigenvalue of $A$, shattering of the $\epsilon$-pseudospectrum with respect to a grid with side length $\omega$ implies $\epsilon \le \omega/2$.
\end{observation}

\noindent As a warm-up for more sophisticated arguments later on, we give here an easy consequence of the shattering property.

\begin{lemma}
    If $\lambda_1, \dots, \lambda_n$ are the eigenvalues of $A$, and $\Lambda_\epsilon(A)$ is shattered with respect to a grid $\g$ with side length $\omega$, then every eigenvalue condition number satisfies $\kappa(\lambda_i) \le \frac{2\omega}{\pi \epsilon}$.
\end{lemma}

\begin{proof}
    Let $v,w^\ast$ be a right/left eigenvector pair for some eigenvalue $\lambda_i$ of $A$, normalized so that $w^\ast v = 1$. Letting $\Gamma$ be the positively oriented boundary of the square of $\g$ containing $\lambda_i$, we can extract the projector $vw^\ast$ by integrating, and pass norms inside the contour integral to obtain
    \begin{align}
        \kappa(\lambda_i) &= \|vw^\ast\|
        = \left\|\frac{1}{2\pi i}\oint_\Gamma (z - A)^{-1}\dee z \right\|
        \le \frac{1}{2\pi}\oint_\Gamma \left\|(z - A)^{-1}\right\|\dee z
        \le \frac{2\omega}{\pi \epsilon}.
    \end{align}
    In the final step we have used the fact that, given the definition of pseudospectrum (\ref{eqn:pseudodef2}) above, $\Lambda_\epsilon(A) \cap \g = \emptyset$ means $\|(z - A)^{-1}\| \le 1/\epsilon$ on $\g$.
\end{proof}

The theorem below quantifies the extent to which perturbing by a Ginibre matrix results in a shattered pseudospectrum.  See Figure \ref{fig:shattering} for an illustration in the case where the initial matrix is poorly conditioned.  In general, not all eigenvalues need move so far upon such a perturbation, in particular if the respective $\kappa_i$ are small.
\begin{figure}
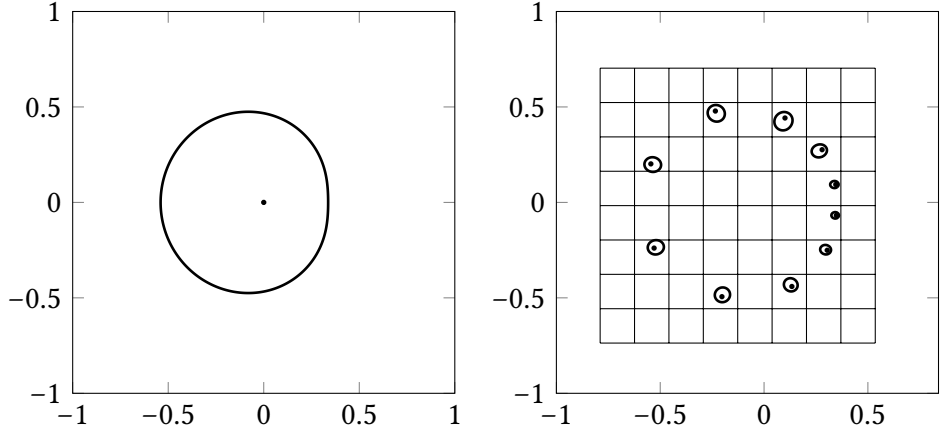

    \centering
    \input{Content/Figures/Toeplitz.tex}
    \input{Content/Figures/Toeplitz-perturbed.tex}
    
    \caption{ $T$ is a sample of an upper triangular $10\times 10$ Toeplitz matrix with zeros on the diagonal and an independent standard real Gaussian repeated along each diagonal above the main diagonal. $G$ is a sample of a $10\times 10$ complex Ginibre matrix with unit variance entries. Using the MATLAB package EigTool \cite{wright2002eigtool}, the boundaries of the $\eps$-pseudospectrum of $T$ (left) and $T+10^{-6} G$ (right) for $\eps = 10^{-6}$ are plotted along with the spectra.  The latter pseudospectrum is shattered with respect to the pictured grid.}
    \label{fig:shattering}
\end{figure}
\bigskip
\newcommand{\epsexactshatter}{\frac{\gamma^5}{16n^9}}
\begin{theorem}[Exact Arithmetic Shattering]\label{thm:exactshatter} 
    Let $A\in \mathbb{C}^{n\times n}$ and $X:=A+\gamma G_n$ for $G_n$ a complex Ginibre matrix. Assume $\|A\|\leq 1$ and $0< \gamma < 1/2$. Let $\g := \grid (z, \omega,\lceil 8/\omega\rceil , \lceil 8/\omega\rceil )$ with $\omega := \frac{\gamma^4}{4n^5}$, and $z$ chosen uniformly at random from the square of side $\omega$ cornered at $-4-4i$. Then, $\kappa_V(X)\le n^2/\gamma$, $\|A-X\|\le 4\gamma$, and $\Lambda_\eps(X)$ is shattered with respect to $\g$ for
   
    $$\eps := \epsexactshatter,$$
    with probability at least $1-13/n$.
\end{theorem}

\begin{proof}
 Condition on the event in Theorem \ref{thm:smoothed}, so that
 $$
 \kappa_V(X)\le \frac{n^2}{\gamma},\quad \|X-A\|\le 4\gamma,\quad \text{ and } \gap(X)\ge\frac{\gamma^4}{n^5}=4\omega.$$
 Consider the random grid $\g$. Since $D(0,3)$ is contained in the square of side length $8$ centered at the origin, every eigenvalue of $X$ is contained in one square of $\g$ with probability $1$. Moreover, since $\gap(X)>4\omega$, no square can contain two eigenvalues. Let 
$$\dist_\g(z):=\min_{y\in \g}|z-y|.$$ 
 Let $\lambda_i := \lambda_i(X)$.  We now have for each $\lambda_i$  and every $s < \frac{\omega}{2}$ :
\begin{equation*}
 \P[\dist_\g(\lambda_i)>s] = \frac{(\omega-2s)^2}{\omega^2} = 1- \frac{4 s}{\omega}+ \frac{4s^2 }{\omega^2} \geq    1-\frac{4s}{\omega},
 \end{equation*}
since the distribution of $\lambda_i$ inside its square is uniform with respect to Lebesgue measure.
Setting $s=\omega /4n^2$, this probability is at least $1-1/n^2$, so by a union bound 
\begin{equation}\label{event:goodspace} \P[\min_{i\le n} \dist_\g(\lambda_i)>\omega/4n^2]>1-1/n,\end{equation}
i.e., every eigenvalue is well-separated from $\g$ with probability $1-1/n$.

We now recall from \eqref{eqn:lambdakappa} that
$$ \Lambda_\eps(X)\subset \bigcup_{i\le n} D(\lambda_i, \kappa_V(X)\epsilon).$$
Thus, on the events \eqref{event:goodkappa} and \eqref{event:goodspace}, we see that $\Lambda_\eps(X)$ is shattered with respect to $\g$ as long as
$$ \kappa_V(X)\epsilon < \frac{\omega}{4n^2},$$
which is implied by
$$ \epsilon < \frac{\gamma^4}{4n^5}\cdot \frac{1}{4n^2} \cdot \frac{\gamma}{n^2} =\frac{\gamma^5}{16n^9}.$$
Thus, the advertised claim holds with probability at least
$$ 1-\frac{1}{n}- \frac{13}{n} = 1 - \frac{13}{n} ,$$
as desired.
\end{proof}

Finally, we show that the shattering property is retained when the Gaussian perturbation is added in finite precision rather than exactly. This also serves as a pedagogical warmup for our presentation of more complicated algorithms later in the paper: we use $E$ to represent an adversarial roundoff error (as in step 2), and for simplicity neglect roundoff error completely in computations whose size does not grow with $n$ (such as steps 3 and 4, which set scalar parameters).

\begin{figure}[ht]
    \begin{boxedminipage}{\textwidth}
        $$ \SHATTER $$ {\small
        \textbf{Input:} Matrix $A \in \C^{n\times n}$, Gaussian perturbation size $\gamma\in (0,1/2)$.\\
        \textbf{Requires:} $\|A\| \le 1$.\\
        \textbf{Algorithm:} $(X,\g, \epsilon) = \SHATTER(A,\gamma)$
        \begin{enumerate}
            \item $G_{ij}\gets \N(1/{n})$ for $i,j=1,\ldots,n$.
            \item $X\gets A + \gamma G + E$.
            \item Let $\g$ be a random grid with $\omega = \frac{\gamma^4}{4n^5}$ and bottom left corner $z$ chosen as in Theorem \ref{thm:exactshatter}.
            \item $\epsilon \gets \frac{1}{2}\cdot\epsexactshatter$
        \end{enumerate}
        \textbf{Output:} Matrix $X\in \C^{n\times n}$, grid $\g$, shattering parameter $\epsilon>0$.\\
        \textbf{Ensures:} $\|X-A\|\le 4\gamma$, $\kappa_V(X)\le n^2/\gamma$, and $\Lambda_\epsilon(X)$ is shattered with respect to $\g$, with probability at least $1-13/n$.}
    \end{boxedminipage}
\end{figure}

\begin{theorem}[Finite Arithmetic Shattering]\label{thm:finiteshatter} Assume there is a $\cn$-stable Gaussian sampling algorithm $\N$ satisfying the requirements of Definition \ref{def:gaussian}. Then $\SHATTER$ has the advertised guarantees as long as the machine precision satisfies
\begin{equation}\label{eqn:shatterprecision}
\mach \le \frac{1}{2}\epsexactshatter\cdot\frac{1}{(3+\cn)\sqrt{n}},
\end{equation}\label{eqn:shatterbitops}
and runs in $$n^2T_\N+ n^2=O(n^2)$$ arithmetic operations.
\end{theorem}
\begin{proof}
    The two sources of error in $\SHATTER$ are:
    \begin{enumerate}
        \item An additive error of operator norm at most $n\cdot \cn\cdot (1/\sqrt{n})\cdot \mach\le  \cn\sqrt{n}\cdot \mach$ from $\N$, by Definition \ref{def:gaussian}.
        \item An additive error of norm at most $\sqrt{n}\cdot\|X\|\cdot \mach\le 3\sqrt{n}\mach$, with probability at least $1-1/n$,  from the roundoff $E$ in step 2.
    \end{enumerate}
    Thus, as long as the precision satisfies \eqref{eqn:shatterprecision}, we have
    $$ \|\SHATTER(A,\gamma)-\mathrm{shatter}(A,\gamma)\|\le \frac{1}{2}\epsexactshatter,$$
    where $\mathrm{shatter}(\cdot)$ refers to the (exact arithmetic) outcome of Theorem \ref{thm:exactshatter}. The correctness of $\SHATTER$ now follows from Proposition \ref{prop:decrementeps}. Its running time is bounded by
    $$n^2T_\N+ n^2$$
    arithmetic operations, as advertised.
\end{proof}

\section{Matrix Sign Function}
\label{sec:matrix-sign}
The algortithmic centerpiece of this work is the analysis, in finite arithmetic, of a well-known iterative method for approximating to the matrix sign function. Recall from Section \ref{sec:intro} that if $A$ is a matrix whose spectrum avoids the imaginary axis, then
$$
    \sgn(A) = P_+ - P_-
$$
where the $P_+$ and $P_-$ are the spectral projectors corresponding to eigenvalues in the open right and left half-planes, respectively. The iterative algorithm we consider approximates the matrix sign function by repeated application to $A$ of the function 
\begin{equation}
    g(z) := \frac{1}{2}(z + z^{-1}). 
    \marginnote{$\newton$} 
\end{equation}
This is simply Newton's method to find a root of $z^2 - 1$, but one can verify that the function $g$ fixes the left and right halfplanes, and thus we should expect it to push those eigenvalues in the former towards $-1$, and those in the latter towards $+1$.  

We denote the specific finite-arithmetic implementation used in our algorithm by $\SGN$; the pseudocode is provided below.

\begin{figure}[ht]
    \begin{boxedminipage}{\textwidth}
    $$\SGN$$ {\small
        \textbf{Input:} Matrix $A \in \C^{n\times n}$, pseudospectral guarantee $\epsilon$, circle parameter $\alpha$, and desired accuracy $\delta$ \\
        \textbf{Requires:} $\Lambda_\epsilon(A) \subset \cpm{\alpha}$. \\
        \textbf{Algorithm:} $S = \SGN(A,\epsilon,\alpha,\delta)$
        \begin{enumerate}
            \item $N \gets  \lceil \lg(1/(1-
\alpha)) + 3 \lg \lg(1/(1-\alpha)) + \lg \lg (1/(\sgnerr \eps)) + 7.59 \rceil$ 
            
            \item $A_0 \gets A$
            \item For $k = 1,...,N$, 
            \begin{enumerate}
                \item $A_k \gets \tfrac{1}{2}(A_{k-1} + A^{-1}_{k-1}) + E_k$
            \end{enumerate}
            \item $S \gets A_N$
        \end{enumerate}
        \textbf{Output:} Approximate matrix sign function $S$ \\
        \textbf{Ensures:} $\|S - \sgn(A)\| \le \delta$ }
    \end{boxedminipage}
\end{figure}

In Subsection \ref{sec:circlesappolonius} we briefly discuss the specific preliminaries that will be used throughout this section. In Subsection \ref{sec:exactarithnewton} we give a \emph{pseudospectral} proof of the rapid global convergence of this iteration when implemented in exact arithmetic. In Subsection \ref{sec:exactarithnewton} we show that the proof provided in Subsection \ref{sec:finitearithnewton} is robust enough to handle the finite arithmetic case; a formal statement of this main result is the content of Theorem \ref{prop:sgnerr}.

\subsection{Circles of Apollonius}
\label{sec:circlesappolonius}

It has been known since antiquity that a circle in the plane may be described as the set of points with a fixed ratio of distances to two focal points. By fixing the focal points and varying the ratio in question, we get a family of circles named for the Greek geometer Apollonius of Perga. We will exploit several interesting properties enjoyed by these \emph{Circles of Apollonius} in the analysis below.

More precisely, we analyze the Newton iteration map $\newton$ in terms of the family of Apollonian circles whose foci are the points $\pm 1 \in \C$. 
For the remainder of this section we will write $m(z) = \tfrac{1 - z}{1 + z}$ for the M\"obius transformation taking the right half-plane to the unit disk, and for each $\alpha \in (0,1)$ we denote by 
$$
    \cp{\alpha} = \left\{z \in \C : |m(z)| \le \alpha \right\}, \quad \cm{\alpha} = \{ z \in \mathbb{C} : |m(z)|^{-1} \leq \alpha \}
    \marginnote{$\cp{\alpha}, \cm{\alpha}$}
$$
the closed region in the right (respectively left) half-plane bounded by such a circle. Write $\partial\cp{\alpha}$ and $\partial\cm{\alpha}$ for their boundaries, and $\cpm{\alpha} = \cp{\alpha} \cup \cm{\alpha}$ for their union. See Figure \ref{fig:Calpha} for an illustration.

The region $\cp{\alpha}$ is a disk centered at $\tfrac{1 + \alpha^2}{1 - \alpha^2} \in \R$, with radius $\tfrac{2\alpha}{1-\alpha^2}$, and whose intersection with the real line is the interval $(m(\alpha),m(\alpha)^{-1})$; $\cm{\alpha}$ can be obtained by reflecting $\cp{\alpha}$ with respect to the imaginary axis. 
  \begin{figure}
      \centering
      \begin{tikzpicture}[line cap=round,line join=round
,x=1.0cm,y=1.0cm]
\begin{axis}[
x=1.0cm,y=1.0cm,
axis lines=middle,
xmin=-0.5799999999999995,
xmax=5.400000000000007,
ymin=-3.059999999999997,
ymax=3.039999999999999,
xtick={-0.0,1.0,...,5.0},
ytick={-3.0,-2.0,...,3.0},]
\clip(-0.58,-3.06) rectangle (5.4,3.04);
\draw [line width=1.pt] (4.555555555555556,0.) circle (4.4444444444444455cm);
\draw [line width=1.pt] (2.387533875338754,0.) circle (2.1680216802168024cm);
\draw [line width=1.pt] (1.3924160411495559,0.) circle (0.9689284966655702cm);
\draw [line width=1.pt] (1.057925451863556,0.) circle (0.34526259817812477cm);
\end{axis}
\end{tikzpicture}
      \caption{ Apollonian circles appearing in the analysis of the Newton iteration.  Depicted are $\partial \cp{\alpha^{2^{k}}}$ for $\alpha=0.8$ and $k = 0, 1, 2, 3$, with smaller circles corresponding to larger $k$.}
      \label{fig:Calpha}
  \end{figure}
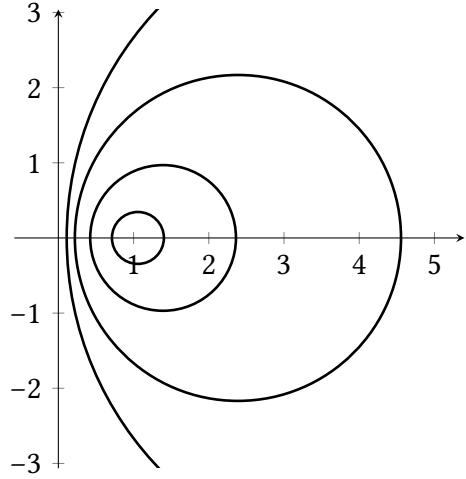
For $\alpha > \beta > 0$, we will write
$$ 
    \ap{\alpha}{\beta} = \cp{\alpha} \setminus \cp{\beta}
    \marginnote{$\ap{\alpha}{\beta},\am{\alpha}{ \beta} $}
$$
for the \emph{Apollonian annulus} lying inside  $\cp{\alpha}$ and outside $\cp{\beta}$; note that the circles are not concentric so this is not strictly speaking an annulus, and note also that in our notation this set does not include $\partial \cp{\beta}$. In the same way define $\am{\alpha}{\beta}$ for the left half-plane and write $\apm{\alpha}{\beta} = \ap{\alpha}{\beta} \cup \am{\alpha}{\beta}$.

\begin{observation}[\cite{roberts1980linear}]
\label{obs:newtonmap}
The Newton map $\newton$ is a two-to-one map from $\cp{\alpha}$ to $\cp{\alpha^2}$, and a two-to-one map from $\cm{\alpha}$ to $\cm{\alpha^2}$.  
\end{observation}

\begin{proof}
 This follows from the fact that for each $z$ in the right half-plane,
$$
    |m(g(z))| = \left|\frac{1 - \tfrac{1}{2}(z + 1/z)}{1 + \tfrac{1}{2}(z + 1/z)}\right| = \left|\frac{(1-z)^2}{(z + 1)^2}\right| = |m(z)|^2
$$
and similarly for the left half-plane.
\end{proof}
It follows from Observation \ref{obs:newtonmap} that under repeated application of the Newton map $g$, any point in the right or left half-plane converges to $+1$ or $-1$, respectively.

\subsection{Exact Arithmetic}
\label{sec:exactarithnewton}

In this section, we set $A_0 := A$ and $A_{k+1} := \newton (A_k)$ for all $k \ge 0$. \marginnote{$A_k$} In the case of exact arithmetic, Observation \ref{obs:newtonmap} implies global convergence of the Newton iteration when $A$ is diagonalizable. For the convenience of the reader we provide this argument (due to \cite{roberts1980linear}) below. 

\begin{proposition}
\label{prop:newtonexactarithmetic}
Let $A$ be a diagonalizable $n \times n$ matrix and assume that $\Lambda(A) \subset  \cpm{\alpha}$ for some $\alpha \in (0,1)$. Then for every $N \in \mathbb{N}$ we have the guarantee
$$\|A_N - \sgn(A) \| \leq  \frac{4\alpha^{2^N}}{\alpha^{2^{N+1}}+1} \cdot \kappa_V(A).$$
 Moreover, when $A$ does not have eigenvalues on the imaginary axis  the minimum $\alpha$ for which $\Lambda(A) \subset  \cpm{\alpha}$ is given by
 $$\alpha^2 = \max_{1 \le i \le n} \left\{ 1 - \frac{4|\re(\lambda_i(A))|}{|\lambda_i(A)-\sgn(\lambda_i(A))|^2} \right\}$$
\end{proposition}

\begin{proof}
 Consider the spectral decomposition $A = \sum_{i=1}^n \lambda_i v_i w_i^*,$
and denote by $\lambda_i^{(N)}$ the eigenvalues of $A_N$. 

By Observation \ref{obs:newtonmap} we have that $\Lambda(A_N) \subset \cpm{\alpha^{2^N}}$ and $\sgn(\lambda_i) = \sgn(\lambda_i^{(N)})$. Moreover,   $A_N$ and $\sgn (A)$ have the same eigenvectors. Hence  
\begin{equation}
\label{eq:distancetosignexact}
    \|A_N - \sgn(A) \| \leq \left\|\sum_{\re(\lambda_i) > 0} (\lambda_i^{(N)}-1) v_i w_i^* \right\| + \left\|\sum_{\re(\lambda_i) < 0} (\lambda_i^{(N)}+1) v_i w_i^* \right\|.
\end{equation}

Now we will use that for any matrix $X$ we have that  $\| X \| \leq \kappa_V(X) \spr(X)$ where $\spr(X)$ denotes the spectral radius of $X$. Observe that the spectral radii of the two matrices appearing on the right hand side of (\ref{eq:distancetosignexact}) are bounded by $\max_{i} |\lambda_i- \sgn(\lambda_i)|$, which in turn is bounded by the radius of the circle $\cp{\alpha^{2^N}}$, namely $2\alpha^{2^N}/(\alpha^{2^{N+1}}+1)$. On the other hand, the eigenvector condition number of these matrices is bounded by $\kappa_V(A)$. This concludes the first part of the statement. 

In order to compute $\alpha$ note that if $z = x+ i y$ with $ x > 0$, then 
$$|m(z)|^2 = \frac{(1-x)^2+ y^2}{(1+x)^2+ y^2} = 1- \frac{4x}{(1+x)^2+y^2},$$
and analogously when $x < 0$ and we evaluate $|m(z)|^{-2}$. 
\end{proof}

The above analysis becomes useless when trying to prove the same statement in the framework of finite arithmetic. This is due to the fact that at each step of the iteration the roundoff error can make the eigenvector condition numbers of the $A_k$ grow. In fact, since $\kappa_V(A_k)$ is sensitive to infinitesimal perturbations whenever $A_k$ has a multiple eigenvalue, it seems difficult to control it against adversarial perturbations as the iteration converges to $\sgn(A_k)$ (which has very high multiplicity eigenvalues). A different approach, also due to \cite{roberts1980linear}, yields a proof of convergence in exact arithmetic even when $A$ is not diagonalizable. However, that proof relies heavily on the fact that $m(A_N)$ is an exact power of $m(A_0)$, or more precisely, it requires the sequence $A_k$ to have the same generalized eigenvectors, which is again not the case in the finite arithmetic setting. 

Therefore, a \emph{robust} version, tolerant to perturbations, of the above proof is needed. To this end, instead of simultaneously keeping track of the eigenvector condition number and the spectrum of the matrices $A_k$, we will just show that for certain $\eps_k > 0$, the $\epsilon_k-$pseudospectra of these matrices  are contained in a certain shrinking region dependent on $k$. This invariant is inherently robust to perturbations smaller than $\epsilon_k$, unaffected by clustering of eigenvalues due to convergence, and  allows us to bound the accuracy and other quantities of interest via the functional calculus.  For example, the following lemma shows how to obtain a bound on $\|A_N- \sgn(A)\|$ solely using information from the pseudospectrum of $A_N$. 

\begin{lemma}[Pseudospectral Error Bound] \label{lem:boundforsign}
Let $A$ be any $n \times n$ matrix and let $A_N$ be the $N$th iterate of the Newton iteration under exact arithmetic.  Assume that $\eps_N > 0$ and $\alpha_N \in (0, 1)$ satisfy  $\Lambda_{\epsilon_N}(A_N) \subset \cpm{\alpha_N}$. Then we have the guarantee 
\begin{equation}
\label{eq:boundforsign}
\|A_N - \sgn(A)\| \leq \frac{8  \alpha_N^2}{(1- \alpha_N)^2 (1+\alpha_N) \eps_N}.
\end{equation}
\end{lemma}

\begin{proof}
Note that $\sgn(A) = \sgn(A_N)$.  Using the functional calculus we get
\begin{align*}
    \Vert A_N - \sgn (A_N) \Vert &= \left\Vert \frac{1}{2\pi i} \oint_{\partial \cpm{\alpha_N}} z(z-A_N)^{-1}\,dz - \frac{1}{2\pi i}\left(  \oint_{\partial \cp{\alpha_N}} (z-A_N)^{-1}\,dz - \oint_{\partial \cm{\alpha_N}} (z-A_N)^{-1}\,dz\right) \right\Vert \\ &= \left\Vert \frac{1}{2\pi i}\oint_{\partial\cp{\alpha_N}} z (z-A_N)^{-1} - (z-A_N)^{-1}\,dz + \frac{1}{2\pi i} \oint_{\partial\cm{\alpha_N}} z (z-A_N)^{-1} + (z-A_N)^{-1}\,dz \right\Vert \\
    &\le  \frac{1}{2\pi } \left\Vert \oint_{\partial\cp{\alpha_N}} (z - 1)  (z-A_N)^{-1} \,dz \right\Vert + \frac{1}{2\pi} \left\Vert \oint_{\partial\cm{\alpha_N}} (z + 1)  (z-A_N)^{-1} \,dz\right\Vert 
    \\ &\le  2 \cdot \frac{1}{2\pi} \ell(\partial \cp{\alpha_{N}} ) \sup \{|z - 1| : z \in \cp{\alpha_N}\} \frac{1}{\eps_N} \\
    & =  \frac{ 4 \alpha_N}{1 - \alpha_N^2} \left( \frac{1+\alpha_N}{1-\alpha_N} - 1\right) \frac{1}{\eps_N} \\
    &= \frac{ 8 \alpha_N^2}{(1- \alpha_N)^2 (1+\alpha_N) \eps_N}.
\end{align*}    
\end{proof}
In view of Lemma \ref{lem:boundforsign}, we would now like to find sequences $\alpha_k$ and $\epsilon_k$ such that 
$$\Lambda_{\epsilon_k}(A_k)\subset \cpm{\alpha_k}$$ and $\alpha_k^2/\epsilon_k$ converges rapidly to zero.  The dependence of this quantity on the \emph{square} of $\alpha_k$ turns out to be crucial.  As we will see below, we can find such a sequence with $\epsilon_k$ shrinking roughly at the same rate as $\alpha_k$.  This yields quadratic convergence, which will be necessary for our bound on the required machine precision in the finite arithmetic analysis of Section \ref{sec:finitearithnewton}.

The lemma below is instrumental in determining the sequences $\alpha_k, \eps_k$.

\begin{lemma}[Key Lemma]  
\label{thm:iteratedpseudospec}
 If $\Lambda_\epsilon(A) \subset \cpm{\alpha}$, then for every $\alpha'>\alpha^2$,  we have $\Lambda_{\epsilon'}(\newton(A))\subset \cpm{\alpha'}$ where
    $$ \epsilon' := \epsilon \, \frac{(\alpha' - \alpha^2)(1-\alpha^2)}{8\alpha}.$$
\end{lemma}

\begin{proof}
     From the definition of pseudospectrum, our hypothesis implies $\|(z - A)^{-1}\| < 1/\epsilon$ for every $z$ outside of $\cpm{\alpha}$. The proof will hinge on the observation that, for each $\alpha' \in (\alpha^2,\alpha)$, this resolvent bound allows us to bound the resolvent of $\newton(A)$ everywhere in the Appolonian annulus $\apm{\alpha}{\alpha'}$.
     
     Let $w \in \apm{\alpha}{\alpha'}$; see Figure \ref{fig:Calphaprime} for an illustration.  We must show that $w \not\in \Lambda_{\eps'}(g(A))$.  Since $w \not\in \cpm{\alpha^2}$, Observation \ref{obs:newtonmap} ensures no $z \in \cpm{\alpha}$ satisfies $g(z) = w$; in other words, the function $(w - \newton(z))^{-1}$ is holomorphic in $z$ on $\cpm{\alpha}$. As $\Lambda(A) \subset \Lambda_\epsilon(A) \subset \cpm{\alpha}$, Observation \ref{obs:newtonmap} also guarantees that $\Lambda(\newton(A)) \subset \cpm{\alpha^2}$. Thus for $w$ in the union of the two Appolonian annuli in question, we can calculate the resolvent of $\newton(A)$ at $w$ using the holomorphic functional calculus:
    $$
        (w - \newton(A))^{-1} =  \frac{1}{2\pi i}\oint_{\partial \cpm{\alpha}} (w - g(z))^{-1}(z - A)^{-1}\dee z,
    $$
    where by this we mean to sum the integrals over $\partial \cp{\alpha}$ and $\partial\cm{\alpha}$, both positively oriented. Taking norms, passing inside the integral, and applying Observation \ref{obs:newtonmap} one final time, we get:
    \begin{align*}
        \left\| (w - \newton(A))^{-1} \right\|
        &\le \frac{1}{2\pi}\oint_{\partial \cpm{\alpha}}|(w - g(z))^{-1}|\cdot \|(z - A)^{-1}\| \dee z\\
        &\le \frac{\ell\left(\partial\cp{\alpha}\right) \sup_{y \in \cp{\alpha^2}}|(w - y)^{-1}| + \ell\left(\partial \cm{\alpha}\right)\sup_{y \in \cm{\alpha^2}}|(w - y)^{-1}|}{2\pi \epsilon} \\
        &\le \frac{1}{\epsilon} \frac{8\alpha}{(\alpha' - \alpha^2)(1 - \alpha^2)}.
    \end{align*}
    In the last step we also use the forthcoming Lemma \ref{lem:circledist}.  Thus, with $\epsilon'$ defined as in the theorem statement, $\apm{\alpha}{\alpha'}$ contains none of the $\epsilon'$-pseudospectrum of $\newton(A)$. Since $\Lambda(\newton(A)) \subset \cpm{\alpha^2}$, Theorem \ref{thm:componentsofpseudoespec} tells us that there can be no $\epsilon'$-pseudospectrum in the remainder of $\C \setminus \cpm{\alpha'}$, as such a connected component would need to contain an eigenvalue of $\newton(A)$.
\end{proof}

\begin{figure}[H]
    \centering
    \begin{tikzpicture}[line cap=round,line join=round,>=triangle 45,x=3.1286877114840492cm,y=3.067139756504037cm]
\begin{axis}[
x=3.1286877114840492cm,y=3.067139756504037cm,
axis lines=middle,
xmin=0.368234986938463,
xmax=2.285971927565875,
ymin=-0.9932584740266744,
ymax=0.9629614821317215,
xticklabels={,,,1,,2},
]
\clip(0.368234986938463,-0.9932584740266744) rectangle (2.285971927565875,0.9629614821317215);
\draw [line width=1.pt] (1.3172285946008575,0.) ellipse (2.682457312592048cm and 2.6296876605410584cm);
\draw [line width=1.pt] (1.1066491112574062,0.) ellipse (1.4829856170299647cm and 1.453812129481835cm);
\draw [line width=1.pt] (1.0381991332578568,0.) ellipse (0.8729960468458597cm and 0.8558223541210236cm);
\begin{scriptsize}
\draw[color=black] (1.4802770271416008,0.731304896711305) node {$C^+_\alpha$};
\draw[color=black] (1.2706550315257612,0.34088857732328734) node {$C^{+}_{\alpha'}$};
\draw[color=black] (1.1612751645419606,0.1340596637465323) node {$ C^+_{\alpha^2} $};
\draw [fill=black] (1.7664512178975795,-0.3664747414014565) circle (2pt);
\draw[color=black] (1.8498310848813801,-0.40175083897152597) node {$w$};
\draw [fill=black] (0.6702450462687188,-0.5625864010280601) circle (2pt);
\draw[color=black] (0.6440299315771878,-0.6711319476884198) node {$z$};
\draw [fill=black] (0.7727724781569572,0.08605888135057987) circle (2pt);
\draw[color=black] (0.9298248762159127,0.1070113784622618) node {$g(z)$};
\end{scriptsize}
\end{axis}
\end{tikzpicture}
    \caption{  Illustration of the proof of Lemma \ref{thm:iteratedpseudospec}}
    \label{fig:Calphaprime}
\end{figure}
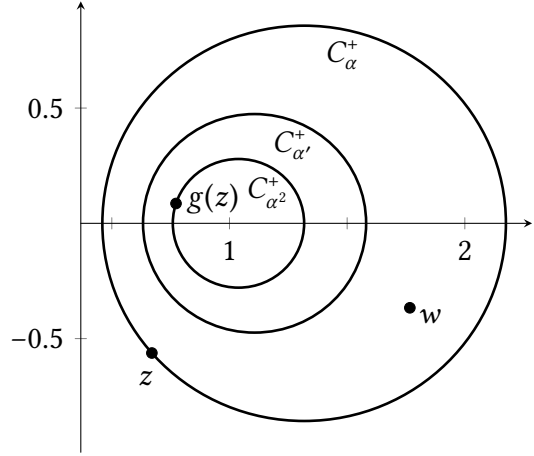

\begin{lemma} \label{lem:circledist}
 Let $1 > \alpha, \beta > 0$ be given.  Then for any $x \in \partial \cpm{\alpha}$ and $y \in \partial \cpm{\beta}$, we have $|x-y| \ge (\alpha-\beta)/2$.
\end{lemma}
\begin{proof}


Without loss of generality $x \in \partial\cp{\alpha}$ and $y \in \partial\cp{\beta}$.  Then we have 
$$|\alpha- \beta| = \left||m(x)|- |m(y)|\right| \le |m(x)- m(y)| = \frac{2|x-y|}{{|1+x||1+y|}} \leq 2|x-y|.$$
\end{proof}

Lemma \ref{thm:iteratedpseudospec} will also be useful in bounding the condition numbers of the $A_k$, which is necessary for the finite arithmetic analysis. 

\begin{corollary}[Condition Number Bound] \label{cor:condbounds}
    Using the notation of Lemma \ref{thm:iteratedpseudospec}, if $\Lambda_\epsilon(A) \subset \cpm{\alpha}$, then
        \begin{align*}
            \|A^{-1}\| &\le \frac{1}{\epsilon} \quad \text{and} \quad
            \|A\| \le \frac{4\alpha}{(1 - \alpha)^2 \eps}.
        \end{align*}
\end{corollary}

\begin{proof}
    The bound $\| A^{-1} \| \le 1/\eps$ follows from the fact that $0 \notin \cpm{\alpha} \supset \Lambda_{\eps}(A).$ In order to bound $A$ we
    use the contour integral bound
    \begin{align*}
        \Vert A \Vert &= \left\Vert \frac{1}{2\pi i} \oint_{\partial\cpm{\alpha}} z (z - A)^{-1}\,dz \right\Vert \\
        &\le \frac{\ell(\partial \cpm{\alpha})}{2\pi} \left(\sup_{z \in \partial\cpm{\alpha}} |z| \right) \frac{1}{\eps} \\
        &= \frac{4 \alpha}{1 - \alpha^2} \frac{1+\alpha}{1-\alpha} \frac{1}{\eps}.
    \end{align*}
\end{proof}

Another direct application of Lemma \ref{thm:iteratedpseudospec} yields the following. 

\begin{lemma} 
\label{lem:speudoincircle}
Let $\epsilon > 0$. If $\Lambda_\epsilon(A) \subset \cpm{\alpha}$, and $ 1/\alpha > D > 1$ then for every $N$ we have the guarantee $$\Lambda_{\epsilon_N}(A_N) \subset \cpm{\alpha_N},$$ 
for $\alpha_N =(D\alpha)^{2^N}/D$ and $\epsilon_N = \frac{\alpha_N \epsilon}{\alpha} \left(\frac{(D-1)(1-\alpha^2)}{8D}\right)^N $. 
\end{lemma}

\begin{proof}
Define recursively $\alpha_0 = \alpha$, $\epsilon_0 = \epsilon$\marginpar{$\alpha_0,\epsilon_0$}, $\alpha_{k+1} = D \alpha_k^2$ and $\epsilon_{k+1}= \frac{1}{8} \epsilon_k \alpha_k (D-1)(1-\alpha_0^2).$ It is easy to see by induction that this definition is consistent with the definition of $\alpha_N$ and $\epsilon_N$ given in the statement. 

We will now show by induction that $\Lambda_{\epsilon_k}(A_k) \subset \cpm{\alpha_k}$. Assume the statement is true for $k$, so from Lemma \ref{thm:iteratedpseudospec} we have that the statement is also true for $A_{k+1}$ if we pick the pseudospectral parameter to be
$$\epsilon' = \epsilon_k \frac{(\alpha_{k+1}-\alpha_k^2)(1-\alpha_k^2)}{8\alpha_k} = \frac{1}{8} \epsilon_k \alpha_k (D-1)(1-\alpha_k^2).$$
On the other hand 
$$ \frac{1}{8} \epsilon_k \alpha_k (D-1)(1-\alpha_k^2) \geq \frac{1}{8} \epsilon_k \alpha_k (D-1) (1-\alpha_0^2) = \epsilon_{k+1},$$
which concludes the proof of the statement. 
\end{proof}

We are now ready to prove the main result of this section, a pseudospectral version of Proposition \ref{prop:newtonexactarithmetic}.  

\begin{proposition}
Let $A\in \mathbb{C}^{n\times n}$ be a diagonalizable matrix and assume that $\Lambda_\epsilon(A) \subset  \cpm{\alpha}$ for some $\alpha \in (0,1)$. Then, for any $1 < D < \frac{1}{\alpha}$ for every $N$ we have the guarantee
$$\|A_N - \sgn(A) \| \leq  (D\alpha)^{2^N}\cdot \frac{\pi \alpha (1-\alpha^2)^2}{8\epsilon}\cdot \left(\frac{8D}{(D-1)(1-\alpha^2)}\right)^{N+2}. $$
\end{proposition}

\begin{proof}
Using the choice of $\alpha_k$ and $\epsilon_k$  given in the proof of Lemma \ref{lem:speudoincircle} and  the bound (\ref{eq:boundforsign}), 
we get that 
\begin{align*}
    \|A_N- \sgn(A)\| &\leq \frac{8 \pi \alpha_N^2}{(1- \alpha_N)^2 (1+\alpha_N) \eps_N} 
    \\ &= \frac{8 \pi \alpha_0 \alpha_N }{\epsilon_0 (1- \alpha_N)^2 (1+\alpha_N) } \left(\frac{8D}{(D-1)(1-\alpha_0^2)}\right)^N 
    \\ &= (D \alpha_0)^{2^N} \frac{8 D^3 \pi \alpha_0}{(D-(D\alpha_0)^{2^{N}})^2(D+(D\alpha_0)^{2^N}) \epsilon_0} \left(\frac{8D}{(D-1)(1-\alpha_0^2)}\right)^N 
    \\& \leq  (D \alpha_0)^{2^N} \frac{8D^2 \pi \alpha_0}{(D-1)^2 \epsilon_0} \left(\frac{8D}{(D-1)(1-\alpha_0^2)}\right)^N
    \\ & = (D\alpha_0)^{2^N}\, \frac{\pi \alpha_0 (1-\alpha_0^2)^2}{8\epsilon_0} \left(\frac{8D}{(D-1)(1-\alpha_0^2)}\right)^{N+2},
\end{align*}
  where the last inequality was taken solely to make the expression more intuitive, since not much is lost by doing so.
\end{proof}

\subsection{Finite Arithmetic}
\label{sec:finitearithnewton}




Finally, we turn to the analysis of $\SGN$ in finite arithmetic. By making the machine precision small enough, we can bound the effect of roundoff to ensure that the parameters $\alpha_k$, $\eps_k$ are not too far from what they would have been in the exact arithmetic analysis above.  We will stop the iteration before any of the quantities involved become prohibitively small, so we will only need $\polylog(1-\alpha_0, \eps_0, \beta)$ bits of precision, where $\beta$ is the accuracy parameter.

In exact arithmetic, recall that the Newton iteration is given by $A_{k+1} = g(A_{k}) = \frac{1}{2} (A_k + A_k^{-1}).$ Here we will consider the finite arithmetic version $\G$ of the Newton map $\newton$, defined as $\G(A) := \newton(A)+E_A$ \marginnote{$\G$} where $E_A$ is an adversarial perturbation coming from the round-off error. Hence,  the sequence of interest is given by $\widetilde{A}_0 := A$ and $\widetilde{A}_{k+1} := \G(\widetilde{A}_k)$ \marginnote{$\widetilde{A}_k$}.

In this subsection we will prove the following theorem concerning the runtime and precision of $\SGN$.  Our assumptions on the size of the parameters $\alpha_0, \beta, \pinv(n)$ and $\cinv$ are in place only to simplify the analysis of constants; these assumptions are not required for the execution of the algorithm.

\begin{theorem}[Main guarantees for $\SGN$] \label{prop:sgnerr}
Assume $\INV$ is a $(\pinv(n), \cinv)$-stable matrix inversion algorithm satisfying Definition \ref{def:inv}. Let $\eps_0\in (0,1), \sgnerr\in (0,1/12)$, assume $\pinv(n) \ge 1$ and $\cinv \log n \ge 1$, and assume $A = \widetilde{A}_0$ is a floating-point matrix with $\eps_0$-pseudospectrum contained in  $\cpm{\alpha_0}$ where $0 < 1 - \alpha_0 < 1/100$.   Run $\SGN$ with

\[ N = \lceil \lg(1/(1-
\alpha_0)) + 3 \lg \lg(1/(1-\alpha_0)) + \lg \lg (1/(\sgnerr \eps_0)) + 7.59 \rceil \]
iterations (as specified in the statement of the algorithm).   
 Then $\widetilde{A_N}=\SGN(A)$ satisfies the advertised accuracy guarantee
\[
    \Vert \widetilde{A_N} - \sgn(A) \Vert \le \sgnerr
\]
when run with machine precision satisfying
\[
    \mach \le  \mach_{\SGN} := \frac{  \alpha_0^{2^{N+1}(\cinv \log n + 3)}}{\pinv(n) \sqrt{n}  N},
\]
corresponding to at most
\newcommand{\sgnfinalprecision}{\lg(1/\mach_{\SGN})) = O(  \log n \log^3(1/(1-\alpha_0)) (\log(1/\sgnerr)+ \log(1/\eps_0)))}
\[
    \sgnfinalprecision
\]
required bits of precision. The number of arithmetic operations is at most
\newcommand{\sgnfinalbitops}{N(4  n^2 + T_\INV(n)) }
\[ 
    \sgnfinalbitops.
\]


\end{theorem}

Later on, we will need to call $\SGN$ on a matrix with shattered pseudospectrum; the lemma below calculates acceptable parameter settings for shattering so that the pseudospectrum is contained in the required pair of Appolonian circles, satisfying the hypothesis of Theorem \ref{prop:sgnerr}.

\begin{lemma} \label{lem:params-for-sgn}
    If $A$ has $\epsilon$-pseudospectrum shattered with respect to a grid $\g = \grid(z_0,\omega,s_1,s_2)$ that includes the imaginary axis as a grid line, then one has $\Lambda_{\eps_0}(A) \subseteq \cpm{\alpha_0}$ where $\epsilon_0 = \epsilon/2$ and
    $$
        \alpha_0 = 1 - \frac{\epsilon}{\diag(\g)^2}.
    $$
    In particular, if $\eps$ is at least $1/\poly(n)$ and $\omega s_1$ and $\omega s_2$ are at most $\poly(n$), then $\eps_0$ and $1-\alpha_0$ are also at least $1/\poly(n)$.
\end{lemma}

\begin{proof}
    First, because it is shattered, the $\epsilon/2$-pseudospectrum of $A$ is at least distance $\epsilon/2$ from $\g$. Recycling the calculation from Proposition \ref{prop:newtonexactarithmetic}, it suffices to take
    $$
        \alpha_0^2 = \max_{z \in \Lambda_{\epsilon/2}(A)}\left(1 - \frac{4|\Re z|}{|z - \sgn(z)|^2}\right).
    $$
    From what we just observed about the pseudospectrum, we can take $|\Re z| \ge \epsilon/2$. To bound the denominator, we can use the crude bound that any two points inside the grid are at distance no more than $\diag(\g)$. Finally, we use $\sqrt{1 - x} \le 1 - x/2$ for any $x\in (0,1)$.
\end{proof}


The proof of Theorem \ref{prop:sgnerr} will proceed as in the exact arithmetic case, with the modification that $\eps_k$ must be decreased by an additional factor after each iteration to account for roundoff. At each step, we set the machine precision $\mach$ small enough so that the $\eps_k$ remain close to what they would be in exact arithmetic. For the analysis we will introduce an explicit auxiliary sequence $e_k$ that lower bounds the $\eps_k$, provided that $\mach$ is small enough.

\begin{lemma}[One-step additive error]
\label{lem:formalizenoise}
Assume the matrix inverse is computed by an algorithm $\INV$ satisfying the guarantee in Definition \ref{def:inv}.  Then $\G(A) = g(A) + E$ for some error matrix $E$ with norm \begin{equation}\label{eqn:Ebound}
        \Vert E \Vert \le \left(\Vert A \Vert + \Vert A^{-1} \Vert  + \pinv(n) \kappa(A)^{\cinv \log n}\|A^{-1}\| \right) 2 \sqrt{n} \mach.\end{equation}
\end{lemma}

The proof of this lemma is deferred to Appendix \ref{sec:deferredsign}. 

With the error bound for each step in hand, we now move to the analysis of the whole iteration.  It will be convenient to define $s := 1 - \alpha_0$, which should be thought of as a small parameter. \marginnote{$s$}  As in the exact arithmetic case, for $k \ge 1,$ we will recursively define decreasing sequences $\alpha_k$ and $\eps_k$ maintaining the property 
\begin{equation} \label{eqn:pscontained}
    \Lambda_{\eps_k}(\widetilde{A}_k) \subset 
    \cpm{\alpha_k} \qquad \text{for all $k \ge 0$}
\end{equation} by induction as follows:  
\begin{enumerate}
\item The base case $k=0$ holds because by assumption, $\Lambda_{\eps_0} \subset \cpm{\alpha_0}$.
    \item Here we recursively define $\alpha_{k+1}$. Set \marginnote{$\alpha_{k}$}
    \[\alpha_{k+1} := (1 + s/4) \alpha_k^2.\]   
    
    In the notation of Subsection \ref{sec:exactarithnewton}, this corresponds to setting $D = 1+s/4$.  This definition ensures that
    $\alpha_k^2 \le \alpha_{k+1} \le \alpha_k$ for all $k$, and also gives us the bound $(1+s/4)\alpha_0 \le 1-s/2$.  We also have the closed form
    \[ \alpha_k = (1+s/4)^{2^k - 1} \alpha_0^{2^k},\]
    which implies the useful bound 
    \begin{equation}
        \alpha_k \le (1-s/2)^{2^k}. \end{equation}
    
    \item Here we recursively define $\eps_{k+1}$. Combining Lemma \ref{thm:iteratedpseudospec}, the recursive definition of $\alpha_{k+1}$, and the fact that $1 - \alpha_k^2 \ge 1 - \alpha_0^2 \ge 1 - \alpha_0 = s$, we find that $\Lambda_{\epsilon'}\left(g(\widetilde{ A}_k)\right) \subset \cpm{\alpha_{k+1}}$, where 
    \[ 
        \eps'
        = \eps_k \frac{\left(\alpha_{k+1} - \alpha_k^2\right)(1-\alpha_k^2)}{8\alpha_k}
        = \eps_k \frac{s\alpha_k(1-\alpha_k^2)}{32} 
        \ge \eps_k\frac{ \alpha_k s^2}{32}.
    \]
    
    Thus in particular 
    \[
        \Lambda_{\eps_k\alpha_k s^2/32} \left(g(\widetilde{A}_k)\right) \subset \cpm{\alpha_{k+1}}.
    \]
    Since $\widetilde{A}_{k+1} = \G(\widetilde{A}_k) = g(\widetilde{A}_k) + E_k$, \marginnote{$E_k$} for some error matrix $E_k$ arising from roundoff, Proposition~\ref{prop:decrementeps} ensures that if we set 
    \begin{equation} \label{eqn:epsdecrease}
        \eps_{k+1} := \eps_k\frac{ s^2 \alpha_k }{32} - \Vert E_k \Vert
        \marginnote{$\eps_k$}
    \end{equation}\
    we will have $\Lambda_{\eps_{k+1}}(\widetilde{A}_{k+1}) \subset 
    \cpm{\alpha_{k+1}}, $ as desired.
\end{enumerate}

We now need to show that the $\eps_{k}$ do not decrease too fast as $k$ increases. In view of  (\ref{eqn:epsdecrease}), it will be helpful to set the machine precision small enough to guarantee that $\Vert E_k \Vert$ is a small fraction of $\eps_k\frac{ \alpha_k s^2}{32}$.

First, we need to control the quantities $\|\widetilde{A}_k\|$, $\|\widetilde{A}_k^{-1}\|$, and $\kappa(\widetilde{A}_k) =\|\widetilde{A}_k\|\|\widetilde{A}_k^{-1}\|$ appearing in our upper bound (\ref{eqn:Ebound}) on $\Vert E_k \Vert$ from Lemma \ref{lem:formalizenoise}, as functions of $\epsilon_k$. By Corollary \ref{cor:condbounds}, we have 
\begin{align*}
        \|\widetilde{A}_k^{-1}\| &\le \frac{1}{\epsilon_k} \quad \text{and} \quad
        \|\widetilde{A}_k\| \le 4\frac{\alpha_k}{(1 - \alpha_k)^2\eps_k} \le \frac{4}{s^2 \eps_k}.
\end{align*}
Thus, 
we may write the coefficient of $\mach$ in the bound (\ref{eqn:Ebound}) as 
\[
    K_{\eps_k} :=  \left[ \frac{4}{s^2\eps_k} + \frac{1}{\eps_k}  + \pinv(n) \left( \frac{4}{s^2\eps_k^2} \right)^{\cinv \log n}\frac{1}{\eps_k} \right] 2 \sqrt{n} 
    \marginnote{$K_{\eps_k}$}
\] 
so that Lemma \ref{lem:formalizenoise} reads
\begin{equation}
    \label{eqn:nkbound}
    \Vert{E_k}\Vert \le  K_{\eps_k} \mach.
\end{equation}
Plugging this into the definition (\ref{eqn:epsdecrease}) of $\eps_{k+1}$,
we have 
\begin{equation}
\label{eqn:epskbound}
\eps_{k+1} \ge \eps_k  \frac{s^2\alpha_k}{32} - K_{\eps_k} \mach.
\end{equation} 


  Now suppose we take $\mach$ small enough so that 
\begin{equation}
    \label{eqn:setmachk}
      K_{\eps_k} \mach 
     \le  \frac{1}{3} \eps_k \frac{s^2\alpha_k}{32}.
\end{equation}
  For such $\mach$, we then have
\begin{equation}\label{eqn:epsstep} \eps_{k+1} \ge  \frac{2}{3}\eps_k \frac{s^2\alpha_k}{32} = \frac{1}{48} \epsilon_k s^2 \alpha_k, 
\end{equation}
which implies
\begin{equation} \label{eqn:nkeps}
    \Vert E_k \Vert \le \frac{1}{2} \eps_{k+1};
\end{equation}  
this bound is loose but sufficient for our purposes. Inductively, we now have the following bound on $\eps_k$ in terms of $\alpha_k$:

\begin{lemma}[Preliminary lower bound on $\eps_k$] \label{lem:epskbound}
Let $k \ge 0$, and for all $0 \le i \le k-1$, assume $\mach$ satisfies the requirement (\ref{eqn:setmachk}):
\[K_{\eps_i}\mach \le \frac{1}{3} \eps_i \frac{s^2 \alpha_i}{32}.\]
Then we have 
\[ 
    \eps_k \ge e_k := \eps_0 \left(\frac{s^2}{50}\right)^k \alpha_k.
    \marginnote{$e_k$}
\]
In fact, it suffices to assume the hypothesis only for $i=k-1$.
\end{lemma}
\begin{proof}
The last statement follows from the fact that $\eps_i$ is decreasing in $i$ and $K_{\eps_i}$ is increasing in $i$.

Since (\ref{eqn:setmachk}) implies (\ref{eqn:epsstep}), we may apply (\ref{eqn:epsstep}) repeatedly to obtain
\begin{align*}
    \eps_k &\ge \eps_0 (s^2/48)^k \prod_{i=0}^{k-1} \alpha_i & \\
         &= \eps_0 (s^2/48)^k  (1+s/4)^{2^k - 1 -  k}\alpha_0^{2^k-1}  &{\text{by the definition of $\alpha_i$}}\\
        &= \eps_0 \left(\frac{s^2}{48(1+s/4)}\right)^k \frac{\alpha_k}{\alpha_0} &\\
        &\ge \eps_0 \left(\frac{s^2}{50}\right)^k \alpha_k. &\alpha_0 \le 1, s < 1/8
    \end{align*} 
\end{proof}
We now show that the conclusion of Lemma \ref{lem:epskbound} still holds if we replace $\eps_i$ everywhere in the hypothesis by $e_i$, which is an explicit function of $\eps_0$ and $\alpha_0$ defined in Lemma \ref{lem:epskbound}.  Note that we do not know $\eps_i \ge e_i$ a priori, so to avoid circularity we must use a short inductive argument.
\begin{corollary}[Lower bound on $\eps_k$ with explicit hypothesis] \label{cor:bootstrapping}
Let $k \ge 0$, and for all $0 \le i \le k-1$, assume $\mach$ satisfies 
\begin{equation} 
    K_{e_i} \mach \le \frac{1}{3} e_i \frac{s^2 \alpha_i}{32}
\end{equation}
where $e_i$ is defined in Lemma \ref{lem:epskbound}.  Then we have 
\[ \eps_k \ge e_k.\]
In fact, it suffices to assume the hypothesis only for $i=k-1$.
\end{corollary}
\begin{proof}
The last statement follows from the fact that $e_i$ is decreasing in $i$ and $K_{e_i}$ is increasing in $i$.

Assuming the full hypothesis of this lemma, we prove $\eps_i \ge e_i$ for $0 \le i \le k$ by induction on $i$.  For the base case, we have $\eps_0 \ge e_0 = \eps_0 \alpha_0$.

For the inductive step, assume $\eps_i \ge e_i$. Then as long as $i \le k-1$, the hypothesis of this lemma implies
\[K_{\eps_i} \mach \le \frac{1}{3} \eps_i \frac{s^2 \alpha_i}{32}, \ \] so we may apply Lemma \ref{lem:epskbound} to obtain $\eps_{i+1} \ge e_{i+1}$, as desired.
\end{proof}

\begin{lemma}[Main accuracy bound]
\label{lem:projdiff}
Suppose $\mach$ satisfies the requirement (\ref{eqn:setmachk}) for all $0 \le k \le N$.  
Then
\begin{equation}
    \label{eqn:projdiff} 
    \Vert \widetilde{A}_N - \sgn(A) \Vert \le \frac{8}{s} \sum_{k=0}^{N-1} \frac{\Vert E_k \Vert}{\eps_{k+1}^2} + \frac{8 \cdot 50^N }{s^{2N+2}\eps_0}(1 - s/2)^{2^N}.
\end{equation}
\end{lemma}

\begin{proof}
Since $\sgn = \sgn\circ \newton$, for every $k$ we have
\begin{align*}
    \Vert \sgn(\widetilde{A_{k+1}}) - \sgn(\widetilde{A_k})\Vert &= \Vert \sgn(\widetilde{A_{k+1}}) - \sgn(g(\widetilde{A_k})) \Vert 
    = \Vert \sgn(\widetilde{A_{k+1}}) - \sgn(\widetilde{A_{k+1}} - E_k) \Vert . 
\end{align*}
From the holomorphic functional calculus we can rewrite $\Vert \sgn(\widetilde{A_{k+1}}) - \sgn(\widetilde{A_{k+1}} - E_k) \Vert$ as the norm of a certain contour integral, which in turn can be bounded as follows:
\begin{align*}
   & \frac{1}{2\pi}\left\Vert \oint_{\partial\cp{\alpha_{k+1}}} [(z-\widetilde{A_{k+1}})^{-1} - (z - (\widetilde{A_{k+1}} - E_k))^{-1} ]\, dz -\oint_{\partial\cm{\alpha_{k+1}}} [(z-\widetilde{A_{k+1}})^{-1} - (z - (\widetilde{A_{k+1}} - E_k))^{-1} ]\, dz  \right\Vert\\
    = & \frac{1}{2\pi}\left\Vert \oint_{\partial\cp{\alpha_{k+1}}} [(z - (\widetilde{A_{k+1}} - E_k))^{-1}E_k(z-\widetilde{A_{k+1}})^{-1}  ]\, dz - \oint_{\partial\cm{\alpha_{k+1}}} [(z - (\widetilde{A_{k+1}} - E_k))^{-1}E_k(z-\widetilde{A_{k+1}})^{-1}  ]\, dz  \right\Vert\\
    \le & \frac{1}{\pi}  \oint_{\partial\cp{\alpha_{k+1}}} \Vert (z - (\widetilde{A_{k+1}} - E_k))^{-1}\Vert \Vert E_k \Vert \Vert (z - \widetilde{A_{k+1}})^{-1} \Vert\,dz\\
    \le & \frac{1}{\pi}\ell(\partial \cp{\alpha_{k+1}}) \Vert E_k \Vert \frac{1}{\eps_{k+1} - \Vert E_k \Vert}\frac{1}{\eps_{k+1}} \\
    = & \frac{4\alpha_{k+1}}{1 - \alpha_{k+1}^2} \Vert E_k \Vert \frac{1}{\eps_{k+1} - \Vert E_k \Vert}\frac{1}{\eps_{k+1}},
\end{align*}
where we use the definition (\ref{eqn:pseudodef2}) of pseudospectrum and Proposition \ref{prop:decrementeps}, together with the property (\ref{eqn:pscontained}). Ultimately, this chain of inequalities implies
$$\Vert \sgn(\widetilde{A_{k+1}}) - \sgn(\widetilde{A_{k}}) \Vert\leq \frac{4\alpha_{k+1}}{1 - \alpha_{k+1}^2} \Vert E_k \Vert \frac{1}{\eps_{k+1} - \Vert E_k \Vert}\frac{1}{\eps_{k+1}}. $$
Summing over all $k$ and using the triangle inequality, we obtain
\begin{align*}
    \Vert \sgn(\widetilde{A_N}) - \sgn(\widetilde{A_0}) \Vert &\le \sum_{k=1}^{N-1}  \frac{4\alpha_{k+1}}{1 - \alpha_{k+1}^2} \Vert E_k \Vert \frac{1}{\eps_{k+1} - \Vert E_k \Vert}\frac{1}{\eps_{k+1}} \\
    &\le \frac{8}{s} \sum_{k=0}^{N-1} \frac{\Vert E_k \Vert}{\eps_{k+1}^2},
\end{align*}
where in the last step we use $\alpha_k \le 1$ and $1 - \alpha_{k+1}^2 \ge s$, as well as (\ref{eqn:nkeps}).

By Lemma \ref{lem:boundforsign} (to be precise, by repeating the proof of that lemma with $\widetilde{A_N}$ substituted for $A_N$), we have
\begin{align*}
     \|\widetilde{A_N} - \sgn(\widetilde{A_N})\| &\leq \frac{8  \alpha_N^2}{(1- \alpha_N)^2 (1+\alpha_N) \eps_N} \\
     &\leq \frac{8 }{s^2} \alpha_N \frac{\alpha_N}{\eps_N} \\
     &\le \frac{8}{s^2} \alpha_N\frac{1}{\eps_0}\left(\frac{50}{s^2}\right)^N \\
     &\le \frac{8}{s^2\eps_0} (1-s/2)^{2^N}\left(\frac{50}{s^2}\right)^N \\
     &\le \frac{8 \cdot 50^N }{s^{2N+2}\eps_0}(1 - s/2)^{2^N}.
\end{align*}
where we use $s < 1/2$ in the last step.

Combining the above with the triangle inequality, we obtain the desired bound.

\end{proof}
We would like to apply Lemma \ref{lem:projdiff} to ensure $\Vert \widetilde{A_N} - \sgn(A) \Vert$ is at most $\sgnerr$, the desired accuracy parameter.  
The upper bound \eqref{eqn:projdiff} in Lemma $\ref{lem:projdiff}$ is the sum of two terms; we will make each term less than $\sgnerr/2$.  The bound for the second term will yield a sufficient condition on the number of iterations $N$.  Given that, the bound on the first term will then give a sufficient condition on the machine precision $\mach$.  This will be the content of Lemmas \ref{lem:Nbound1} and \ref{lem:Nbound2}.

 We start with the second term.  The following preliminary lemma will be useful:
 \begin{lemma} \label{lem:prelimN}
 Let $1/800 > t > 0$ and $1/2 > c > 0$ be given.  Then for
 \[j \ge \lg(1/t) + 2 \lg \lg (1/t) + \lg \lg(1/c) + 1.62,\]
 we have
 \[ \frac{(1-t)^{2^j}}{t^{2j}} < c.\]
 \end{lemma}
The proof is deferred to Appendix \ref{sec:deferredsign}.

\begin{lemma}[Bound on second term of \eqref{eqn:projdiff}] \label{lem:Nbound1}Suppose we have 
\[ N \ge \lg(8/s) + 2 \lg \lg(8/s) + \lg \lg (16/(\sgnerr s^2 \eps_0)) + 1.62.\]
Then
\[  \frac{8 \cdot 50^N }{s^{2N+2}\eps_0}(1 - s/2)^{2^N} \le \sgnerr/2.\]
\end{lemma} 
\begin{proof}
It is sufficient that 
\[  
    \frac{8 \cdot 64^N }{s^{2N+2}\eps_0}(1 - s/8)^{2^N} \le \sgnerr/2.
\]
The result now follows from applying Lemma \ref{lem:prelimN} with $c = \sgnerr s^2 \eps_0/16$ and $t=s/8$. 

\end{proof}

Now we move to the first term in the bound of Lemma \ref{lem:projdiff}.  
\begin{lemma}[Bound on first term of \eqref{eqn:projdiff}]
\label{lem:Nbound2}

Suppose \[ N \ge \lg(8/s) + 2 \lg \lg(8/s) + \lg \lg (16/(\sgnerr s^2 \eps_0)) + 1.62,\]
and suppose the machine precision $\mach$ satisfies 
\[\mach \le  \frac{  (1-s)^{2^{N+1}(\cinv \log n + 3)}}{\pinv(n) \sqrt{n}  N} .\]
Then we have 
\[
    \frac{8}{s} \sum_{k=0}^{N-1} \frac{\Vert E_k \Vert}{\eps_{k+1}^2}  \le \sgnerr/2.
\]
\end{lemma}

\begin{proof}
It suffices to show that for all $0 \le k \le N-1$,
\[ 
    \Vert E_k \Vert \le \frac{ \sgnerr \eps_{k+1}^2 s}{16N}.
\]
In view of (\ref{eqn:nkbound}), which says $\Vert{E_k}\Vert \le  K_{\eps_k} \mach $, 
it is sufficient to have for all $0 \le k \le N-1$
\begin{align} \label{eqn:machsgnbound}     \mach 
    &\le \frac{1}{K_{\eps_k}}\frac{     \sgnerr \eps_{k+1}^2 s}{16N}.
\end{align}
For this, we claim it is sufficient to have for all $0 \le k \le N-1$ 
\begin{align} \label{eqn:esgnbound}     
    \mach &\le \frac{1}{K_{e_k}}\frac{ \sgnerr e_{k+1}^2 s}{16N}. 
\end{align}
Indeed, on  the one hand, since $\sgnerr < 1/6$ and by the loose bound $e_{k+1} < s \alpha_{k+1} < s \alpha_k$ we have that (\ref{eqn:esgnbound}) implies $\mach \leq \frac{1}{3K_{e_k}} \frac{ s^2 e_k}{32}$, which means that the assumption in Corollary \ref{cor:bootstrapping} is satisfied. On the other hand Corollary \ref{cor:bootstrapping} yields $e_k\leq \epsilon_k$ for all $0\leq k \leq N$, which in turn, combined with (\ref{eqn:esgnbound}) would give (\ref{eqn:machsgnbound}) and conclude the proof. 

We now show that (\ref{eqn:esgnbound}) holds for all $0\leq k\leq N-1$. Because $1/K_{e_k}$ and $e_k$ are decreasing in $k$, it is sufficient to have the single condition
\[  
    \mach \le \frac{1}{K_{e_N}}\frac{ \sgnerr e_{N}^2 s}{16N}.
\]
We continue the chain of sufficient conditions on $\mach$, where each line implies the line above:
\begin{align*} 
    \mach &\le \frac{1}{K_{e_N}}\frac{ \sgnerr e_N^2 s}{16N} \\
    \mach &\le \frac{1}{\left[ \frac{4}{s^2e_N} + \frac{1}{e_N}  + \pinv(n) \left( \frac{4}{s^2e_N^2} \right)^{\cinv \log n}\frac{1}{e_N} \right] 2 \sqrt{n}} \frac{ \sgnerr e_N^2 s}{16N} \\
    \mach &\le \frac{1}{6 \pinv(n) \left( \frac{4}{s^2 e_N} \right)^{\cinv \log n + 1} 2\sqrt{n}}\frac{ \sgnerr e_N^2 s}{16 N} \\
    \mach &\le \frac{ \sgnerr }{6 \cdot 2 \cdot 16 \pinv(n) \sqrt{n}  N} \left(\frac{e_N s^2}{4} \right)^{\cinv \log n + 3}.
\end{align*}
where we use the bound $\frac{1}{e_N} \le \frac{4}{s^2 e_N^2}$ without much loss, and we also use our assumption $\pinv(n) \ge 1$ and $\cinv \log n \ge 1$ for simplicity.

Substituting the value of $e_N$ as defined in Lemma \ref{lem:epskbound}, we get the sufficient condition 

\[\mach \le  \frac{ \sgnerr }{192 \pinv(n) \sqrt{n}  N} \left(\frac{\eps_0 (s^2/50)^N \alpha_N s^2}{4} \right)^{\cinv \log n + 3}. \]

Replacing $\alpha_N$ by the smaller quantity $\alpha_0^{2^N} = (1-s)^{2^N}$ and cleaning up the constants yields the sufficient condition 

\[ \mach \le  \frac{ \sgnerr }{192 \pinv(n) \sqrt{n}  N} \left(\frac{\eps_0 (s^2/50)^N (1-s)^{2^N} s^2}{4} \right)^{\cinv \log n + 3}.\]

Now we finally will use our hypothesis on the size of $N$ to simplify this expression. Applying Lemma \ref{lem:Nbound1}, we have 
\[ \eps_0 (s^2/50)^N /4 \ge \frac{4 (1-s)^{2^N}}{s^2 \sgnerr}.\]

Thus, our sufficient condition becomes
\[ \mach \le  \frac{ \sgnerr }{192 \pinv(n) \sqrt{n}  N} \left(\frac{4(1-s)^{2^{N+1}}}{\sgnerr} \right)^{\cinv \log n + 3}.\]

To make the expression simpler, since $\cinv \log n + 3 \ge 4$ we may pull out a factor of $4^4 > 192$ and remove the occurrences of $\sgnerr$ to yield the sufficient condition

\[\mach \le  \frac{  (1-s)^{2^{N+1}(\cinv \log n + 3)}}{\pinv(n) \sqrt{n}  N}.  \]

\end{proof}

Matching the statement of Theorem \ref{prop:sgnerr}, we give a slightly cleaner sufficient condition on $N$ that implies the hypothesis on $N$ appearing in the above lemmas. The proof is deferred to Appendix \ref{sec:deferredsign}.
\begin{lemma}[Final sufficient condition on $N$] \label{lem:cleanN}
If \[ N = \lceil \lg(1/s) + 3 \lg \lg(1/s) + \lg \lg (1/(\sgnerr \eps_0)) + 7.59 \rceil, \]
then
\[  N \ge \lg(8/s) + 2 \lg \lg(8/s) + \lg \lg (16/(\sgnerr s^2 \eps_0)) + 1.62.\]
\end{lemma}

Taking the logarithm of the machine precision yields the number of bits required:
\begin{lemma}[Bit length computation] \label{lem:sgnbitlength}
Suppose
\[ N = \lceil \lg(1/s) + 3 \lg \lg(1/s) + \lg \lg (1/(\sgnerr \eps_0)) + 7.59 \rceil \]
and 
\[
    \mach_{\SGN} =  \frac{  (1-s)^{2^{N+1}(\cinv \log n + 3)}}{\pinv(n) \sqrt{n}  N}.
\]
 Then
\[
    \lg(1/\mach_{\SGN}) = O\big(\log n \log(1/s)^3 (\log(1/\sgnerr) + \log(1/\eps_0))\big).
\]
\end{lemma}
\begin{proof}
In the course of the proof, for convenience we also record a nonasymptotic bound (for $s<1/100$, $\beta < 1/12$, $\eps_0 < 1$ and $\cinv \log n > 1$ as in the hypothesis of Theorem \ref{prop:sgnerr}), at the cost of making the computation somewhat messier.  

Immediately we have
\[ 
    \lg(1/\mach_{\SGN}) \le  \lg \pinv(n) + \frac{1}{2}\lg n + \lg N +  (\cinv \log n + 3) 2^{N+1} \log(1/(1-s)).
\]
Note that $\log(1/(1-s)) < s$ for $s < 1/2$.  Also, $2^{N+1} \le (1/s) \lg(1/s)^3 (\lg(1/\beta) + \lg(1/\eps_0))2^{9.59}.$  Putting this together, we have
\[ 
    \lg(1/\mach_{\SGN}) \le \lg \pinv(n) + \frac{1}{2}\lg n + \lg N +  1000 (\cinv \log n + 3)\lg(1/s)^3 (\lg(1/\beta) + \lg(1/\eps_0)).
\]


We now crudely bound $\lg N$.  Note that for $s < 1/100$ we have $\lg(1/s) + 3 \lg \lg(1/s) + 7.59 \le 1/s$.  Thus, 
\begin{align*} 
\lg N &\le \lg(1/s + \lg \lg (1/(\beta \eps_0)))& \\
&\le \lg(1/s + \lg (1/(\beta \eps_0)))& \\
&\le \lg(1/s) + \lg \lg (1/(\beta \eps_0))& \lg(a+b) \le  \lg a + \lg b \text{ for } a,b>2\\
&\le \lg(1/s)^3 \lg (1/(\beta \eps_0)).&
\end{align*}

Combining the above, we may fold the $\lg N$ and $\lg n$ terms into the final term to obtain
\begin{equation} \label{eqn:nonasy-sgn}
    \lg(1/\mach_{\SGN}) \le \lg \pinv(n) + 5000\cinv \log n \lg(1/s)^3 (\lg(1/\sgnerr) + \lg(1/\eps_0)) 
\end{equation} 
where we use that  $\cinv \log n > 1$ and therefore $\cinv \log n + 3 < 4 \cinv \log n.$

Using  that $\mu_{\INV}(n) = \poly(n)$ and discarding subdominant terms, we obtain the desired asymptotic bound.


\end{proof}
This completes the proof of Theorem \ref{prop:sgnerr}.  Finally, we may prove the theorem advertised in Section \ref{sec:intro}.
\begin{proof}[Proof of Theorem \ref{thm:signintro}]
Set $\eps := \min\{ \frac{1}{K}, 1\}$.  Then $\Lambda_\eps(A)$ does not intersect the imaginary axis, and furthermore $\Lambda_\eps(A) \subseteq D(0, 2)$ because $\Vert A \Vert \le 1$.  Thus, we may apply Lemma \ref{lem:params-for-sgn} with $\diag(\g) = 4\sqrt{2}$ to obtain parameters $\alpha_0, \eps_0$ with the property that $\log(1/(1-\alpha_0))$ and $\log(1/\eps_0)$ are both $O(\log K)$.   Theorem \ref{prop:sgnerr} now yields the desired conclusion. 
\end{proof}

	\section{Spectral Bisection Algorithm}
\label{sec:spectralbisec}


In this section we will prove Theorem \ref{thm:bkwd}.  
As discussed in Section \ref{sec:intro}, our algorithm is not new, and in its idealized form it reduces to the two following tasks:
\begin{enumerate}[leftmargin=3\parindent]
    \item[\textit{Split:}] Given an $n\times n$ matrix $A$, find a partition of the spectrum into pieces of roughly equal size, and output spectral projectors $P_{\pm}$ onto each of these pieces.
    \item[\textit{Deflate:}] Given an $n\times n$ rank-$k$ projector $P$, output an $n\times k$ matrix $Q$ with orthogonal columns that span the range of $P$.
\end{enumerate}
These routines in hand, on input $A$ one can compute $P_{\pm}$ and the corresponding $Q_{\pm}$, and then find the eigenvectors and eigenvalues of $A_{\pm} := Q_{\pm}^\ast A Q_{\pm}$. The observation below verifies that this recursion is sound.

\begin{observation}
    The spectrum of $A$ is exactly $\Lambda(A_+) \sqcup \Lambda(A_-)$, and every eigenvector of $A$ is of the form $Q_{\pm}v$ for some eigenvector $v$ of one of $A_{\pm}$.
\end{observation}

The difficulty, of course, is that neither of these routines can be executed exactly: we will never have access to true projectors $P_{\pm}$, nor to the actual orthogonal matrices $Q_{\pm}$ whose columns span their range, and must instead make do with approximations. Because our algorithm is recursive and our matrices nonnormal, we must take care that the errors in the sub-instances $A_{\pm}$ do not corrupt the eigenvectors and eigenvalues we are hoping to find. Additionally, the Newton iteration we will use to split the spectrum behaves poorly when an eigenvalue is close to the imaginary axis, and it is not clear how to find a splitting which is balanced.

Our tactic in resolving these issues will be to pass to our algorithms a matrix \textit{and} a grid with respect to which its $\epsilon$-pseudospectrum is shattered. To find an approximate eigenvalue, then, one can settle for locating the grid square it lies in; containment in a grid square is robust to perturbations of size smaller than $\epsilon$. The shattering property is robust to small perturbations, inherited by the subproblems we pass to, and---because the spectrum is quantifiably far from the grid lines---allows us to run the Newton iteration in the first place.

Let us now sketch the implementations and state carefully the guarantees for $\SPLIT$ and $\SPAN$; the analysis of these will be deferred to Appendices \ref{sec:split} and \ref{sec:deflate}. Our splitting algorithm is presented a matrix $A$ whose $\epsilon$-pseudospectrum is shattered with respect to a grid $\g$. For any vertical grid line with real part $h$, $\Tr\, \sgn(A-h)$ gives the difference between the number of eigenvalues lying to its left and right. As
$$
    |\Tr \,\SGN(A-h) - \Tr\, \sgn(A-h)| \le n\| \SGN(A-h) - \sgn(A-h)\|,
$$
we can determine these eigenvalue counts \textit{exactly} by running $\SGN$ to accuracy $O(1/n)$ and rounding $\Tr\, \SGN(A-h)$ to the nearest integer. We will show in Appendix \ref{sec:split} that, by mounting a binary search over horizontal and vertical lines of $\g$, we will always arrive at a partition of the eigenvalues into two parts with size at least $\min\{n/5,1\}$. Having found it, we run $\SGN$ one final time at the desired precision to find the approximate spectral projectors.


\begin{figure}[ht]
    \begin{boxedminipage}{\textwidth}
        $$ \SPLIT $$ {\small
        \textbf{Input:} Matrix $A \in \C^{n\times n}$, pseudospectral parameter $\epsilon$, grid $\g = \grid(z_0,\omega,s_1,s_2)$, and desired accuracy $\sgnerr$ \\
        \textbf{Requires:} $\Lambda_\epsilon(A)$ is shattered with respect to $\g$, and $\sgnerr \le 0.05/n$ \\
        \textbf{Algorithm:} $(\widetilde{P_{\pm}},\g_{\pm}, n_{\pm}) = \SPLIT(A,\epsilon,\g,\sgnerr)$
        \begin{enumerate}
            \item Execute a binary search over horizontal grid shifts $h$ until 
            $$
                \Tr\, \SGN \left(A - h,\epsilon/4,1 - \frac{\epsilon}{2\diag(\g)^2},\sgnerr\right) \le 3n/5.
            $$
            \item If this fails, set $A \gets iA$ and repeat with vertical grid shifts 
            \item Once a shift is found,
            $$
                \widetilde{P_{\pm}} \gets \tfrac{1}{2}\left(\SGN\left(A-h,\epsilon/4,1 - \frac{\epsilon}{2\diag(\g)^2},\sgnerr\right) \pm I\right),
            $$
            and $\g_{\pm}$ are set to the two subgrids
        \end{enumerate}
        \textbf{Output:} Two matrices $\widetilde{P_{\pm}} \in \C^{n\times n}$, two subgrids $\g_{\pm}$, and two numbers $n_{\pm}$  \\
        \textbf{Ensures:} Each subgrid $\g_{\pm}$ contains $n_\pm$ eigenvalues of $A$, $n_\pm \geq n/5$, and $\|\widetilde{P_{\pm}} - P_{\pm}\| \le \sgnerr$, where $P_{\pm}$ are the true spectral projectors for the eigenvalues in the subgrids $\g_{\pm}$ respectively.}
    \end{boxedminipage}
\end{figure}

\begin{theorem}[Guarantees for $\SPLIT$] \label{thm:split-guarantee}
    Assume $\INV$ is a $(\mu_\INV,\cinv)$-stable matrix inversion algorithm satisfying Definition \ref{def:inv}. Let $\epsilon \le 0.5$, $\sgnerr \le 0.05/n$, and $\|A\| \le 4$ and $\g$ have side lengths of at most $8$, and define 
    $$
        N_{\SPLIT} := \lg \frac{256}{\epsilon} + 3\lg\lg \frac{256}{\epsilon} + \lg\lg \frac{4}{\sgnerr \epsilon} + 7.59.
        \marginnote{$N_\SPLIT$}
    $$
    Then $\SPLIT$ has the advertised guarantees when run on a floating point machine with precision 
    $$
       \mach \le \mach_\SPLIT:= \min\left\{\frac{\left(1 - \frac{\epsilon}{256}\right)^{2^{N_{\SPLIT}+1} (c_{\INV} \log n + 3)}}{\mu_{\INV}(n)\sqrt n N_{\SPLIT}},\, \frac{\epsilon}{100 n},\, \frac{\epsilon^2}{512} \right\},
       \marginnote{$\mach_\SPLIT$}
    $$
    Using at most 
    $$
       T_{\SPLIT}(n,\g,\epsilon,\sgnerr) \le 12 \lg \frac{1}{\omega(\g)}\cdot N_{\SPLIT} \cdot \left(T_{\INV}(n) + O(n^2) \right)
       \marginnote{$T_\SPLIT$}
    $$
    arithmetic operations. The number of bits required is
    $$
        \lg 1/\mach_{\SPLIT} = O\left(\log n \log^3 \frac{256}{\epsilon}\left(\log\frac{1}{\beta} + \log\frac{4}{\epsilon}\right)\right).
    $$
\end{theorem}

Deflation of the approximate projectors we obtain from $\SPLIT$ amounts to a standard rank-revealing QR factorization. This can be achieved deterministically in $O(n^3)$ time with the classic algorithm of Gu and Eisenstat \cite{gu1996efficient}, or probabilistically in matrix-multiplication time with a variant of the method of \cite{demmel2007fast}; we will use the latter.

\begin{figure}[ht]
    \begin{boxedminipage}{\textwidth}
        $$\SPAN$$ {\small
        \textbf{Input:} Matrix $\widetilde{P} \in \C^{n\times n}$, desired rank $k$, input precision $\sgnerr$, and desired accuracy $\spnerr$ \\
        \textbf{Requires:}  $\|\widetilde{P} - P\| \le \sgnerr \leq \frac{1}{4}$ for some rank-$k$ projector $P$. \\
        \textbf{Algorithm:} $\widetilde{Q} = \SPAN(P,k,\sgnerr,\spnerr)$
        \begin{enumerate}
            \item $H \gets n\times n$ Haar unitary $+E_1$
            \item $(U,R) \gets \QR(PH^\ast)$
            \item $\widetilde{Q} \gets$ first $k$ columns of $U$.
        \end{enumerate}
        \textbf{Output:} A tall matrix $\widetilde{Q} \in \C^{n\times k}$\\
        \textbf{Ensures:} There exists a matrix $Q \in \C^{n\times k}$ whose orthogonal columns span $\text{range}(P)$, such that  $\|\widetilde{Q} - Q\| \le \spnerr$, with probability at least $ 1 - \frac{(20n)^3\sqrt{\sgnerr}}{\spnerr^2}$.}
    \end{boxedminipage}
\end{figure}

\begin{theorem}[Guarantees for $\SPAN$] \label{thm:deflate-guarantee-usable}
    Assume $\MM$ and $\QR$ are matrix multiplication and QR factorization algorithms satisfying Definitions \ref{def:MM} and \ref{def:qr}. Then $\SPAN$ has the advertised guarantees when run on a machine with precision: 
    $$
       \mach \le \mach_\SPAN:= \min \left\{ \frac{\sgnerr}{ 4\|\widetilde{P}\| \max(\mu_{\QR}(n),\mu_{\MM}(n))}, \frac{\spnerr}{2\mu_{\QR}(n)}\right\}.
       \marginnote{$\mach_\SPAN$}
    $$
    The number of arithmetic operations is at most: 
    $$
        T_{\SPAN}(n) = n^2 T_\N+ 2T_\QR(n)+T_\MM(n).
        \marginnote{$T_\SPAN$}
    $$
\end{theorem}

\begin{remark}
The proof of the above theorem, which is deferred to Appendix \ref{sec:deflate}, closely follows and builds on the analysis of the randomized rank revealing factorization algorithm ($\RURV$)  introduced in \cite{demmel2007fast} and further studied in \cite{ballard2019generalized}. The parameters in the theorem are optimized for the particular application of finding a basis for a deflating subspace given an approximate spectral projector.

The main difference with the analysis in \cite{demmel2007fast} and \cite{ballard2019generalized} is that here, to make it applicable to complex matrices, we make use of Haar unitary random matrices instead of Haar orthogonal random matrices. In our analysis of the unitary case, we discovered a strikingly simple formula  (Corollary \ref{cor:densitysigmar}) for the density of the smallest singular value of an $r\times r$  sub-matrix of an $n\times n$ Haar unitary; this formula is leveraged to obtain guarantees that work for any $n$ and $r$, and not only for when $n-r \geq 30$, as was the case in \cite{ballard2019generalized}.  Finally, we explicitly account for finite arithmetic considerations in the Gaussian randomness used in the algorithm, where true Haar unitary matrices can never be produced. 
\end{remark}

We are ready now to state completely an algorithm $\EIG$ which accepts a shattered matrix and grid and outputs approximate eigenvectors and eigenvalues with a \textit{forward-error} guarantee. Aside from the a priori un-motivated parameter settings in lines 2 and 3---which we promise to justify in the analysis to come---$\EIG$ implements an approximate version of the split and deflate framework that began this section. 

\begin{figure}[ht]
\begin{boxedminipage}{\textwidth}
    $$ \EIG $$ {\small
    \noindent \textbf{Input:} Matrix $A\in\C^{m\times m}$, desired eigenvector accuracy $\eigerr$, grid $\g = \grid(z_0,\omega,s_1,s_2)$, pseudospectral guarantee $\epsilon$, acceptable failure probability $\eigprob$, and global instance size $n$ \\
    \noindent \textbf{Requires:} $\Lambda_\epsilon(A)$ is shattered with respect to $\g$, and $m\le n$.\\
    \noindent \textbf{Algorithm:} $\EIG(A,\eigerr,\g,\epsilon, \theta,n)$
    \begin{enumerate}
        \item If $A$ is $1\times 1$, $(\widetilde{V},\widetilde D) \gets (1, A)$
        \item $\spnerr \gets \frac{\delta\epsilon^2}{200}$
        \item $\sgnerr \gets \frac{\spnerr^4}{(20 n)^6}\frac{\eigprob^2}{4n^8}$
        \item $(\widetilde P_+, \widetilde P_-,\g_+,\g_-,n_+,n_-) \gets \SPLIT(A,\epsilon,\g,\sgnerr)$
        \item $\widetilde Q_{\pm} \gets \SPAN(\widetilde P_{\pm},n_{\pm},\sgnerr, \spnerr)$
        \item $\widetilde A_{\pm} \gets \widetilde Q_{\pm}^\ast \widetilde A \widetilde Q_{\pm} + E_{6,\pm}$
        \item $(\widetilde V_{\pm},\widetilde D_{\pm}) \gets \EIG(\widetilde A_{\pm},4\eigerr/5,\g_{\pm},4\epsilon/5,\eigprob,n)$.
        \item $\widetilde V \gets \begin{pmatrix} \widetilde{Q}_+ \widetilde{V}_+  & \widetilde{Q}_- \widetilde{V}_-  \end{pmatrix} + E_8$ 
        \item $\widetilde{V} \gets \text{normalize}(\widetilde{V}) + E_9$
        \item $ \widetilde{D} \gets \begin{pmatrix} \widetilde{D}_+ & \\ & \widetilde{D}_- \end{pmatrix}$
    \end{enumerate}
    \noindent \textbf{Output:} Eigenvectors and eigenvalues $(\widetilde{V},\widetilde{D})$ \\
    \noindent \textbf{Ensures:} With probability at least $1 - \eigprob$, each entry $\widetilde{\lambda_i} = \widetilde{D}_{i,i}$ lies in the same square as exactly one eigenvalue $\lambda_i \in \Lambda(A)$, and each column $\widetilde{v}_i$ of $\widetilde{V}$ has norm $1 \pm n\mach$, and satisfies $\|\widetilde{v}_i - v_i\| \le \eigerr$ for some exact unit right eigenvector $Av_i = \lambda_i v_i$. }
\end{boxedminipage}
\end{figure}

\begin{theorem}[$\EIG$: Finite Arithmetic Guarantee] \label{thm:eig-finite-guarantee}
    Assume  $\MM, \QR$, and $\INV$ are numerically stable algorithms for matrix multiplication, QR factorization, and inversion satisfying Definitions \ref{def:MM}, \ref{def:qr}, and \ref{def:inv}. Let $\delta < 1$, $A \in \C^{n\times n}$ have $\|A\| \le 3.5$ and, for some $\epsilon < 1/2$, have $\epsilon$-pseudospectrum shattered with respect to a grid $\g = \grid(z_0,\omega,s_1,s_2)$ with side lengths at most $8$ and $\omega \le 1$. Define 
    \begin{equation*}
        N_{\EIG} := \lg \frac{256 n}{\epsilon} + 3\lg\lg \frac{256 n}{\epsilon} + \lg\lg \frac{(5n)^{26}}{\theta^2\delta^4\epsilon^9} + 7.59.
        \marginnote{$N_\EIG$}
    \end{equation*}
    Then $\EIG$ has the advertised guarantees when run on a floating point machine with precision satisfying:
    \begin{align*}
        \lg 1/\mach 
        &\ge \max\left\{\lg^3 \frac{ n}{\epsilon}\lg \left( \frac{(5n)^{26}}{\eigprob^2\eigerr^4\epsilon^8} \right)2^{9.59}(c_{\INV}\log n + 3) + \lg N_{\EIG}, \lg\frac{(5n)^{30}}{\eigprob^2\eigerr^4\epsilon^8} + \lg \max\{\mu_{\MM}(n),\mu_{\QR}(n),n\} \right\} \\
        &= O\left(\log^3 \frac{n}{\epsilon}\log \frac{n}{\theta\delta\epsilon}\log n\right).
    \end{align*}
    The number of arithmetic operations is at most
    \begin{align*}
        T_{\EIG}(n,\eigerr,\g,\epsilon,\theta, n) 
        &= 60 N_{\EIG}\lg\frac{1}{\omega(\g)}\left( T_{\INV}(n) + O(n^2)\right) + 10 T_{\QR}(n) + 25 T_{\MM}(n) \\
        &= O\left(\log \frac{1}{\omega(\g)}\left(\log \frac{n}{\epsilon} + \log\log \frac{1}{\theta\delta}\right) T_{\MM}(n) \right).
    \end{align*}
\end{theorem}

\begin{remark}
    We have not fully optimized the large constant $2^{9.59}$ appearing in the bit length above.
\end{remark}

Theorem \ref{thm:eig-finite-guarantee} easily implies Theorem \ref{thm:bkwd} when combined with $\SHATTER$.
\begin{theorem}[Restatement of Theorem \ref{thm:bkwd}] 
 There is a randomized algorithm $\EIG$ which on input any matrix $A\in \C^{n\times n}$ with $\|A\|\le 1$ and a desired accuracy parameter $\delta \in (0,1)$ outputs a diagonal $D$ and invertible $V$ such that 
$$ \|A-VDV^{-1}\|\le \delta \quad\mathrm{and}\quad \kappa(V) \le 32n^{2.5}/\delta$$
in 
$$O\left(T_\MM(n)\log^2\frac{n}{\delta}\right)$$
arithmetic operations on a floating point machine with $$O\left(\log^4\frac{n}{\delta}\log n\right)$$ bits of precision, with probability at least $1-14/n$. Here $T_\MM(n)$ refers to the running time of a numerically stable matrix multiplication algorithm (detailed in Section \ref{sec:logstable}). 
\end{theorem}
\begin{proof}
Given $A$ and $\delta$, consider the following two step algorithm:
    \begin{enumerate}
        \item $(X, \g, \epsilon)\gets \SHATTER(A,\delta/8)$.
        \item $(V,D)\gets \EIG(X,\delta',\g,\epsilon,1/n,n)$, where
        \begin{equation}\label{eqn:eigdeltasetting}
          \delta' := \frac{\delta^3}{n^{4.5} \cdot 6 \cdot 128 \cdot 2}.
        \end{equation}
    \end{enumerate}
With probability at least $1 - 13/n$, $\SHATTER(A,\delta/8)$ succeeds, in which case the output $(X,\grid,\epsilon)$ output easily satisfy the assumptions in Theorem \ref{thm:eig-finite-guarantee}: $\delta' \le \delta < 1$, $\epsilon = \tfrac{(\delta/8)^5}{32 n^9} \le 1/2$, $\g$ is defined by $\SHATTER$ to have side length $8$, $\|X\| \le \|A\| + \|X - A\| \le 1 + 4(\delta/8) \le 3.5$, and $X$ has $\epsilon$-pseudospectrum shattered with respect to $\g$. On this event, $X = WCW^{-1}$, and (using the proof of Theorem \ref{thm:polygap}) if we normalize $W$ to have unit length columns, then $\kappa(W) = \|W\|\|W^{-1}\| \le 8n^2/\delta$.

We will show that the choice of $\delta'$ in \eqref{eqn:eigdeltasetting} guarantees 
$$\|X-VDV^{-1}\|\le \delta/2.$$
Since $\|X\| \le \|A\| + \|A - X\| \le 1 + 4\gamma \le 3$ from Theorem \ref{thm:finiteshatter}, the hypotheses of Theorem \ref{thm:eig-finite-guarantee} are satisfied. Thus $\EIG$ succeeds with probability at least $1-1/n$, and by a union bound, both $\EIG$ and $\SHATTER$ succeed with probabiility at least $1 - 14/n$. On this event, we have $V=W+E$ for some $\|E\|\le \delta'\sqrt{n}$, so
$$\|V-W\|\le \delta'\sqrt{n},$$
as well as
$$
    \sigma_n(V)\ge \sigma_n(W)-\|E\|\ge \frac{\delta}{8n^2}-\delta'\sqrt{n}\ge \frac{\delta}{16n^2},
$$
since our choice of $\delta'$ satisfies the much cruder bound of
\begin{equation*}\label{eqn:deltaprime}
    \delta'\le \frac{\delta}{16n^{2.5}},
\end{equation*}
This implies that
$$
    \kappa(V)=\|V\|\|V^{-1}\|\le 2\sqrt{n}\cdot \frac{16n^2}{\delta},
$$
establishing the last item of the theorem.We can control the perturbation of the inverse as:
 \begin{align*}
     \Vert V^{-1} - W^{-1} \Vert &= \Vert W^{-1} (W - V) V^{-1} \Vert\\
     &\le \kappa(W)\|W - V\| \|V^{-1}\| \\
     &\le \frac{8n^2}{\delta}\cdot \delta'\sqrt{n} \cdot \frac{16 n^2}{\delta}\\
     &\le \frac{128 n^{4.5}\delta'}{\delta^2}.
 \end{align*}

The grid output by $\SHATTER(A,\delta/8)$ has $\omega = \tfrac{\delta^4}{4*8^4*n^5} \le \tfrac{\delta}{\sqrt{2}}$ provided $\delta < 1$. Thus the guarantees on $\EIG$ in Theorem \ref{thm:eig-finite-guarantee} tell us each eigenvalue of $X = WCW^{-1}$ shares a grid square with exactly one diagonal entry of $D$, which means that $\|C - D\| \le \sqrt{2}\omega \le \delta$. So, we have:
\begin{align*}
    \|VDV^{-1}-WCW^{-1}\| &\le \| (V-W) D V^{-1} \| + \| W (D-C) V^{-1} \| + \| WC (V^{-1} - W^{-1}) \| \\
    &\le \delta' \sqrt{n} \cdot 5 \cdot \frac{16n^2}{\delta} + \sqrt{n} \delta' \frac{16 n^2}{\delta} + \sqrt{n} \cdot 5 \cdot \frac{128n^{4.5} \delta'}{\delta^2} \\
    &= \frac{\delta' n^{4.5}}{\delta} \left( 5 \cdot 16 + 16 + \frac{5 \cdot 128}{\delta} \right) \\
    &\le \frac{\delta' n^{4.5}}{\delta^2} \cdot 6\cdot 128
\end{align*}
    which is at most $\delta/2$, for $\delta'$ chosen as above. We conclude that
$$ \|A-VDV^{-1}\|\le \|A-X\|+\|X-VDV^{-1}\|\le \delta,$$
with probability $1-14/n$ as desired.

To compute the running time and precision, we observe that $\SHATTER$ outputs a grid with parameters
$$ \omega = \Omega\left(\frac{\delta^4}{n^5}\right),\quad \epsilon =\Omega \left(\frac{\delta^5}{n^9}\right).$$
Plugging this into the guarantees of $\EIG$, we see that it takes
$$O\left(\log \frac{n}{\delta}\left(\log \frac{n}{\delta} + \log\log \frac{n}{\delta}\right) T_{\MM}(n) \right) = O(T_\MM(n)\log^2(n/\delta))$$
arithmetic operations, on a floating point machine with precision
$$O\left(\log^3 \frac{n}{\delta}\log \frac{n}{\delta}\log n\right) = O(\log^4(n/\delta)\log(n))$$
bits, as advertised.
\end{proof}
\subsection{Proof of Theorem \ref{thm:eig-finite-guarantee}}\label{sec:eigproof}

A key stepping-stone in our proof will be the following elementary result controlling the spectrum, pseudospectrum, and eigenvectors after perturbing a shattered matrix.

\begin{lemma}[Eigenvector Perturbation for a Shattered Matrix] \label{lem:perturb-shattered}
    Let $\Lambda_{\epsilon}(A)$ be shattered with respect to a grid whose squares have side length $\omega$, and assume that $\|\widetilde A - A\| \le \eta < \epsilon$. Then, (i) each eigenvalue of $\widetilde A$ lies in the same grid square as exactly one eigenvalue of $A$, (ii) $\Lambda_{\epsilon - \eta}(\widetilde A)$ is shattered with respect to the same grid, and (iii) for any right unit eigenvector $\widetilde v$ of $\widetilde{A}$, there exists a right unit eigenvector of $A$ corresponding to the same grid square, and for which
    $$
        \|\widetilde v - v\| \le \frac{\sqrt{8}\omega}{\pi}\frac{\eta}{\epsilon(\epsilon - \eta)}.
    $$
\end{lemma}

\begin{proof}
    For (i), consider $A_t = A + t(\widetilde A - A)$ for $t \in [0,1]$. By continuity, the entire trajectory of each eigenvalue is contained in a unique connected component of $\Lambda_\eta(A) \subset \Lambda_\epsilon(A)$. For (ii), $\Lambda_{\epsilon - \eta}(\widetilde A) \subset \Lambda_{\epsilon}(A)$, which is shattered by hypothesis. Finally, for (iii), let $w^\ast$ and $\widetilde{w}^\ast$ be the corresponding left eigenvectors to $v$ and $\widetilde{v}$ respectively, normalized so that $w^\ast v = \widetilde{w}^\ast \widetilde{v} = 1$. Let $\Gamma$ be the boundary of the grid square containing the eigenvalues associated to $v$ and $\widetilde{v}$ respectively. Then, using a contour integral along $\Gamma$ as in (\ref{eq:projectorstability}) above, one gets
    \begin{align*}
        \| \widetilde{v}\widetilde{w}^\ast - vw^\ast\| \le \frac{2\omega}{\pi}\frac{\eta}{\epsilon(\epsilon - \eta)}.
    \end{align*}
    $$
    $$
    Thus, using that $\|v\|=1$ and $w^\ast v  = 1$,
    
    $$
        \|\widetilde{v}\widetilde{w}^*-vw^*\|\geq  \|(\widetilde{v}\widetilde{w}^*-vw^*) v\| = \|(\widetilde{w}^* v) \widetilde{v}-v\|.
    $$
    Now, since $(\widetilde{v}^*v) \widetilde{v}$ is the orthogonal projection of $v$ onto the span of $\widetilde{v}$, we have that
    $$
        \|(\widetilde{w}^*v)\widetilde{v}- v\| \geq \|(\widetilde{v}^*v) \widetilde{v}- v\| = \sqrt{1-|\widetilde{v}^*v|^2}.
    $$
    Multiplying $v$ by a phase we can assume without loss of generality that $\widetilde{v}^* v\geq 0$ which implies that 
    $$\sqrt{1-(\widetilde{v}^*v)^2} = \sqrt{(1-\widetilde{v}^*v)(1+\widetilde{v}^*v)} \geq \sqrt{1-\widetilde{v}^*v}.$$
    The above discussion can now be summarized in the following chain of inequalities
   $$\sqrt{1-\widetilde{v}^*v}\leq \sqrt{1-(\widetilde{v}^*v)^2}\leq \|(\widetilde{w}^*v)\widetilde{v}- v\| \leq \|\widetilde{v}\widetilde{w}^*-vw^*\| \leq \frac{2\omega}{\pi}\frac{\eta}{\epsilon(\epsilon - \eta)}.$$
   Finally, note that $\|v-\widetilde{v}\| = \sqrt{2-2\widetilde{v}^*v} \leq \frac{\sqrt{8}\omega}{\pi} \frac{\eta}{\epsilon(\epsilon- \eta)}$
   as we wanted to show. 
\end{proof}

The algorithm $\EIG$ works by recursively reducing to subinstances of smaller size, but requires a pseudospectral guarantee to ensure speed and stability. We thus need to verify that the pseudospectrum does not deteriorate too subtantially when we pass to a sub-problem.

\begin{lemma}[Shattering is preserved after compression] \label{lem:compress-shattered} 
    Suppose $P$ is a spectral projector of $A\in\C^{n\times n}$ of rank $k$. Let $Q\in \C^{n\times k}$ be such that $Q^*Q=I_k$ and that its columns span the same space as the columns of $P$. Then for every $\epsilon>0$,
    $$\Lambda_\epsilon(Q^*AQ)\subset \Lambda_\epsilon(A).$$
    Alternatively, the same pseudospectral inclusion holds if again $Q^*Q=I_k$ and, instead, the columns of $Q$ span the same space as the rows of $P$. 
\end{lemma}
\begin{proof} We will first analyze the case when the columns of $Q$ span the same space as the columns of $P$.  To begin, note that if $z\in \Lambda_\epsilon(Q^*AQ)$ then there exists $v\in \C^k$ satisfying $\|(z-Q^*AQ)v\|\le \epsilon \|v\|$. Since $I_k=Q^*I_nQ$ we have
    $$ \|Q^*(z-A)Qv\|\le \epsilon\|v\|.$$
    And, because $Q^*$ acts as an isometry on $\range(Q)$ (the span of the columns of $Q$) and by assumption this space is invariant under $P$ (and hence under $(z-A)$), we have that  $(z-A)Qv\in \range(Q)$, and therefore $\|Q^*(z-A)Qv\| = \|(z-A)Qv\|$. From where we obtain
    $$
        \|(z-A)Qv\|\le \epsilon\|v\|=\epsilon\|Qv\|,
    $$
    showing that $z\in \Lambda_\epsilon(A)$.
    
    For the case in which the columns of $Q$ span the rows of $P$, the above proof can be easily modified by now taking $v$ with the property that $\|v^* Q^* (z-A)Q\|\leq \epsilon \|v\|$.
\end{proof}

\begin{observation}
    Since $\eigerr,\omega(\g),\epsilon \le 1$, our assumption on $\spnerr$ in Line 2 of the pseudocode of $\EIG$ implies the following bounds on $\spnerr$ which we will use below: 
    $$
        \spnerr \le \min\left\{0.02,\epsilon/75,\eigerr/100,\frac{\eigerr\epsilon^2}{200\omega(\g)}\right\}.
    $$
\end{observation}

Initial lemmas in hand, let us begin to analyze the algorithm. At several points we will make an assumption on the machine precision in the margin. These will be collected at the end of the proof, where we will verify that they follow from the precision hypothesis of Theorem \ref{thm:eig-finite-guarantee}.\\

\noindent {\bf Correctness.}

\begin{lemma}[Accuracy of $\widetilde{\lambda_i}$] \label{lem:subproblem-shattered}
    When $\SPAN$ succeeds, each eigenvalue of $A$ shares a square of $\g$ with a unique eigenvalue of either $\widetilde{A_{+}}$ or $\widetilde{A_{-}}$, and furthermore $\Lambda_{4\epsilon/5} (\widetilde{A_{\pm}}) \subset \Lambda_\epsilon(A)$.
\end{lemma}

\begin{proof}
   Let $P_{\pm}$  be the true projectors onto the two bisection regions found by $\SPLIT(A,\sgnerr)$, $Q_{\pm}$ \marginnote{$P_{\pm}, Q_{\pm}$} be the matrices whose orthogonal columns span their ranges, and $A_{\pm} := Q_{\pm}^\ast A Q_{\pm}$. \marginnote{$A_{\pm}$}
   From Theorem \ref{thm:deflate-guarantee-usable}, on the event that $\SPAN$ succeeds, the approximation $\widetilde{Q_{\pm}}$ that it outputs satisfies $\|\widetilde{Q_{\pm}} - Q_{\pm}\| \le \spnerr$, so in particular $\|\widetilde{Q_{\pm}}\| \le 2$ as $\spnerr \le 1$. The error $E_{6,\pm}$ from performing the matrix multiplications necessary to compute $\widetilde{A_{\pm}}$ admits the bound
   \begin{align*}
       \|E_{6,\pm}\| 
       &\le \mu_{\MM}(n)\|\widetilde{Q_{\pm}}\|\|A\widetilde{Q_{\pm}}\|\mach + \mu_{\MM}(n)^2\|\widetilde{Q_{\pm}}A\|\mach + \mu_{\MM}(n)^2\|\widetilde{Q_{\pm}}\|^2\|A\|\mach \\
       &\le 16\left(\mu_{\MM}(n)\mach + \mu_{\MM}(n)^2\mach^2\right) & & \|A\| \le 4 \text{ and } \|\widetilde{Q_{\pm}}\| \le 1 + \eta \le 1.02 \le \sqrt{2} \\
       &\le 3\eta & & \mach \le \frac{\eta}{10 \mu_{\MM}(n)^2}.
   \end{align*}
   Iterating the triangle inequality, we obtain
   \begin{align*}
       \|\widetilde{A_{\pm}} - A_{\pm}\| 
       &\le \|E_{6,\pm}\| + \|(\widetilde{Q_{\pm}} - Q_{\pm})A\widetilde{Q_{\pm}}\| + \|Q_{\pm}A(\widetilde{Q_{\pm}} - Q_{\pm})\| \\
       &\le 3\spnerr + 8\spnerr + 4\spnerr & & \|\widetilde{Q_{\pm}} - Q_{\pm}\| \le \spnerr \\
       &\le \epsilon/5 & & \spnerr \le \epsilon/75.
   \end{align*}
   We can now apply Lemma \ref{lem:perturb-shattered}.
\end{proof}

Everything is now in place to show that, if every call to $\SPAN$ succeeds, $\EIG$ has the advertised accuracy guarantees. After we show this, we will lower bound this success probability and compute the running time.

When $A \in \C^{1\times 1}$, the algorithm works as promised. Assume inductively that $\EIG$ has the desired guarantees on instances of size strictly smaller than $n$. In particular, maintaining the notation from the above lemmas, we may assume that
$$
    (\widetilde{V}_{\pm}, \widetilde{D_{\pm}}) = \EIG(\widetilde{A_{\pm}},4\epsilon/5,\g_{\pm},4\eigerr/5,\eigprob,n)
$$
satisfy (i) each eigenvalue of $\widetilde{D_{\pm}}$ shares a square of $\g_{\pm}$ with exactly one eigenvalue of $\widetilde{A_{\pm}}$, and (ii) each column of $\widetilde{V_{\pm}}$ is $4\eigerr/5$-close to a true eigenvector of $\widetilde{A_{\pm}}$. From Lemma \ref{lem:perturb-shattered}, each eigenvalue of $\widetilde{A_{\pm}}$ shares a grid square with exactly one eigenvalue of $A$, and thus the output
$$
    \widetilde{D} = \begin{pmatrix} \widetilde{D_+} &  \\  & \widetilde{D_-} \end{pmatrix}
$$
satisfies the eigenvalue guarantee.

To verify that the computed eigenvectors are close to the true ones, let $\widetilde{\widetilde{v}_{\pm}}$ \marginnote{$\widetilde{\widetilde{v}_{\pm}}$} be some approximate right unit eigenvector of one of $\widetilde{A_{\pm}}$ output by $\EIG$ (with norm $1 \pm n\mach$), $\widetilde{v}_{\pm}$ \marginnote{$\widetilde{v}_{\pm}$} the exact unit eigenvector of $\widetilde{A_\pm}$ that it approximates, and $v_{\pm}$ \marginnote{$v_{\pm}$} the corresponding exact unit eigenvector of $A_{\pm}$. Recursively, $\EIG(A,\epsilon,\g,\eigerr,\eigprob,n)$ will output an approximate unit eigenvector 
$$
    \widetilde{v} := \frac{\widetilde{Q_{\pm}} \widetilde{\widetilde{v}_{\pm}} + e}{\|\widetilde{Q_{\pm}} \widetilde{\widetilde{v}_{\pm}} + e\|} + e',
    \marginnote{$\widetilde{v}$} 
$$ 
 whose proximity to the actual eigenvector $v := Q v_{\pm}$ \marginnote{$v$} we need now to quantify. The error terms here are $e$, a column of the error matrix $E_{8}$ whose norm we can crudely bound by
$$
    \|e\| \le \|E_{8}\| \le \mu_{\MM}(n) \|\widetilde{Q_{\pm}}\|\|\widetilde{V_{\pm}}\|\mach \le 4\mu_{\MM}(n)\mach \le \spnerr,
$$
and $e'$, a column $E_9$ incurred by performing the normalization in floating point; in our initial discussion of floating point arithmetic we assumed in \eqref{eqn:flnorm} that $\|e'\| \le n\mach$. 

First, since $\widetilde{v} - e'$ and $\widetilde{Q_{\pm}}\widetilde{\widetilde{v}_{\pm}} + e$ are parallel, the distance between them is just the difference in their norms:
$$
    \left\|\frac{\widetilde{Q_{\pm}}\widetilde{\widetilde{v}_{\pm}} + e}{\|\widetilde{Q_{\pm}}\widetilde{\widetilde{v}_{\pm}} + e\|} - \widetilde{Q_\pm}\widetilde{\widetilde{v}_{\pm}} + e \right\| 
    \le \left|\|\widetilde{Q_{\pm}}\widetilde{\widetilde{v}_{\pm}} + e\| - 1\right| 
    \le (1 + \spnerr)(1 + \mach) + 4\mu_{\MM}\mach - 1 
    \le 4\spnerr.
$$
Inductively $\| \widetilde{\widetilde{v}_{\pm}} - \widetilde{\widetilde{v}_{\pm}} \| \le 4\eigerr/5$, and since $\|A_{\pm} - \widetilde{A_{\pm}}\| \le \epsilon/5$ and $A_{\pm}$ has shattered $\epsilon$-pseudospectrum from Lemma \ref{lem:compress-shattered}, Lemma \ref{lem:perturb-shattered} ensures
\begin{align*}
    \| \widetilde{\widetilde{v}_{\pm}} - v_{\pm}\| 
    &\le \frac{\sqrt{8}\omega(\g) \cdot 15\spnerr}{\pi\cdot \epsilon(\epsilon - 15\spnerr)} \\
    &\le \frac{\sqrt{8}\omega(\g)\cdot 15\eta}{\pi \cdot 4\epsilon^2/5} & & \eta \le \epsilon/75 \\
    &\le \delta/10 & & \eta \le \frac{\delta\epsilon^2}{200 \omega(\g)}.
\end{align*}
Thus putting together the above, iterating the triangle identity, and using $\|Q_{\pm}\| = 1$, 
\begin{align*}
    \left\| \widetilde{v} - v \right\| 
    &=\left\|\frac{\widetilde{Q_{\pm}}\widetilde{\widetilde{v}_{\pm}} + e}{\|\widetilde{Q_{\pm}}\widetilde{\widetilde{v}_{\pm}} + e\|} + e' - Q_{\pm}v_{\pm} \right\| \\
    &\le \left\|\frac{\widetilde{Q_{\pm}}\widetilde{\widetilde{v}_{\pm}} + e}{\|\widetilde{Q_{\pm}}\widetilde{\widetilde{v}_{\pm}} + e\|} - \widetilde{Q_{\pm}}\widetilde{\widetilde{v}_{\pm}} + e \right\| 
        + \|e'\| 
        + \|e\| 
        + \|(\widetilde{Q_{\pm}} - Q_{\pm})\widetilde{\widetilde{v}_{\pm}}\| 
        + \|Q_{\pm}(\widetilde{\widetilde{v}_{\pm}} - \widetilde{v}_{\pm})\| 
        + \|Q_{\pm}(\widetilde{v}_{\pm} - v_{\pm})\| \\
    &\le 4\spnerr + n\mach + \mu_{\MM}(n)\mach + \eta(1 + n\mach) + 4\eigerr/5 + \eigerr/10  \\
    &\le 8\spnerr + 4\eigerr/5 + \eigerr/10 \hspace{10cm}  n\mach,\mu_{\MM}(n)\mach \le \spnerr \\
    &\le \eigerr \hspace{13.3cm} \spnerr \le \eigerr/200.
\end{align*}

This concludes the proof of correctness of $\EIG$.\\

\noindent {\bf Running Time and Failure Probability.} Let's begin with a simple lemma bounding the depth of $\EIG$'s recursion tree.

\begin{lemma}[Recursion Depth] \label{lem:eig-recursion-depth}
    The recursion tree of $\EIG$ has depth at most $\log_{5/4} n$, and every branch ends with an instance of size $1\times 1$.
\end{lemma}

\begin{proof}
    By Theorem \ref{thm:split-guarantee}, $\SPLIT$ can always find a bisection of the spectrum into two regions containing $n_\pm$ eigenvalues respectively, with $n_+ + n_- = n$ and $n_{\pm} \ge 4n/5$, and when $n\le 5$ can always peel off at least one eigenvalue. Thus the depth $d(n)$ satisfies
    \begin{equation}
        d(n) = \begin{cases} n & n\le 5 \\ 1 + \max_{\theta \in [1/5,4/5]} d(\theta n) & n > 5 \end{cases}
    \end{equation}
    As $n \le \log_{5/4}n$ for $n \le 5$, the result is immediate from induction.
\end{proof}

We pause briefly to verify that the assumptions $\eigerr < 1$, $\epsilon < 1/2$, $\grid$ has side lengths at most $9$, and $\|A\| \le 3.5$ in Theorem \ref{thm:eig-finite-guarantee} ensure that every call to $\SPLIT$ throughout the algorithm satisfies the hypotheses of Theorem \ref{thm:split-guarantee}, namely that $\epsilon \le 0.5, \beta \le 0.05/n, \|A\| \le 4,$ and $\grid$ has side lengths of at most $8$. Since $\eigerr,\epsilon,$ and $\beta$ are non-increasing as we travel down the recursion tree of $\EIG$ --- with $\beta$ monotonically decreasing in $\eigerr$ and $\epsilon$ --- we need only verify that the hypotheses of Theorem \ref{thm:split-guarantee} hold on the initial call to $\EIG$. The condition on $\epsilon$ is immediately satisfied; for the one on $\beta$, we have
$$
    \beta = \frac{\eta^4 \theta^2}{(20 n)^6 \cdot 4 n^8} = \frac{\theta^2\delta^4 \epsilon^8}{200^4 (20 n)^6 \cdot 4 n^8},
$$
which is clearly at most $0.05/n$.

On each new call to $\EIG$ the grid only decreases in size, so the initial assumption is sufficient. Finally, we need that every matrix passed to $\SPLIT$ throughout the course of the algorithm has norm at most $4$. Lemma \ref{lem:subproblem-shattered} shows that if $\|A\| \le 4$ and has its $\epsilon$-pseudospectrum shattered, then $\|\widetilde{A_{\pm}} - A_{\pm}\| \le \epsilon/5$, and since $\|A_{\pm}\| = \|A\|$, this means $\|\widetilde{A_{\pm}}\| \le \|A\| + \epsilon/5$. Thus each time we pass to a subproblem, the  norm of the matrix we pass to $\EIG$ (and thus to $\SPLIT$) increases by at most an additive $\epsilon/5$, where $\epsilon$ is the input to the outermost call to $\EIG$. Since $\epsilon$ decreases by a factor of $4/5$ on each recursion step, this means that by the end of the algorithm the norm of the matrix passed to $\EIG$ will increase by at most an additive $(\epsilon + (4/5)\epsilon + (4/5)^2 \epsilon + \cdots)/5 = \epsilon \le 1/2$. Thus we will be safe if our initial matrix has norm at most $3.5$, as assumed.

\begin{lemma}[Lower Bounds on the Parameters] \label{lem:param-lbs}
    Assume $\EIG$ is run on an $n\times n$ matrix, with some parameters $\eigerr$ and $\epsilon$. Throughout the algorithm, on every recursive call to $\EIG$, the corresponding parameters $\eigerr'$ and $\epsilon'$ satisfy
    $$
        \eigerr' \ge \eigerr/n 
        \qquad \epsilon' \ge \epsilon/n.
    $$
    On each such call to $\EIG$, the parameters $\spnerr'$ and $\sgnerr'$ passed to $\SPLIT$ and $\SPAN$ satisfy
    $$
        \spnerr' \ge \frac{\eigerr\epsilon^2}{200n^3}
        \qquad \sgnerr' \ge \frac{\eigprob^2\eigerr^4\epsilon^8}{(5n)^{26}}.
    $$
\end{lemma}

\begin{proof}
    Along each branch of the recursion tree, we replace $\epsilon \gets 4\epsilon/5$ and $\eigerr \gets 4\eigerr/5$ at most $\log_{5/4}n$ times, so each can only decrease by a factor of $n$ from their initial settings. The parameters $\eta'$ and $\beta'$ are computed directly from $\epsilon'$ and $\eigerr'$.
\end{proof}

\begin{lemma}[Failure Probability]
    $\EIG$ fails with probability no more than $\eigprob$.
\end{lemma}

\begin{proof}
    Since each recursion splits into at most two subproblems, and the recursion tree has depth $\log_{5/4}n$, there are at most 
    $$
        2\cdot 2^{\log_{5/4}n} = 2n^{\frac{\log 2}{\log 5/4}} \le 2n^4
    $$ 
    calls to $\SPAN$. We have set every $\spnerr$ and $\sgnerr$ so that the failure probability of each is $\eigprob/2n^4$, so a crude union bound finishes the proof.
\end{proof}

The arithmetic operations required for $\EIG$ satisfy the recursive relationship
\begin{align*}
    T_{\EIG}(n,\eigerr,\g,\epsilon,\eigprob,n)
    &\le T_{\SPLIT}(n,\epsilon,\sgnerr) + T_{\SPAN}(n,\sgnerr,\eta) + 2 T_{\MM}(n) \\ 
        &\qquad + T_{\EIG}(n_+,4\eigerr/5,\g_{+},4\epsilon/5,\eigprob,n) + T_{\EIG}(n_-,4\eigerr/5,\g_{-},4\epsilon/5,\eigprob,n) \\
        &\qquad + 2T_{\MM}(n) + O(n^2).
\end{align*}
All of $T_{\SPLIT}$, $T_{\SPAN}$, and $T_{\MM}$ are of the form $\polylog(n)\poly(n)$, with all coefficients nonnegative and exponents in the $\poly(n)$ no smaller than $2$. So, for any $n_+ + n_- = n$ and $n_{\pm} \ge 4 n/5$, holding all other parameters fixed, $T_{\SPLIT}(n_+,...) + T_{\SPLIT}(n_-,...) \le \left((4/5)^2 + (1/5)^2\right)T_{\SPLIT}(n,...) = (17/25)T_{\SPLIT}(n,...)$ and the same holds for $T_{\SPAN}$ and $T_{\MM}$. Applying this recursively, with all parameters other than $n$ set to their lower bounds from Lemma \ref{lem:param-lbs}, we then have
\begin{align*}
    T_{\EIG}(n,\eigerr,\g,\epsilon,\eigprob,n) 
    & \le \frac{1}{1 - 17/25}\left(
        T_{\SPLIT}\left(n,\epsilon/n,\g,\frac{\eigerr^4\epsilon^8 \eigprob^2}{(5n)^{26}}\right) \right. \\
        &\qquad\qquad + \left. T_{\SPAN}\left(n,\sgnerr/n,\epsilon/n,\frac{\eigerr^4\epsilon^8 \eigprob^2}{(5n)^{26}}\right) + 4T_{\MM}(n) 
        + O(n^2 ) \right) \\
    & = \frac{25}{8}\left(
        12 N_{\EIG} \lg\frac{1}{\omega(\g)} \left(T_{\INV}(n) 
        + O(n^2)\right) + 2 T_{\QR}(n) \right. \\
        &\qquad \qquad +  5 T_{\MM}(n) + n^2 T_{\N} + O(n^2)\bigg) \\
    &\le 60 N_{\EIG}\lg\frac{1}{\omega(\g)}\left( T_{\INV}(n) + O(n^2)\right) + 10 T_{\QR}(n) + 25 T_{\MM}(n),
\end{align*}
where
\begin{equation*}
    N_{\EIG} := \lg \frac{256 n}{\epsilon} + 3\lg\lg \frac{256 n}{\epsilon} + \lg\lg \frac{(5n)^{26}}{\theta^2\delta^4\epsilon^9} + 7.59.
\end{equation*}
In the above inequalities,  we've substituted in the expressions for $T_{\SPLIT}$ and $T_{\SPAN}$ from Theorems \ref{thm:split-guarantee} and \ref{thm:deflate-guarantee-usable}, respectively; $N_{\EIG}$ is defined by recomputing $N_{\SPLIT}$ with the parameter lower bounds, and the $\epsilon^9$ is not an error. The final inequality uses our assumption $T_\N = O(1)$. Thus using the fast and stable instantiations of $\MM$, $\INV$, and $\QR$ from Theorem \ref{thm:instantiate}, we have
\begin{equation}
    T_{\EIG}(n,\delta,\g,\epsilon,\theta,n) = O\left(\log \frac{1}{\omega(\g)}\left(\log \frac{n}{\epsilon} +  \log\log \frac{1}{\theta\delta}\right) T_{\MM}(n,\mach)\right);
\end{equation}
exact constants can be extracted by analyzing $N_{\EIG}$ and opening Theorem \ref{thm:instantiate}. \\

\noindent {\bf Required Bits of Precision.} We will need the following bound on the norms of all spectral projectors.
\begin{lemma}[Sizes of Spectral Projectors] \label{lem:proj-size}
    Throughout the algorithm, every approximate spectral projector $\widetilde{P}$ given to $\SPAN$ satisfies $\|\widetilde{P}\| \le 10n/\epsilon$.
\end{lemma}
\begin{proof}
    Every such $\widetilde{P}$ is $\sgnerr$-close to a true spectral projector $P$ of a matrix whose $\epsilon/n$-pseudosepctrum is shattered with respect to the initial $8\times 8$ unit grid $\g$. Since we can generate $P$ by a contour integral around the boundary of a rectangular subgrid, we have
    $$
        \|\widetilde{P}\| \le 2 + \|P\| \le 2 + \frac{32}{2\pi}\frac{n}{\epsilon} \le 10n/\epsilon,
    $$
    with the last inequality following from $\epsilon < 1$.
\end{proof}

Collecting the machine precision requirements $\mach \le \mach_{\SPLIT},\mach_{\SPAN}$ from Theorems \ref{thm:split-guarantee} and \ref{thm:deflate-guarantee-usable}, as well as those we used in the course of our proof so far, and substituting in the parameter lower bounds from Lemma \ref{lem:param-lbs}, we need $\mach$ to satisfy
\begin{align*}
    \mach 
    &\le \min \left\{ \frac{\left(1 - \frac{\epsilon}{256n}\right)^{2^{N_{\EIG} + 1}(c_{\INV}\log n + 3)}}{\mu_{\INV}(n)\sqrt{n} N_{\EIG}}, \right. \\
    &\qquad \left. \frac{\epsilon}{100 n^2}, \frac{\theta^2\delta^4\epsilon^8}{(5n)^{26}}\frac{1}{4\|\widetilde{P}\|\max\{\mu_{\QR}(n),\mu_{\MM}(n)\}}, \right. \\
    &\qquad \left. \frac{\delta\epsilon^2}{100 n^3 \cdot 2\mu_{\QR}(n)}, \frac{\delta\epsilon^2}{100 n^3 \max\{ 4\mu_{\MM}(n),n, 2\mu_{\QR}(n)\}} \right\}
\end{align*}
From Lemma \ref{lem:proj-size}, $\|\widetilde{P}\| \le 10n/\epsilon$, so the conditions in the second two lines are all satisfied if we make the crass upper bound
\begin{equation}
    \label{eq:precision-first-two-lines}
    \mach \le \frac{\eigprob^2\eigerr^4\epsilon^8}{(5n)^{30}}\frac{1}{\max\{\mu_{\QR}(n),\mu_{MM}(n),n\}},
\end{equation}
i.e. if $\lg 1/\mach \ge O\left(\lg\frac{n}{\eigprob\eigerr\epsilon}\right)$. Unpacking the first requirement, using the definition $ N_{\EIG} := \lg \tfrac{256 n}{\epsilon} + 3\lg\lg \tfrac{256 n}{\epsilon} + \lg\lg \tfrac{(5n)^{26}}{\theta^2\delta^4\epsilon^9} + 7.59$ from Theorem \ref{thm:eig-finite-guarantee}, and recalling that $\epsilon \le 1/2$, $n \ge 1$, and $(1 - x)^{1/x} \ge 1/4$ for $x \in (0,1/512)$, we have
\begin{align*}
    \frac{\left(1 - \frac{\epsilon}{256n}\right)^{2^{N_{\EIG} + 1}(c_{\INV}\log n + 3)}}{\mu_{\INV}(n)\sqrt{n} N_{\EIG}}
    &= \frac{\left(\left(1 - \frac{\epsilon}{256n}\right)^{\frac{256n}{\epsilon}}\right)^{\lg^3\frac{256 n}{\epsilon} \lg \frac{(5n)^{26}}{\eigprob^2\eigerr^4\epsilon^8}2^{8.59}(c_{\INV}\log n + 3)}}{\mu_{\INV}(n)\sqrt{n}N_{\EIG}} \\
    &\ge \frac{4^{-\lg^3\frac{256 n}{\epsilon} \lg \frac{(5n)^{26}}{\eigprob^2\eigerr^4\epsilon^8}2^{8.59}(c_{\INV}\log n + 3)}}{\mu_{\INV}(n)\sqrt{n} N_{\EIG}},
\end{align*}
so setting $\mach$ smaller than the final expression is sufficient to guarantee $\EIG$ and all subroutines can execute as advertised.  This gives
\begin{align*}
    \lg 1/\mach 
    &\ge \lg^3 \frac{n}{\epsilon}\lg \frac{(5n)^{26}}{\eigprob^2\eigerr^4\epsilon^8}2^{9.59}(c_{\INV}\log n + 3) + \lg N_{\EIG} \\
    &= O\left(\log^3\frac{n}{\epsilon}\log \frac{n}{\eigprob\eigerr\epsilon}\log n\right)\nonumber.
\end{align*}
This dominates the precision requirement from \eqref{eq:precision-first-two-lines}, and completes the proof of Theorem \ref{thm:eig-finite-guarantee}.

\begin{remark}
    A constant may be extracted directly from the expression above --- leaving $\epsilon,\delta,\theta$ fixed, a crude bound on it is $2^{9.59} \cdot 26 \cdot 8 \cdot c_{\INV} \approx 160303 c_{\INV}$. This can certainly be optimized, the improvement with the highest impact would be tighter analysis of $\SPLIT$, with the aim of eliminating the additive $7.59$ term in $N_{\SPLIT}$.
\end{remark}
	\section{Conclusion and Open Questions}
    \label{sec:conclusion}

    In this paper, we reduced the approximate diagonalization problem to a polylogarithmic number of matrix multiplications, inversions, and $QR$ factorizations 
	on a floating point machine with precision depending only polylogarithmically on $n$ and $1/\delta$. The key phenomena enabling this were: (a) every 
	matrix is $\delta$-close to a matrix with well-behaved pseudospectrum, and such a matrix can be found by a complex Gaussian perturbation; and (b) the spectral bisection
	algorithm can be shown to converge rapidly to a forward approximate solution on such a well-behaved matrix, using a polylogarithmic in $n$ and $1/\delta$ amount of precision and number of iterations. The combination of these facts yields a $\delta$-backward approximate solution for the original problem. 
	
	Using fast matrix multiplication, we obtain algorithms with nearly optimal asymptotic computational complexity (as a function of $n$, compared to matrix multiplication), for general
	complex matrices with no assumptions. Using na\"ive matrix multiplication, we get easily implementable algorithms with $O(n^3)$ type complexity and much better constants which are likely faster in practice. The constants in our bit complexity and precision estimates (see Theorem \ref{thm:eig-finite-guarantee} and equations \eqref{eqn:nonasy-sgn} and \eqref{eqn:eigdeltasetting}), while not huge, are likely suboptimal. 
	The reasonable practical performance of spectral bisection based algorithms is witnessed by the many empirical papers (see e.g. \cite{bai1997inverse}) which have studied it. The more recent of these works further show that such algorithms are communication-avoiding and have good parallelizability properties.
	
    \begin{remark}[Hermitian Matrices] A curious feature of our algorithm is that even when the input matrix is Hermitian or real symmetric,
    it begins by adding a complex non-Hermitian perturbation to regularize the spectrum. If one is only interested in this special case, one
    can replace this first step by a Hermitian GUE or symmetric GOE perturbation and appeal to the result of \cite{aizenman2017matrix} instead of Theorem \ref{thm:smoothed},
    which also yields a polynomial lower bound on the minimum gap of the perturbed matrix. It is also possible to obtain a much stronger analysis of the Newton iteration in the Hermitian case, since the iterates are all Hermitian and $\kappa_V=1$ for such matrices. By combining these observations, one can obtain a running
    time for Hermitian matrices which is significantly better (in logarithmic factors) than our main theorem. We do not pursue this further since our 
    main goal was to address the more difficult non-Hermitian case.
    \end{remark}
    
	We conclude by listing several directions for future research.
	\begin{enumerate}
		\item Devise a {deterministic} algorithm with similar guarantees. The main bottleneck to doing this is deterministically finding a regularizing perturbation,
	    which seems quite mysterious. Another bottleneck is computing a rank-revealing QR factorization in near matrix multiplication time deterministically (all of the currently known deterministic algorithms require $\Omega(n^3)$ time).
	    \item Determine the correct exponent for smoothed analysis of the eigenvalue gap of $A+\gamma G$ where $G$ is a complex Ginibre matrix. We currently obtain roughly $(\gamma/n)^{8/3}$ in Theorem \ref{thm:polygap}.
	    Is it possible to match the $n^{-4/3}$ type dependence \cite{vinson2011closest} which is known for a pure Ginibre matrix?
	    \item Reduce the dependence of the running time and precision to a smaller power of $\log(1/\delta)$. The bottleneck in the current algorithm is the number of bits
	    of precision required for stable convergence of the Newton iteration for computing the sign function. Other, ``inverse-free'' iterative schemes
	    have been proposed for this, which conceivably require lower precision.
	    \item Study the convergence of ``scaled Newton iteration'' and other rational approximation methods (see \cite{higham2008functions, nakatsukasa2016computing}) for computing the sign function on non-Hermitian matrices. Perhaps these have even faster convergence and better stability properties?
	\end{enumerate}
	More broadly, we hope that the techniques introduced in this paper---pseudospectral shattering and pseudospectral analysis of matrix iterations using contour integrals---are useful in attacking other problems in numerical linear algebra.
	
 	\subsection*{Acknowledgments} We thank Peter B\"urgisser for introducing us to this problem, and Ming Gu, Olga Holtz, Vishesh Jain, Ravi Kannan, Pravesh Kothari, Lin Lin,  Satyaki Mukherjee, Yuji Nakatsukasa, and Nick Trefethen for helpful conversations. We thank the referees for a careful reading of the paper and many helpful comments which improved it. We thank the Institute for Pure and Applied Mathematics, where part of this work was carried out. 
	\bibliographystyle{abbrv}
	\bibliography{diag}
	\appendix 
	
	\section{Deferred Proofs from Section \ref{sec:matrix-sign}}
\label{sec:deferredsign}

\begin{lemma}[Restatement of Lemma \ref{lem:formalizenoise}]
    Assume the matrix inverse is computed by an algorithm $\INV$ satisfying the guarantee in Definition \ref{def:inv}.  Then $\G(A) = g(A) + E$ for some error matrix $E$ with norm \begin{equation}
        \Vert E \Vert \le \left(\Vert A \Vert + \Vert A^{-1} \Vert  + \pinv(n) \kappa(A)^{\cinv \log n}\|A^{-1}\| \right) 2 \sqrt{n} \mach.\end{equation}
\end{lemma} 
\begin{proof} The computation of $\G(A)$ consists of three steps:

\begin{enumerate}
    \item Form $A^{-1}$ according to Definition \ref{def:inv}.  This incurs an additive error of $E_{\INV} = \pinv(n)\cdot \u\cdot \kappa(A)^{\cinv \log n}\|A^{-1}\|$.  The result is $\INV(A) = A^{-1} + E_\INV.$
    \item Add $A$ to $\INV(A)$.  This incurs an entry-wise relative error of size $\mach$: The result is \[(A + A^{-1} + E_{\INV}) \circ (J + E_{add})\] where $J$ denotes the all-ones matrix, $\Vert E_{add} \Vert_{max} \le \mach$, and where $\circ$ denotes the entrywise (Hadamard) product of matrices.
    \item Divide the resulting matrix by 2, which is an exact operation in our floating-point model as we can simply decrement the exponent. The final result is 
    \[\G(A) = \frac{1}{2}(A + A^{-1} + E_{\INV}) \circ (J + E_{add}). \]
    
\end{enumerate}
Finally, recall that for any $n \times n$ matrices $M$ and $E$, we have the relation (\ref{eqn:maxnorm})
\[
    \Vert M \circ E\Vert 
    \le  \Vert M \Vert \Vert E \Vert_{max} \sqrt{n}.
\]
Putting it all together, we have 
\begin{align*}
    \Vert \G(A) - \newton(A) \Vert 
    &\le \frac{1}{2} \left( \Vert A \Vert + \Vert A^{-1} \Vert \right) \mach \sqrt{n}  + \Vert E_\INV \Vert ( 1 + \mach) \sqrt{n} \\
    &\le  \frac{1}{2} \left( \Vert A \Vert + \Vert A^{-1} \Vert \right) \mach \sqrt{n}  + \pinv(n)\cdot \u\cdot \kappa(A)^{\cinv \log n}\|A^{-1}\| ( 1 + \mach) \sqrt{n} \\
    &\le \left(\Vert A \Vert + \Vert A^{-1} \Vert  + \pinv(n) \kappa(A)^{\cinv \log n}\|A^{-1}\| \right) 2 \sqrt{n} \mach
\end{align*}
where we use $\mach < 1$ in the last line.

\end{proof}

In what remains of this section we will repeatedly use the following simple calculus fact. 

\begin{lemma} \label{lem:log}
Let $x, y >0$, then $$\log(x+y) \leq \log(x)+ \frac{y}{x} \quad \mathrm{and} \quad \lg(x+y) \leq \lg(x)+ \frac{1}{\log 2}\frac{y}{ x}.$$
\end{lemma}

\begin{proof}
This follows directly from the concavity of the logarithm. 
\end{proof}

\begin{lemma}[Restatement of Lemma \ref{lem:prelimN}]

 Let $1/800 > t > 0$ and $1/2 > c > 0$ be given.  Then for
 \[j \ge \lg(1/t) + 2 \lg \lg (1/t) + \lg \lg(1/c) + 1.62,\]
 we have
 \[ \frac{(1-t)^{2^j}}{t^{2j}} < c.\]
\end{lemma} 

\begin{proof}[Proof of Lemma \ref{lem:prelimN}]
An exact solution for $j$ can be written in terms of the \emph{Lambert $W$-function}; see \cite{corless1996lambertw} for further discussion and a useful series expansion.  For our purposes, it is simpler to derive the necessary quantitative bound from scratch.

Immediately from the assumption $t < 1/800$, we have $j > \log(1/t) \ge 9$.

First let us solve the case $c = 1/2$.  We will prove the contrapositive, so assume 
\[\frac{(1-t)^{2^j}}{t^{2j}} \ge 1/2.\]
Then taking $\log$ on both sides, we have
\[ 2j \log (1/t) + 1 \ge - 2^j \log (1-t) \ge 2^j t.\]
Taking $\lg$ of both sides and applying the second inequality in Lemma \ref{lem:log} with $x=2j \log(1/t)$ and $y=1$, using $\lg x =  1 + \lg j + \lg \log (1/t)$, we obtain
\[ 1 + \lg j + \lg \log (1/t) + \frac{1}{\log 2} \frac{1}{2 j \log(1/t)} \ge j + \lg t.\]
Since $t < 1/800$ we have $\frac{1}{\log 2} \frac{1}{2 j \log(1/t)} < 0.01$, so $$j - \lg j \le \lg (1/t) + \lg \log (1/t) + 1.01 \le \lg(1/t) + \lg \lg (1/t) + 0.49 =: K.$$  But since $j \ge 9$, we have $j - \lg j \ge 0.64 j$, so \[j \le \frac{1}{0.64}(j-\lg j) \le \frac{1}{0.64}K\] which implies \[j \le K + \lg j \le K + \lg (1.57 K) = K + \lg K + 0.65.\] 

Note $K \le 1.39 \lg (1/t)$, because $K - \lg (1/t) = \lg \lg (1/t) + 0.49 \le  0.39 \lg (1/t)$ for $t \le 1/800$. Thus \[\lg K \le \lg (1.39 \lg (1/t)) \le \lg \lg (1/t) + 0.48,\] so for the case $c=1/2$ we conclude the proof of the contrapositive of the lemma:
\begin{align*}
    j &\le K + \lg K + 0.65 \\
    &\le \lg(1/t) + \lg \lg (1/t) + 0.49 + (\lg \lg (1/t) + 0.48) + 0.65 \\
    &= \lg(1/t) + 2 \lg \lg(1/t)  + 1.62.
\end{align*}

For the general case, once $(1-t)^{2^j}/t^{2j} \le 1/2$, consider the effect of incrementing $j$ on the left hand side.  This has the effect of squaring and then multiplying by $t^{2j-2}$, which makes it even smaller.  At most $\lg \lg (1/c)$ increments are required to bring the left hand side down to $c$, since $(1/2)^{2^{\lg \lg (1/c)}} = c$.  This gives the value of $j$ stated in the lemma, as desired.
\end{proof}

\begin{lemma}[Restatement of Lemma \ref{lem:cleanN}]

If \[ N = \lceil \lg(1/s) + 3 \lg \lg(1/s) + \lg \lg (1/(\sgnerr \eps_0)) + 7.59 \rceil, \]
then
\[  N \ge \lg(8/s) + 2 \lg \lg(8/s) + \lg \lg (16/(\sgnerr s^2 \eps_0)) + 1.62.\]
\end{lemma}

\begin{proof}[Proof of Lemma \ref{lem:cleanN}]
We aim to provide a slightly cleaner sufficient condition on $N$ than the current condition \[ N \ge \lg(8/s) + 2 \lg \lg(8/s) + \lg \lg (16/(\sgnerr s^2 \eps_0)) + 1.62.\]  Repeatedly using Lemma \ref{lem:log}, as well as the cruder fact $\lg \lg (ab) \le \lg \lg a + \lg \lg b$ provided $a, b \ge 4$,  we have
\begin{align*}
    \lg \lg (16 / ( \sgnerr s^2 \eps_0)) &\le \lg \lg (16/s^2) + \lg \lg (1/(\sgnerr \eps_0))  \\
    &= 1 + \lg(3 + \lg (1/s)) + \lg \lg (1/(\sgnerr \eps_0))  \\
    &\le 1 + \lg \lg (1/s) + \frac{ 3}{\log 2 \lg (1/s)} + \lg \lg (1/(\sgnerr \eps_0))  \\
    &\le \lg \lg(1/s) + \lg \lg(1/(\sgnerr \eps_0)) + 1.66 
\end{align*}
where in the last line we use the assumption $s < 1/100$.  Similarly, 
\begin{align*}
    \lg(8/s) + 2 \lg \lg (8/s) &\le 3 + \lg(1/s) + 2 \lg (3 + \lg (1/s)) \\
    &\le 3 + \lg(1/s) + 2 \left(\lg \lg(1/s) + \frac{3}{\log 2 \lg (1/s)} \right) \\
    &\le \lg(1/s) + 2 \lg \lg (2/s) + 4.31
\end{align*}
Thus, a sufficient condition is
\[ N = \lceil \lg(1/s) + 3 \lg \lg(1/s) + \lg \lg (1/(\sgnerr \eps_0)) + 7.59 \rceil. \]

\end{proof}
	\section{Analysis of $\SPLIT$} \label{sec:split}

Although it has many potential uses in its own right, the purpose of the approximate matrix sign function in our algorithm is to split the spectrum of a matrix into two roughly equal pieces, so that approximately diagonalizing $A$ may be recursively reduced to two sub-problems of smaller size.

First, we need a lemma ensuring that a shattered pseudospectrum can be bisected by a grid line with at least $n/5$ eigenvalues on each side.
\begin{lemma}
    Let $A$ have $\eps$-pseudospectrum shattered with respect to some grid $\g$. Then there exists a horizontal or vertical grid line of $\g$ partitioning $\g$ into two grids $\g_\pm$, each containing at least $\min\{n/5,1\}$ eigenvalues.
\end{lemma}

\begin{proof}
    We will view $\g$ as a $s_1 \times s_2$ array of squares.  Write $r_1,r_2,...,r_{s_1}$ for the number of eigenvalues in each row of the grid. Either there exists $1 \le i < s_2$ such that $r_1 + \cdots + r_i \ge n/5$ and $r_{i+1} + \cdots + r_{s_1} \ge n/5$---in which case we can bisect at the grid line dividing the $i$th from $(i+1)$st rows---or there exists some $i$ for which $r_{i} \ge 3/5$. In the latter case, we can always find a vertical grid line so that at least $n/5$ of the eigenvalues in the $i$th row are on each of the left and right sides. Finally, if $n\le 5$, we may trivially pick a grid line to bisect along so that both sides contain at least one eigenvalue.
\end{proof}

\begin{figure}[ht]
    \begin{boxedminipage}{\textwidth}
        $$\SPLIT$$  {\small
        \textbf{Input:} Matrix $A \in \C^{n\times n}$, grid $\g = \grid(z_0,\omega,s_1,s_2)$ pseudospectral guarantee $\epsilon$, and a desired accuracy $\nu$. \\
        \textbf{Requires:} $\Lambda_\epsilon(A)$ is shattered with respect to $\g$, and  $\sgnerr \le 0.05/n$. \\
        \textbf{Algorithm:} $(\widetilde{P}_+,\widetilde{P}_-,\g_+,\g_-) = \SPLIT(A,\g,\epsilon,\sgnerr)$
        \begin{enumerate}
            \item $h \gets \Re z_0 + \omega s_1/2$
            \item $M \gets A - h + E_2$
            \item $\alpha_0 \gets 1 - \frac{\epsilon}{2\diag(\g)^2}$
            \item $\phi \gets \text{round} \left( \Tr\, \SGN(M,\epsilon/4,\alpha_0, \sgnerr) + e_4\right)$
            \item If $|\phi| < \min(3n/5,n-1)$
            \begin{enumerate}
                \item $\g_- = \grid(z_0,\omega,s_1/2,s_2)$
                \item $z_0 \gets z_0 + h$
                \item $\g_+ = \grid(z_0,\omega,s_1/2,s_2)$
                \item $(\widetilde{P}_+, \widetilde{P}_-) = \frac{1}{2}(1 \pm \SGN(A - h,\sgnerr))$
            \end{enumerate}
            \item Else, execute a binary search over horizontal grid-line shifts $h$ until $\Tr\,\SGN(A - h,\epsilon/4,\alpha_0, \sgnerr)\leq \frac{3n}{5}$, at which point output $\g_{\pm}$, the subgrids on either side of the shift $h$, and set \\$\widetilde{P}_{\pm} \gets \tfrac{1}{2}\left(\SGN(h-A,\epsilon/4,\alpha_0,\sgnerr)\right)$.
            \item If this fails, set $A \gets iA$, and execute a binary search among vertical shifts from the original grid.
        \end{enumerate}
        \textbf{Output:} Sub-grids $\g_{\pm}$, approximate spectral projectors $\widetilde{P}_{\pm}$, and ranks $n_{\pm}$. \\
        \textbf{Ensures:} There exist true spectral projectors $P_{\pm}$ satisfying (i) $P_+ + P_- = 1$, (ii)$\rank (P_{\pm}) = n_{\pm} \ge n/5$, (iii) $\|P_{\pm} - \widetilde{P}_{\pm}\| \le \sgnerr$, and (iv) $P_\pm$ are the spectral projectors onto the interiors of $\g_{\pm}$.}
    \end{boxedminipage}
\end{figure}

\begin{proof}[Proof of Theorem \ref{thm:split-guarantee}]
    
    The main observation is that, given any matrix $A$, we can determine how many eigenvalues are on either side of any horizontal or vertical line by approximating the sign function of a shift of the matrix. To be precise,  in exact arithmetic $\Tr\,\sgn(A - h) = n_+ - n_-$, where $n_\pm$ are the eigenvalue counts for $A$ on either side of the line $\Re z = h$. We will now show that under the shattered pseudospectrum  assumption, one can exactly compute $n_+-n_-$ using the advertised precision. 
    
    Running $\SGN$ to a final accuracy of $\sgnerr$,
    \begin{align*}
        |\Tr\, \SGN(M) + e_4 - \Tr\,\sgn(M)| 
        &\le |\Tr\, \SGN(M) - \Tr\,\sgn(M)| + |e_4| \\
        &\le n\big(\|\SGN(M) - \sgn(M)\| + \|\SGN(M)\|\mach \big) & & \text{Using \eqref{eqn:fltrace} to bound $|e_4|$} \\
        &\le n\big(\sgnerr + (\sgnerr + \|\sgn(M)\|)\mach\big).
    \end{align*}
   It remains to control $\|\sgn(M)\|$ and quantify the distance between $\sgn(M) = \sgn(A-h+E_2)$ and $\sgn(A-h)$. We first do the latter. Since we need only to modify the diagonal entries of $A$ when creating $M$, the incurred  \textit{diagonal} error matrix $E_2$ has norm at most $\mach \max_i |A_{i,i} - h|$. Using $|A_{i,i}| \le \|A\| \le 4$ and $|h| \le 4$, the fact that $\mach \le \epsilon/100 n \le \epsilon/16$ ensures that the $\epsilon/2$-pseudospectrum of $M$ will still be shattered with respect to  $\g$. We can then form $\sgn(A-h)$ and $\sgn(M)$ by integrating around the boundary of the portions of $\g$ on either side of the line $\Re z = h$, then using the resolvent identity as in Section \ref{sec:matrix-sign}, and the fact that $\Lambda_\epsilon(A)$ and $\Lambda_{\epsilon/2}(M)$ are shattered we get
    $$
       \|\sgn(A)-\sgn(M)\| \leq \frac{\|E_2\|}{2\pi}\cdot \frac{1}{\epsilon}\cdot \frac{2}{\epsilon}\omega(2s_1 + 4s_2) \leq \frac{128\mach }{\epsilon^2}
    $$
    where in the last inequality  we have used that $\g$ has side lengths of at most $8$ and $\|E_2\|\leq 8  \mach$. 
    
    Now, using the contour integral again and the shattered pseudospectrum assumption 
    $$
        \|\sgn(A - h)\| \le \frac{1}{2\pi}\frac{1}{\epsilon}\omega(2s_1 + 4s_2) \le 8/\epsilon.
    $$
    Combining the above bounds we get a a total additive error of $n(\sgnerr + \sgnerr\mach + 8\mach/\epsilon)+\frac{128 \mach }{\epsilon^2}$ in computing the trace of the sign function. If $\sgnerr \le 0.1/n$ and $\mach \le \min\{\epsilon/100n, \frac{\epsilon^2}{512}$,  this error will strictly be less than $0.5$ and we can round $\Tr\, \SGN(A - h)$ to the nearest real integer. Horizontal bisections work similarly, with $iA - h$ instead. 
    
Now that we have shown that it is possible to compute $n_{+}-n_-$ exactly, recall that   from the above discussion, the $\epsilon/2$-pseudospectrum of $M$ will still be shattered with respect to the translation of the original grid $\g$.  Using Lemma \ref{lem:params-for-sgn} and the fact that $\diag(\g)^2 = 128$, we can safely call $\SGN$ with parameters $\epsilon_0 = \epsilon/4$ and
    $$
        \alpha_0 = 1 - \frac{\epsilon}{256}.
    $$
    Plugging these in to the Theorem \ref{prop:sgnerr} ($\epsilon < 1/2$ so $1-\alpha_0 \le 1/100$, and $\beta \le 0.05/n \le 1/12$ so the hypotheses are satisfied) for final accuracy $\sgnerr$ a sufficient number of iterations is
    \begin{align*}
       N_{\SPLIT} := \lg \frac{256}{\epsilon} + 3\lg\lg \frac{256}{\epsilon} + \lg\lg \frac{4}{\sgnerr\epsilon} + 7.59. 
    \end{align*}
    In the course of these binary searches, we make at most $\lg s_1 s_2$ calls to $\SGN$ at accuracy $\sgnerr$. These  require at most
    $$
        \lg s_1s_2\, T_{\SGN}\left(n,\epsilon/2,1 - \frac{\epsilon}{2\diag(\g)^2},\sgnerr\right) 
    $$
    arithmetic operations. In addition, creating $M$ and computing the trace of the approximate sign function cost us $O(n \lg s_1s_2)$ scalar addition operations. We are assuming that $\g$ has side lengths at most $8$, so $\lg s_1 s_2 \le 12\lg 1/\omega(\g)$. Combining all of this with the runtime analysis and machine precision of $\SGN$ appearing in Theorem \ref{prop:sgnerr}, we obtain
    $$
        T_{\SPLIT}(n,\g,\epsilon,\sgnerr) \le 12 \lg \frac{1}{\omega(\g)}\cdot N_{\SPLIT} \cdot \left(T_{\INV}(n,\mach) + O(n^2) \right).
    $$
\end{proof}

	\section{Analysis of $\SPAN$} \label{sec:deflate}

The algorithm $\SPAN$, defined in Section \ref{sec:spectralbisec}, can be viewed as a small variation of the randomized rank revealing algorithm introduced in \cite{demmel2007fast} and revisited subsequently in \cite{ballard2019generalized}. Following these works, we will call this algorithm $\RURV$.

Roughly speaking, in finite arithmetic,  $\RURV$ takes a matrix $A$ with $\sigma_r(A)/\sigma_{r+1}(A) \gg 1$, for some $1\leq r \leq n-1$, and finds nearly unitary matrices $U,V$ and an upper triangular matrix $R$ such that $ URV\approx A$. Crucially, $R$ has the block decomposition
\begin{equation}
    \label{eq:decofR}
    R =  \begin{pmatrix}  R_{11} & R_{12} \\ & R_{22}  \end{pmatrix},
\end{equation}
where $R_{11} \in \mathbb{C}^{r\times r}$ has smallest singular value close to $\sigma_r(A)$, and $R_{22}$ has largest singular value roughly $\sigma_{r+1}(A)$. We will use and analyze the following implementation of $\RURV$. 

\begin{figure}[ht]
    \begin{boxedminipage}{\textwidth}
    $$\RURV$$
{\small        \textbf{Input}: Matrix $A \in \C^{n\times n}$ \\
        \textbf{Algorithm}: $\RURV(A)$
        \begin{enumerate}
             \item $G \gets n\times n$ complex Ginibre matrix $+  E_1$
            \item  $(V, R) \gets \QR (G)$
            \item $B \gets A V^*+E_3$
            \item $(U, R) \gets \QR (B) $
    \end{enumerate}
        \textbf{Output}: A pair of matrices $(U,R)$. \\ 
        \textbf{Ensures: } $\|R_{22}\| \leq \frac{\sqrt{r(n-r)}}{\theta}\sigma_{r+1}(A)$ with probability $1-\theta^2$, for every $1\leq r \leq n-1$ and $\theta > 0$, where $R_{22}$ is the $(n-r)\times (n-r)$ lower-right corner of $R$.  }
    \end{boxedminipage}
\end{figure}

As discussed in Section \ref{sec:spectralbisec}, we hope to use $\SPAN$ to approximate the range of a projector $P$ with rank $r<n$, given an approximation $\widetilde{P}$ close to $P$ in operator norm. We will show that from the output of $\RURV(\widetilde{P})$ we can obtain a good approximation to such a subspace.  More specifically, under certain conditions, if $(U,  R) = \RURV(\widetilde P)$, then the first $r$ columns of $U$ carry all the information we need. For a formal statement see Proposition \ref{prop:spanexact} and Proposition \ref{prop:spanfinite} below. 

Since it may be of broader use, we will work in somewhat greater generality, and define the subroutine $\SPAN$ which receives a matrix $A$ and an integer $r$ and returns a matrix $S \in \mathbb{C}^{n\times r}$ with nearly orthonormal columns. Intuitively, if $A$ is diagonalizable, then under the  guarantee that $r$ is the smallest integer $k$ such that $\sigma_k(A)/\sigma_{k+1}(A) \gg 1$, the columns of the output $S$ span a space close to the span of the top $r$ eigenvectors of $A$. Our implementation of $\SPAN$ is as follows.

\begin{figure}[ht]
    \begin{boxedminipage}{\textwidth}
  $$\SPAN$$ { \small 
       \noindent \textbf{Input}: Matrix $\widetilde{A} \in \C^{n\times n}$ and parameter $r \le n$ \\
          \textbf{Requires:} $1/3 \le \|A\|$, and $\|\widetilde{A} - A\| \le \indef$ for some $A\in \mathbb{C}^{n\times n}$ with $\rank(A)= \rank(A^2)=r$, as well as $\indef \le 1/4 \le \|\widetilde{A}\|$ and $ 1 \le \mu_{\MM}(n),\mu_{\QR}(n),\cn$. \\
          
          \textbf{Algorithm:} $\widetilde{S} = \SPAN(A, r)$. 
        \begin{enumerate}
         \item $(U, R) \gets \RURV(A)$
            \item $\widetilde{S} \gets$  first $r$ columns of $U$.
            \item Output $\widetilde{S}$ 
        \end{enumerate}
        \textbf{Output}: Matrix $S \in \mathbb{C}^{n\times r}$. \\
         \textbf{Ensures:} There exists a matrix $S \in \C^{n\times k}$ whose orthogonal columns span $\text{range}(A)$, such that  $\|\widetilde{S} - S\| \le \outdef$, with probability at least $ 1 - \frac{ (20n)^3\sqrt{\indef} }{\outdef ^2 \sigma_r(A)} $. }
    \end{boxedminipage}
\end{figure}

Throughout this section we use $\rurv(\cdot)$ and $\spa(\cdot , \cdot) $ \marginnote{$\rurv(\cdot)$ \\ $\spa(\cdot, \cdot)$} to denote the exact arithmetic versions of $\RURV$ and $\SPAN$ respectively. In Subsection \ref{sec:smallestsingval} we present a random matrix result that will be needed in the analysis of $\SPAN$. In Subsection  \ref{sec:rurv} we state the properties of $\RURV$ that will be needed.  Finally in Subsections \ref{sec:deflateguaranties} and \ref{sec:deflatefiniteguaranties} we prove the main guarantees of $\spa$ and $\SPAN$, respectively, that are used throughout this paper.  

\subsection{Smallest Singular Value of the Corner of a Haar Unitary}
\label{sec:smallestsingval}
We recall the defining property of the Haar measure on the unitary group:
\begin{definition}
    A random $n\times n$ unitary matrix $V$ is \textit{Haar-distributed} if, for any other unitary matrix $W$, $VW$ and $WV$ are Haar-distributed as well.
\end{definition}
For short, we will often refer to such a matrix as a \emph{Haar unitary.}

Let $n >  r$ be positive integers. In what follows we will consider an $n\times n$ Haar unitary matrix $V$ and  denote by $X$ its upper-left $r \times r$ corner. The  purpose of the present subsection is to derive a  tail bound for the random variable $\sigma_{r}(X)$. We begin by showing a fact that allows us to reduce our analysis to the case when $r\leq n/2$. 

\begin{observation}
\label{obs:equalsingvals}
    Let $n > r>0$ and $V\in \mathbb{C}^{n\times n}$ be a unitary matrix and denote by $V_{11}$ and $V_{22}$ its upper-left $r\times r$ corner and its lower-right $(n-r)\times (n-r)$ corner respectively. If $r\geq n/2$, then $2r-n$ of the singular values of $V_{11}$ are equal to 1, while the remaining $n-r$ are equal to those of $V_{22}$.  
\end{observation}

\begin{proof}
Decompose $V$ as follows 
$$V = \left(\begin{array}{cc}
    V_{11} & V_{12} \\
   V_{21}  & V_{22}
\end{array} \right).$$
Since $V$ is unitary $VV^*=I_n$, and looking at the upper-left corner of this equation we get $V_{11} V_{11}^*+V_{12}V_{12}^*=I_r$. Then, since $V_{11}V_{11}^*=I_r-V_{12}V_{12}^*$, we have $\Lambda(V_{11}V_{11}^*) = 1 -\Lambda(V_{12}V_{12}^*)$.  

Now, looking at the lower-right corner of the equation $V^*V=I_n$ we get $V_{12}^*V_{12}+V_{22}^*V_{22}=I_{n-r}$ and hence $\Lambda(V_{22}^*V_{22})=1-\Lambda(V_{12}^*V_{12})$. 

Now recall that for any two matrices $X$ and $Y$, the symmetric difference of the sets $\Lambda(XY)$ and $\Lambda(YX)$ is $\{0\}$, with multiplicity equal to the difference between the dimensions. Hence $\Lambda(V_{12} V_{12}^*) = \Lambda(V_{12}^* V_{12})\cup \{0\}$ where the multiplicity of 0 is $r-(n-r)=2r-n$. Combining this with $\Lambda(V_{11}V_{11}^*) = 1 -\Lambda(V_{12}V_{12}^*)$ and $\Lambda(V_{22}^*V_{22})=1-\Lambda(V_{12}^*V_{12})$ we get the desired result. 
\end{proof}

\begin{proposition}[$\sigma_{\min}$ of a submatrix of a Haar unitary]
\label{prop:sigmamin}
Let $n> r>0$  and let $V$ be an $n\times n$ Haar unitary. Let $X$ be the upper left $r\times r$ corner of $V$. Then, for all $\theta\in (0, 1]$
\begin{equation}
\label{eq:tailformulasigma}
\mathbb{P}\left[ \frac{1}{\sigma_{r}(X)} \leq \frac{1}{\theta}  \right]= (1-\theta^2)^{r(n-r)}.
\end{equation}
In particular, for every $\theta \in (0, 1]$ we have  
\begin{equation}
\label{eq:tailboundsigma}    
\mathbb{P}\left[ \frac{1}{\sigma_{r}(X)} \leq \frac{\sqrt{r(n-r)}}{\theta}  \right]\geq 1-\theta^2.
\end{equation}
\end{proposition}

This exact formula for the CDF of the smallest singular value of $X$ is remarkably simple, and we have not seen it anywhere in the literature. It is an immediate consequence of substantially more general results of Dumitriu \cite{dumitriu2012smallest}, from which one can extract and simplify the density of $\sigma_r(X)$. We will begin by introducing the relevant pieces of \cite{dumitriu2012smallest}, deferring the final proof until the end of this subsection.

Some of the formulas presented here are written in terms of the generalized  hypergeometric function which we denote by ${}_2 F_1^\beta(a, b; c; (x_1, \dots, x_m)).$ For our application it is sufficient to know that
\begin{equation}
\label{eq:defgenhypergeom}
{}_2F_1^\beta(0, b;c, (x_1, \dots, x_m)) = 1,
\end{equation}
whenever $c >0$ and  ${}_2F_1$ is well defined. The above equation can be derived directly from the definition of ${}_2F_1^\beta$ \marginnote{${}_2F_1^\beta$} (see Definition 13.1.1 in \cite{forrester2010log} or Definition 2.2 in \cite{dumitriu2012smallest}). 

The generic results in \cite{dumitriu2012smallest} concern the \textit{$\beta$-Jacobi} random matrices, which we have no cause here to define in full. Of particular use to us will be \cite[Theorem 3.1]{dumitriu2012smallest}, which expresses the density of the smallest singular value of such a matrix in terms of the generalized hypergeometric function:

\begin{theorem}[\cite{dumitriu2012smallest}]
\label{thm:formula}
The density of the probability distribution of the smallest eigenvalue $\lambda$, of the $\beta$-Jacobi ensembles of parameters $a, b$ and size $m$, which we denote by $f_{\lambda_{\min}}(\lambda)$ \marginnote{$f_{\lambda_{\min}}$}, is given by 
\begin{equation}
\label{eq:supergeneralformula}
   C_{\beta, a, b, m}\lambda^{\frac{\beta}{2}(a+1)-1}(1-\lambda)^{\frac{\beta}{2}m(b+m)-1} {}_2{F}_1^{2/\beta} \left(1-\frac{\beta(a+1)}{2}, \frac{\beta(b+m-1)}{2}; \frac{\beta(b+2m-1)}{2}+1; (1-\lambda)^{m-1}\right),
\end{equation}
for some normalizing constant $C_{\beta, a, b, m}$. 
\end{theorem}

For a particular choice of parameters, the above theorem can be applied to describe the the distribution of $\sigma_{r}^2(X)$. The connection between singular values of corners of Haar unitary matrices and $\beta$-Jacobi ensembles is the content of \cite[Theorem 1.5]{edelman2008beta}, which we rephrase below to match our context. 

\begin{theorem}[\cite{edelman2008beta}]
\label{thm:haarvsjacobi}
Let $V$ be an $n\times n$ Haar unitary matrix and let $r\leq \frac{n}{2}$. Let $X$ be the $r\times r$ upper-left corner of $V$. Then, the eigenvalues of $XX^*$ distribute as the eigenvalues of a $\beta-$Jacobi matrix of size $r$ with parameters $\beta=2, a=0$ and $b=n-2r$. 
\end{theorem}

In view of the above result, Theorem \ref{thm:formula} gives a formula for the density of $\sigma_{r}^2(X)$. 

\begin{corollary} [Density of $\sigma_{r}^2(X)$] 
\label{cor:densitysigmar}
    Let $V$ be an $n\times n$ Haar unitary and $X$ be its upper-left $r\times r$ corner with $r<n$, then $\sigma_{r}^2(X)$ has the following density 
    \begin{equation} \label{eq:densityminsigma}
        f_{\sigma_{r}^2}(x) := \begin{cases} r(n-r) \left(1-x\right)^{r(n-r)-1} & \text{if}\quad  0\leq x \leq 1,\\ 0 & \text{otherwise.}\end{cases}
        \marginnote{$f_{\sigma_{r}^2}$}
    \end{equation}
\end{corollary}

\begin{proof}
    If $r > n/2$, since we care only about the smallest singular value of $X$, we can use Observation \ref{obs:equalsingvals} to analyse the $(n-r)\times (n-r)$ lower right corner of $V$ instead. Hence, we can assume without loss of generality that $r\leq n/2$. Now, substitute $\beta=2, a=0, b=n-2r, m=r$ in Theorem \ref{thm:formula} and observe that in this case 
    \begin{equation} \label{eq:formfordensity}
        f_{\lambda_{\min}}(x) = C  (1-x)^{r(n-r)-1} {}_2F_1^{1}(0, n-r-1;n;(1-x)^{r-1}) = C (1-x)^{r(n-r)-1}   
    \end{equation}
    where the last equality follows from (\ref{eq:defgenhypergeom}). Using the relation between the distribution of $\sigma_{r}^2(X)$ and the distribution of the minimum eigenvalue of the respective $\beta$-Jacobi ensemble described in Theorem \ref{thm:haarvsjacobi} we have $f_{\sigma^2_{r}}(x) =  f_{\lambda_{\min}}(x)$. By integrating on $[0,1]$ the right side of (\ref{eq:formfordensity}) we find $C= r(n-r)$. 
\end{proof}

\begin{proof}[Proof of Proposition \ref{prop:sigmamin}]
    From (\ref{eq:densityminsigma}) we have that
    $$\mathbb{P}\left[\sigma_{r}^2(X)\leq \theta \right] = r(n-r) \int_0^\theta (1-x)^{r(n-r)-1} dx = 1-(1-\theta)^{r(n-r)},$$
    from where (\ref{eq:tailformulasigma}) follows. To prove (\ref{eq:tailboundsigma}) note that $g(t):=(1-t)^{r(n-r)}$ is convex in $[0, 1]$, and hence $g(t) \geq g(0)+tg'(0)$ for every $t\in [0, 1]$. 
\end{proof}

\subsection{Sampling Haar Unitaries in Finite Precision}

It is a well-known fact that Haar unitary matrices can be numerically generated from complex Ginibre matrices. We refer the reader to \cite[Section 4.6]{edelman2005random} and \cite{mezzadri2006generate} for a detailed discussion. In this subsection we carefully analyze this process in finite arithmetic. 

The following fact (see \cite[Section 5]{mezzadri2006generate}) is the starting point of our discussion. 

\begin{lemma}[Haar from Ginibre]
\label{lem:HaarfromGinibre}
Let $G_n$ be a complex $n\times n$ Ginibre matrix and $U, R\in \mathbb{C}^{n\times n}$ be defined implicitly, as a function of $G_n$, by the equation $G_n = UR$ and the constraints  that $U$ is unitary and $R$ is upper-triangular with nonnegative diagonal entries\footnote{$G_n$ is almost surely invertible and under this event $U$ and $R$ are uniquely determined by these conditions. }. Then, $U$ is Haar distributed in the unitary group. 
\end{lemma}

The above lemma suggests that  $\QR(\cdot)$ can be used to generate random matrices that are approximately Haar unitaries. While doing this, one should keep in mind that when working with finite arithmetic, the matrix $\widetilde{G_n}$ passed to $\QR$ is not exactly Ginibre-distributed, and the algorithm $\QR$ itself incurs round-off errors.

Following  the discussion in Section \ref{sec:gaussians} we can assume that we have access to a random matrix $\widetilde{G}_n$, with 
$$
    \widetilde{G}_n = G_n +E,
    \marginnote{$\widetilde{G}_n$}
$$  
where $G_n$ is a complex $n\times n$ Ginibre matrix and $E\in \mathbb{C}^{n\times n}$ is an adversarial perturbation whose entries are bounded by $\frac{1}{\sqrt{n}} \cn \textbf{u}$. Hence, we have $\|E\|\leq \Vert E \Vert_F \le \sqrt{n} \cn \u$. 

In what follows we use $\qr (\cdot)$ to denote the exact arithmetic version of $\QR(\cdot)$. \marginnote{$\qr$} Furthermore, we assume that for any $A\in \mathbb{C}^{n\times n}$,  $\qr(A)$ returns a pair $(U, R)$ with the property that $R$ has nonnegative entries on the diagonal.  Since we want to compare $\qr(G_n)$ with $\QR(\widetilde{G}_n)$ it is necessary to have a bound on the condition number of the $QR$ decomposition. For this, we cite the following consequence of a result of Sun \cite[Theorem 1.6]{sun1991perturbation}:

\begin{lemma}[Condition number for the $QR$ decomposition \cite{sun1991perturbation}]
\label{lem:condQR}
Let $A, E\in \mathbb{C}^{n\times n}$ with $A$ invertible. Furthermore assume that $\|E\|\|A^{-1}\|\leq \frac{1}{2}$. If $(U, R) = \qr(A)$ and $(\widetilde{U},  \widetilde{R}) = \qr(A+E) $,  then 
\[ \Vert \widetilde{U} - U \Vert_F \le 4 \Vert A^{-1} \Vert \Vert E \Vert_F.\]
\end{lemma}

We are now ready to prove the main result of this subsection. As in the other sections devoted to finite arithmetic analysis, we will assume that $\u$ is small compared to $\mu_{\QR}(n)$; precisely, let us assume that
\begin{equation}
\label{eq:finiteassump2}
\u \mu_\QR(n)\leq  1. 
\end{equation}

\begin{proposition}[Guarantees for finite-arithmetic Haar unitary matrices]
\label{prop:finiteHaar}
    Suppose that $\QR$ satisfies the assumptions in Definition \ref{def:qr} and that it is designed to output upper triangular matrices with nonnegative entries on the diagonal\footnote{Any algorithm that yields the $QR$ decomposition can be modified in a stable way to satisfy this last condition at the cost of $O^*(n\log(1/\u))$ operations}. If $(V, R )=\QR(\widetilde{G_n})$, then there is a Haar unitary matrix $U$ and a random matrix $E$ such that $\widetilde{V} = U+E$. Moreover, for every $1>\alpha > 0$ and $ t > 2\sqrt{2}+1$ we have  
    $$
        \mathbb{P}\left[\|E\|< \frac{8t n^{\frac{3}{2}}}{\alpha}\cn \mu_\QR(n)  \u + \frac{10 n^2}{\alpha } \cn \u  \right] \ge 1-2e\alpha^2-\expboundd. 
    $$
\end{proposition}

\begin{proof}
    From our Gaussian sampling assumption, $\widetilde{G}_n = G_n +E$ where  $\|E\|\leq \sqrt{n} \cn \u$. Also, by the assumptions on $\QR$ from Definition \ref{def:qr}, there are matrices $\widetilde{\widetilde{G_n}}$ and $\widetilde{V}$ such that $(\widetilde{V}, R) = \qr(\widetilde{\widetilde{G_n}})$, and \begin{align*}
        \|\widetilde{V}- V\|
        &< \mu_\QR(n)\u \\ \|\widetilde{\widetilde{G}_n}- \widetilde{G_n}\| &\leq \mu_\QR(n)\u \|\widetilde{G}_n\| \leq \mu_\QR(n)\u\left(\|G_n\|+\sqrt{n}\cn\u\right).
    \end{align*}
    The latter inequality implies, using  $\eqref{eq:finiteassump2}$, that
    \begin{equation}
    \label{eq:boundnoisegin}
        \|\widetilde{\widetilde{G_n}}-G_n\| \leq \mu_\QR(n)\u\left(\|G_n\|+\sqrt{n}\cn\u\right)+ \sqrt{n}\cn\u \leq \mu_\QR(n)\u \|G_n\|+ 2\sqrt{n}\cn\u.
    \end{equation}
    
    Let $(U, R') := \qr(G_n)$. From Lemma \ref{lem:HaarfromGinibre} we know that $U$ is Haar distributed on the unitary group, so using (\ref{eq:boundnoisegin})  and   Lemma \ref{lem:condQR}, and the fact that $\Vert M \Vert \le \Vert M \Vert_F \le \sqrt{n}\Vert M \Vert$ for any $n \times n$ matrix $M$, we know that 
    \begin{equation}
    \label{eq:thenewineq}
    \|U-V\|- \mu_\QR(n)\u \leq \|U-V\|-\|\widetilde{V}-V \| \leq \|U- \widetilde{V}\| \leq 4\sqrt{n} \cn \mu_\QR(n)\u \|G_n\|\|G_n^{-1}\|+ 10n \cn \u \|G_n^{-1}\|. 
    \end{equation}
    Now, from $\|G_n^{-1}\| = 1/\sigma_{n}(G_n)$ and from  Theorem \ref{thm:szarek} we have that
    $$P \left[\|G_n^{-1} \|\geq \frac{n}{\alpha} \right]\leq (\sqrt{2e}\alpha)^{2}= 2e \alpha^2.$$
    On the other hand, from Lemma 2.2 of \cite{banks2019gaussian} we have $P\left[\left\|G_n\right\| > 2\sqrt{2}+ t \right] \leq e^{-n t^2} $. Hence, under the events $\|G_n^{-1}\|\leq \frac{n}{\alpha}$ and $\|G_n\|\leq 2\sqrt{2}+t$, inequality (\ref{eq:thenewineq}) yields
    $$\|U-V\|\leq \frac{4n^\frac{3}{2}}{\alpha} \cn \mu_\QR(n) \u\left(2\sqrt{2}+t+1\right)+\frac{10 n^2}{\alpha} \cn \u .$$
    Finally, if  $t>2\sqrt{2}+1$ we can exchange the term  $2\sqrt{2}+t+1$ for $2t$ in the bound. Then, using a union bound we obtain the advertised guarantee. 
\end{proof}

\subsection{Preliminaries of $\RURV$}
\label{sec:rurv}
Let $A\in \mathbb{C}^{n\times n}$ and $(U, R) =\rurv(A)$. As will become clear later, in order to analyze $\SPAN(A, r)$ it is of fundamental importance to bound the quantity $\|R_{22}\|$, where $R_{22}$ \marginnote{$R_{22}$} is the lower-right $(n-r)\times (n-r)$ block of $R$. To this end, it will suffice to use Corollary \ref{cor:boundonR22} below, which is the complex analog  to the upper bound given in equation (4) of  \cite[Theorem 5.1]{ballard2019generalized}.  Actually, Corollary \ref{cor:boundonR22} is a direct consequence  of  Lemma 4.1 in the aforementioned paper and Proposition \ref{prop:sigmamin} proved above. We elaborate below.  

\begin{lemma}[\cite{ballard2019generalized}]
\label{lem:boundonR22}
Let $n>r>0$,  $A\in \mathbb{C}^{n\times n}$ and $A = P \Sigma Q^*$ be its singular value decomposition. Let $(U, R) = \rurv(A)$, $R_{22}$ be the lower right $(n-r)\times (n-r)$ corner of $R$,  and $V$ be such that $A = URV$. Then, if $X = Q^* V^*$, 
$$\|R_{22}\| \leq \frac{\sigma_{r+1}(A)}{\sigma_{r}(X_{11}) },$$
where $X_{11}$ is the upper left $r\times r$ block of $X$. 
\end{lemma}

This lemma reduces the problem to obtaining a lower bound on $\sigma_{r}(X_{11})$. But, since $V$ is a Haar unitary matrix by construction and $X= Q^*V$ with $Q^*$ unitary, we have that $X$ is distributed as a Haar unitary. Combining Lemma \ref{lem:boundonR22} and Proposition \ref{prop:sigmamin} gives the following result. 

\begin{corollary}
\label{cor:boundonR22}
Let $n > r>0$,  $A\in \mathbb{C}^{n\times n}$, $(U, R)= \rurv(A)$ and $R_{22}$ be the lower right $(n-r)\times (n-r)$ corner of $R$. Then for any $\theta > 0$
$$\mathbb{P}\left[ \|R_{22}\| \leq \frac{\sqrt{r(n-r)}}{\theta}\sigma_{r+1}(A) \right] \geq 1-\theta^2. $$
\end{corollary}

\subsection{Exact Arithmetic Analysis of $\SPAN$}
\label{sec:deflateguaranties}

It is a standard consequence of the properties of the $QR$ decomposition that if $A$ is a matrix of rank $r$, then almost surely $\spa (A, r)$ is a $n\times r$ matrix with orthonormal columns that span the range of $A$. As a warm-up let's recall this argument. 


Let $(U, R) = \rurv (A)$ and $V$ be the unitary matrix used by the algorithm to produce this output. Since we are working in exact arithmetic, $V$ is a Haar unitary matrix, and hence it is almost surely invertible. Therefore, with probability 1 we have that  $\rank(AV^*) = r$ and that the first $r$ columns of $AV^*$ are linearly independent, so since $UR$ is the QR decomposition of $AV^*$, almost surely,  $R_{22} =0$ and $R_{11} \in \mathbb{C}^{r\times r}$, where $R_{11}$ and $R_{22}$ are as in (\ref{eq:decofR}). Writing
$$
    U = \begin{pmatrix}  U_{11} & U_{12} \\ U_{21} & U_{22}  \end{pmatrix}
$$ 
for the block decomposition of $U$ with $U_{11} \in \mathbb{C}^{r\times r}$, note that 
\begin{equation}
\label{eq:RURV}
AV^* =  UR = \begin{pmatrix}  U_{11} R_{11} & U_{11}R_{12}+U_{12} R_{22} \\ U_{21} R_{11}  & U_{21} R_{12}+U_{22} R_{22}  \end{pmatrix}.
\end{equation}
On the other hand, almost surely the first $r$ columns of $AV^*$ span the range of $A$. Using the right side of equation \eqref{eq:RURV} we see that this subspace also coincides with the span of the first $r$ columns of $U$, since $R_{11}$ is invertible. 

We will now prove a robust version of the above observation for a large class of matrices, namely those $A$ for which $\rank(A) = \rank(A^2)$.\footnote{For example, diagonalizable matrices satisfy this criterion.} We make this precise below and defer the proof to the end of the subsection. 

\begin{proposition}[Main guarantee for $\spa $]
\label{prop:spanexact}
Let $\indef > 0$ and $A, \widetilde{A}\in \mathbb{C}^{n\times n}$ be such that $\|A-\widetilde{A}\|\leq \indef$ and  $\rank(A)= \rank(A^2) =r$. Denote $S := \spa(\widetilde{A}, r)$ and $T := \spa(A, r)$. Then, for any $\theta\in (0, 1) $, with  probability $1-\theta^2$ there exists a unitary $\mathcal{}$ $W\in \mathbb{C}^{r\times r}$ such that
\begin{equation}
\label{eq:mainguaranteeSPAN}
\|S- TW^*\| \leq  \sqrt{\frac{8\sqrt{r(n-r)} }{ \sigma_r(T^*AT)}}\cdot \sqrt{\frac{\indef}{\theta}}. 
\end{equation}
\end{proposition}

\begin{remark}[The projector case]
\label{rem:projector}
In the case in which the matrix $A$ of Proposition \ref{prop:spanexact} is a (not necessarily orthogonal) projector, $T^*AT = I_r$, and the $\sigma_r$ term in the denominator of \eqref{eq:mainguaranteeSPAN} becomes a $1$.
\end{remark}

We begin by recalling a result about the  stability of singular values which will be important throughout this section. This fact is a consequence of Weyl's inequalities; see for example \cite[Theorem 3.3.16]{horn2012matrix} .

 \begin{lemma}[Stability of singular values]
 \label{lem:Weylineq}
 Let $X, E \in \mathbb{C}^{n\times n}$. Then, for any $k=1, \dots, n$ we have 
 $$|\sigma_k(X+E)- \sigma_k(X)|\leq \|E\|.$$
 \end{lemma}

We now show that the orthogonal projection $P:=\spa(\widetilde{A}, r)\spa(\widetilde{A}, r)^*$ is close to a projection onto the range of $A$, in the sense that $P A \approx A$. 
\begin{lemma}
\label{lem:intermediatespan}
    Let $\indef > 0$ and $A, \widetilde{A}\in \mathbb{C}^{n\times n}$ be such that $\rank(A)=r$  and $\|A-\widetilde{A}\|\leq \indef$. Let $(U, R) := \rurv(\widetilde{A})$ and $S := \spa(\widetilde{A}, r)$. Then, almost surely
    \begin{equation} \label{eq:intermediatebound}
        \|(SS^* -I_n)A\| \leq \|R_{22}\|+ \indef,
    \end{equation}
    where $R_{22}$ is the lower right $(n-r)\times (n-r)$ block of $R$. 
\end{lemma}

\begin{proof}
We will begin by showing that $\|(SS^*  -I_n) \widetilde{A}\|$ is small. Let  $V$ be the  unitary matrix that was used to generate $(U, R)$. As $\spa(\cdot, \cdot)$ outputs the first $r$ columns of $U$, we have the block decomposition $U = \begin{pmatrix} S & U' \end{pmatrix}$, where $S \in \mathbb{C}^{n\times r}$ and $U' \in \mathbb{C}^{n\times (n-r)}$. 

On the other hand we have  $\widetilde{A} = U RV$, so
$$
    (SS^* - I_n) \widetilde{A} = (SS^* - I) \begin{pmatrix} S & U' \end{pmatrix} R V = \begin{pmatrix} 0 & -U' \end{pmatrix} R V = \begin{pmatrix} 0 & -U' R_{2,2} \end{pmatrix} V .
$$
Since $\|U'\| = \|V\| =1$ from the above equation we get $\|(SS^*-I_n)\widetilde{A}\| \leq \|R_{22}\|$. 
Now we can conclude that
$$
    \|(SS^* -I_n) A\| \leq \|(SS^* -I_n) \widetilde{A}\|+ \|(SS^* -I_n) (A-\widetilde{A})\| \leq \|R_{22}\|+\indef. 
$$
\end{proof}

The inequality (\ref{eq:intermediatebound}) can be applied to quantify the distance between the ranges of $\spa(\widetilde{A}, r)$ and $\spa(A, r)$ in terms of $\|R_{22}\|$, as the following result shows.

\begin{lemma}[Bound in terms of $\|R_{22}\|$]
\label{lem:boundintermsofR22}
    Let $\indef > 0$ and $A, \widetilde{A}\in \mathbb{C}^{n\times n}$ be such that $\rank(A)= \rank(A^2) =r$  and $\|A-\widetilde{A}\|\leq \indef$. Denote by $(U, R) :=\rurv(\widetilde{A})$, $S := \spa(\widetilde{A}, r)$ and $T := \spa(A, r)$. Then, almost surely there exists a unitary $W\in \mathbb{C}^{r\times r}$ such that
    \begin{equation}
    \label{eq:boundintermsofR22}
    \|S- TW^*\| \leq 2\sqrt{\frac{\|R_{22}\|+\indef}{\sigma_r(T^*AT)}}  ,
    \end{equation}
    where $R_{22}$ is the lower right $(n-r)\times (n-r)$ block of $R$. 
\end{lemma}

\begin{proof}
    From Lemma \ref{lem:intermediatespan} we know that almost surely $\|(SS^*-I_n)A\| \leq \|R_{22}\|+ \indef$. We will use this to show that $\|T^*SS^*T- I_r\|$ is small, which can be interpreted as  $S^*T$ being close to unitary. First note that
    \begin{equation}
    \label{eq:TSST}
        \|T^*SS^*T-I_r\| = \sup_{w\in \mathbb{C}^r, \|w\|=1}\|T^*(SS^*-I_r)T w\| = \sup_{ w \in \range(A), \|w\|=1  } \|T^*(SS^*-I_r) w\|.
    \end{equation}
    Now, since $\rank(A) = \rank(A^2)$,  if $w \in \range(A)$ then $w = Av$ for some $v\in \range(A)$. So by the Courant-Fischer formula 
    $$
        \frac{\|w\|}{\|v\|} = \frac{\|Av\|}{\|v\|} \geq \inf_{u\in \range(A)} \frac{\|Au\|}{\|u\|} =  \sigma_r(T^*AT).
    $$
    We can then revisit (\ref{eq:TSST}) and get
    \begin{equation} \label{eq:sigmarTSST}
        \sup_{ w \in \range(A), \|w\|=1  } \|T^*(SS^*-I_r) w\| = \sup_{v \in \range(A), \|v\|\leq 1} \frac{\|T^*(SS^*-I_r) A v\|}{\sigma_r(T^*AT)} \leq \frac{\|T^*(SS^*-I_r) A T \|}{\sigma_r(T^*AT)}.   
    \end{equation}
    
    On the other hand $\|T^*(SS^*-I_r) A T \|\leq \|(SS^*-I_r) A\| \leq \|R_{22}\|+\indef$, so combining this fact with (\ref{eq:TSST}) and (\ref{eq:sigmarTSST}) we obtain
    $$
        \|T^*SS^*T- I_r\| \leq \frac{\|R_{22}\|+\indef}{\sigma_r(T^*AT)}.
    $$
    Now define $X:= S^*T$, $\indef' :=\frac{\|R_{22}\|+\indef}{\sigma_r(T^*AT)} $ and let $X= W | X|$ be the polar
    decomposition of $X$. Observe that
    $$\| |X| - I_r \| \le \sigma_1(X) - 1 \le |\sigma_1(X)^2 - 1| = \| X^* X - I_r \| \le \indef'.$$
    Thus  $\|S^*T- W\| = \|X-W\| = \| (|X|-I_n) W \| \leq \indef '.$ 
    Finally note that
    \begin{align*}
        \|S - TW^\ast\|^2 &= \|(S^\ast - WT^\ast)(S - TW^\ast)\| \\
        &= \|2I_r - S^\ast TW^\ast - WT^\ast S\| \\
        &= \|2I_r - S^\ast T(T^\ast S + W^\ast - T^\ast S) - (S^\ast T + W - S^\ast T)T^\ast S\| \\
        &\le 2\|I_r - S^\ast T T^\ast S\| +  \|S^\ast T(W^\ast - T^\ast S)\| + \|(W - S^\ast T)T^\ast S\| \le 4\indef', 
    \end{align*}
    which concludes the proof. 
\end{proof}

Note that so far our results have been deterministic. The possibility of failure of the guarantee given in Proposition \ref{prop:spanexact} comes from the non-deterministic bound on $\|R_{22}\|$. 

\begin{proof}[Proof of Proposition \ref{prop:spanexact}]
From Lemma \ref{lem:Weylineq} we have $\sigma_{r+1}(\widetilde{A})\leq \indef$. Now combine Lemma \ref{lem:boundintermsofR22} with Corollary \ref{cor:boundonR22}. 
\end{proof}

\subsection{Finite Arithmetic Analysis of $\SPAN$}
\label{sec:deflatefiniteguaranties}

In what follows we will have an approximation $\widetilde{A}$ of a matrix $A$ of rank $r$ with the guarantee that $\|A-\widetilde{A}\|\leq \indef$.  

For the sake of readability we will not present optimal bounds for the error induced by roundoff, and we will assume that 
\begin{equation}
\label{eq:finiteassumtions}
    4 \|A\|\cdot \max\{\cn\mu_{\MM}(n) \textbf{u}, \cn \mu_{\QR}(n) \textbf{u}\} \leq \indef \leq \frac{1}{4}\leq \|A\| 
    \quad \mathrm{and} \quad  
    1\leq \min \{\mu_\MM(n), \mu_\QR(n), \cn\}. 
\end{equation}

We begin by analyzing the subroutine $\RURV$ in finite arithmetic. This was done in \cite[Lemma 5.4]{demmel2007fast}. Here we make the constants arising from this analysis explicit and take into consideration that Haar unitary matrices cannot be exactly generated in finite arithmetic.   

\begin{lemma}[$\RURV$ analysis]
\label{lem:finiterurv} 
Assume that $\QR$ and $\MM$ satisfy the guarantees in  Definitions \ref{def:MM} and \ref{def:qr}. Also suppose that the assumptions in (\ref{eq:finiteassumtions}) hold. Then, if $(U, R) := \RURV (A)$ and $V$ is the matrix used to produce such output, there are unitary matrices $\widetilde{U}, \widetilde{V}$ and a matrix $\widetilde{A}$ such that $\widetilde{A} = \widetilde{U} R \widetilde{V} $ and the following guarantees hold:
\begin{enumerate}
     \item $\|U-\widetilde{U}\| \leq \mu_\QR(n) \textbf{u}$. 
    \item $\widetilde{V}$ is Haar distributed in the unitary group. 
    \item For every $1> \alpha >0$ and $t >2 \sqrt{2}+1$, the event: 
    \begin{equation}
    \label{eq:boundsRURV}
    \|\widetilde{V}-V\| < \frac{8 t n^{\frac{3}{2}}}{\alpha} \cn \mu_\QR(n)  \u + \frac{10 n^2}{\alpha } \u \quad \mathrm{and} \quad \|A-\widetilde{A}\| < \|A\|\left(\frac{9 t n^{\frac{3}{2}}}{\alpha} \cn \mu_\QR(n) \u + 2 \mu_\MM(n)\u + \frac{10n^2}{\alpha}\cn \u\right)
    \end{equation}
    occurs with probability at least $1- 2e \alpha^2- \expboundd$. 
\end{enumerate}
\end{lemma}

\begin{proof}
By definition $V=\QR (\widetilde{G_n})$ with $\widetilde{G}_n = G_n+E$, where $G_n$ is an $n\times n$ Ginibre matrix and $\|E\|\leq \sqrt{n} \u$. A direct application of the guarantees on each step yields the following: 
\begin{enumerate}
    \item From Proposition \ref{prop:finiteHaar}, we know that there is a Haar unitary $\widetilde{V}$ and a random matrix $E_0$, such that $V = \widetilde{V}+E_0$  and 
    \begin{equation}
    \label{eq:rurvE0}
    \mathbb{P}\left[\|E_0\| < \frac{8 t n^{\frac{3}{2}}}{\alpha} \cn \mu_\QR(n)  \u + \frac{10 n^2}{\alpha } \cn \u \right]
\geq 1-2e\alpha^2 -\expboundd.   
\end{equation}
    \item If $B:= \MM(A, V^*) = AV^* +E_1$, then from the guarantees for $\MM$ we have  $\|E_1\|\leq  \|A\|\|V\| \mu_\MM (n)\textbf{u}$. Now from the guarantees for $\QR$ we know that $V$ is $\mu_\QR(n) \u$ away from a unitary, and hence $$\|V\|\mu_\MM(n) \u \leq (1+\mu_\QR(n)\u) \mu_\MM(n)  \u\leq  \frac{5}{4}\mu_\MM(n)\u$$
   where the last inequality follows from the assumptions in (\ref{eq:finiteassumtions}). This translates into $$\|B\| \leq \|A\|\|V\|+ \|E_1\| \leq (1+ \mu_\QR(n)\u) \|A\|+ \|E_1\| \leq \frac{5}{4}\|A\|+ \|E_1\|.$$
    Putting the above together and using 
    (\ref{eq:finiteassumtions}) again, we get 
    \begin{equation}
    \label{eq:E1}
        \|E_1\|\leq \frac{5}{4} \|A\|  \mu_\MM(n) \textbf{u} \quad \text{and} \quad B \leq \frac{5}{4}\|A\|(1+ \mu_\MM(n)\textbf{u})  <  2\|A\| .
    \end{equation}
    \item Let $(U, R) = \QR(B)$.  Then there is a unitary $\widetilde{U}$ and a matrix $\widetilde{B}$ such that $U= \widetilde{U} + E_2$, $B = \widetilde{B}+E_3$, and $\widetilde{B} = \widetilde{U} R$, with error bounds  $\|E_2\| \leq \mu_\QR (n) \textbf{u}$ and $\|E_3\|\leq \|B\|\mu_\QR(n) \textbf{u}$.  Using (\ref{eq:E1}) we obtain
    \begin{equation}
    \label{eq:rurvE3}
    \|E_3\| \leq \|B\| \mu_\QR (n) \textbf{u} < 2\|A\| \mu_\QR(n) \textbf{u}. 
    \end{equation}
    
    \item Finally, define $\widetilde{A} := \widetilde{B} \widetilde{V}$. Note that $\widetilde{A} = \widetilde{U} R \widetilde{V}$ and 
    $$ \widetilde{A} = \widetilde{B} \widetilde{V} = (B-E_3) \widetilde{V} = (AV^* +E_1-E_3) \widetilde{V} = (A(\widetilde{V}+E_0)^*+E_1-E_3 )\widetilde{V} = A + (AE_0^*+E_1-E_3)\widetilde{V}, $$
    which translates into
    $$\|A-\widetilde{A}\| \leq \|A\|\|E_0\| + \|E_1\|+ \|E_3\|. $$
     Hence, on the event described in the left side of (\ref{eq:rurvE0}), we have 
     $$\|A- \widetilde{A}\| \leq \|A\|\left(  \frac{8t n^{\frac{3}{2}}}{\alpha} \cn \mu_\QR(n)  \u + \frac{10 n^2}{\alpha } \cn \u +\frac{5}{4}\mu_\MM(n)\u+2 \mu_\QR(n)\u \right),$$
     and using some crude bounds, the above inequality yields the advertised bound. 
\end{enumerate}

\end{proof}

We can now prove a finite arithmetic version of Proposition \ref{prop:spanexact}. 

\begin{proposition}[Main guarantee for $\SPAN$]
\label{prop:spanfinite}
Let $n> r $ be positive integers, and let $\indef,\theta  > 0$ and $A, \widetilde{A}\in \mathbb{C}^{n\times n}$ be such that $\|A-\widetilde{A}\|\leq \indef$ and  $\rank(A)= \rank(A^2) =r$. Let $S := \SPAN(\widetilde{A}, r)$ and $T := \spa(A, r)$. If $\QR$ and $\MM$ satisfy the guarantees in Definitions \ref{def:MM} and \ref{def:qr}, and (\ref{eq:finiteassumtions}) holds,  then, for every $t > 2\sqrt{2}+1$  there exist a unitary $W\in \mathbb{C}^{r\times r}$ such that
\begin{equation}
\label{eq:finitedeflate}
\| S-  T W^*\| \leq \mu_\QR(n)\textbf{u} + 12\sqrt{\frac{tn^2\sqrt{r(n-r)} }{\sigma_r(T^*AT)}}.\sqrt{\frac{\indef}{\theta^2}}, 
\end{equation}
with probability at least $1-7\theta^2- \expboundd$. 
\end{proposition}

\begin{proof}
Let $(U, R) = \RURV(\widetilde{A})$. From Lemma \ref{lem:finiterurv} we know that there exist  $\widetilde{U}, \widetilde{\widetilde{A}} \in \mathbb{C}^{n\times n}$, such that $\|U-\widetilde{U} \|$ and $\|\widetilde{A}- \widetilde{\widetilde{A}}\|$ are small, and $(\widetilde{U}, R) = \rurv(\widetilde{\widetilde{A}})$ for the respective realization of an exact Haar unitary matrix. Then, from $\|\widetilde{A}\|\leq \|A\|+ \indef$ and (\ref{eq:boundsRURV}), for every $1 > \alpha >0$ and $t > 2 \sqrt{2}+1$ we have 
\begin{equation}
\label{eq:boundtildetildeA}
\left\|A-\widetilde{\widetilde{A}}\right\| \leq \left\|\widetilde{\widetilde{A}}-\widetilde{A}\right\|+ \|\widetilde{A}- A\| \leq (\|A\|+\indef)\left(\frac{9 t n^{\frac{3}{2}}}{\alpha} \mu_\QR(n) \cn\u + 2 \mu_\MM(n)\u + \frac{10n^2}{\alpha}\cn \u\right)+ \indef, 
\end{equation}
with probability $1-2e\alpha^2-\expboundd$. 

Now, from (\ref{eq:finiteassumtions}) we have $\u\leq \indef\leq \frac{1}{4}$ and $\cn \|A\| \mu\u \leq \indef$ for $\mu= \mu_\QR(n), \mu_\MM(n)$, so we can bound the respective terms in (\ref{eq:boundtildetildeA}) by $\indef$:
\begin{align}
\label{eq:boundboundbound}
(\|A\|+\indef)\left(\frac{9 t n^{\frac{3}{2}}}{\alpha} \cn \mu_\QR(n) \u + 2 \mu_\MM(n)\u + \frac{10n^2}{\alpha} \cn\u\right)+ \indef & \leq (1+ \indef)\left(\frac{9 t n^\frac{3}{2}}{\alpha}\indef+2\indef+\frac{10n^2}{\alpha}\indef  \right)+\indef \nonumber  \\ &\leq \frac{(12t+16)}{\alpha}n^2 \indef, 
\end{align}
where the last crude bound uses $1 \leq n^{\frac{3}{2}}\leq n^2, 1+\indef \leq \frac{5}{4}$ and $t > 2$.  

Observe that $\widetilde{S} = \spa(\widetilde{\widetilde{A}}, r)$ is the matrix formed by the first $r$ columns of $\widetilde{U}$, and that by Proposition \ref{prop:spanexact} we know that for every $\theta > 0$,  with probability $1-\theta^2$ there exists a unitary $W$ such that 
\begin{equation}
\label{eq:eventfromrurv}
\|\widetilde{S}- T W^*\| \leq  \sqrt{\frac{8\sqrt{r(n-r)} }{\sigma_r(T^*AT)}}.\sqrt{\frac{\left\|A-\widetilde{\widetilde{A}}\right\|}{\theta}}.  
\end{equation}
On the other hand, $S$ is the matrix formed by the first $r$ columns of $U$. Hence 
$$\|S-\widetilde{S}\|\leq \|U- \widetilde{U}\| \leq \mu_\QR (n) \textbf{u}. $$
Putting the above together we get that under this event
\begin{align}
\label{eq:finalinequality}
\|S- T W^*\|  \leq  \|S-\widetilde{S}\|+ \|\widetilde{S}- TW^*\| 
  \leq  \mu_{\QR} (n)\textbf{u} +  \sqrt{\frac{8\sqrt{r(n-r)} }{\sigma_r(T^*AT)}}.\sqrt{\frac{\left\|A-\widetilde{\widetilde{A}}\right\|}{\theta}}.
\end{align}
 Now, taking $\alpha = \theta$, we note that both  events in  (\ref{eq:boundtildetildeA}) and (\ref{eq:eventfromrurv}) happen with probability at least $1-(2e+1)\theta^2-\expboundd$. The result follows from replacing the constant $2e+1$ with 7, using $t> 2\sqrt{2}+1$ and replacing $8(12t+16)$ with $144t$, and combining the inequalities (\ref{eq:boundtildetildeA}), (\ref{eq:boundboundbound}) and (\ref{eq:finalinequality}). 
\end{proof}

We end by proving Theorem \ref{thm:deflate-guarantee-usable}, the guarantees on $\SPAN$ that we will use when analyzing the main algorithm.

\begin{proof}[Proof of Theorem \ref{thm:deflate-guarantee-usable}]
    As Remark \ref{rem:projector} points out, in the context of this theorem we are passing to $\SPAN$ an approximate projector $\widetilde{P}$, and the above result simplifies. Using this fact, as well as the upper bound $r(n-r) \le n^2/4$, we get that
    $$
        \|S - TW^\ast\| \le \mu_{\QR}(n)\u + \frac{12 \sqrt{t n^3 \indef}}{\theta}.
    $$
    with probability at least $1 - 7\theta^2 - 2 e^{-t^2 n}$ for every $t > 2\sqrt 2$. If our desired quality of approximation is $\|S - TW^\ast\| = \outdef$, then some basic algebra gives the success probability as at least
    $$
        1 - 1008\frac{n^3 t \indef}{(\outdef - \mu_{\QR}(n)\u)^2} - 2e^{-t^2n}.
    $$
    Since $\indef \le 1/4$, we can safely set $t = \sqrt{2/\indef}$, giving 
    $$
        1 - 1426\frac{n^3\sqrt{\indef}}{(\outdef - \mu_{\QR}(n)\u)^2} - 2e^{-2n/\indef}.
    $$
    To simplify even further, we'd like to use the upper bound $2e^{-2n/\indef} \le \frac{n^3\sqrt\indef}{(\outdef - \mu_{\QR}(n)\u)^2}$. These two terms have opposite curvature in $\indef$ on the interval $(0,1)$, and are equal at zero, so it suffices to check that the inequality holds when $\indef = 1$. The terms only become closer by setting $n=1$ everywhere except in the argument of $\mu_{\QR}(\cdot)$, so we need only check that
    $$
        \frac{2}{e^2} \le \frac{1}{(\outdef - \mu_{\QR}(n)\u)^2}.
    $$
    Under our assumptions $\outdef,\mu_{\QR}(n)\u \le 1$, the right hand side is greater than one, and the left hand less. Thus we can make the replacement, use $\mach \le \tfrac{\outdef}{2\mu_{QR}(n)}$, and round for readability to a success probability of no worse than
    $$
        1 - 6000\frac{n^3\sqrt\indef}{\outdef^2};
    $$
    the constant here is certainly not optimal.
    
    Finally, for the running time, we need to sample $n^2$ complex Gaussians, perform two QR decompositions, and one matrix multiplication; this gives the total bit operations as
    $$
        T_{\SPAN}(n) = n^2 T_\N+ 2T_\QR(n)+T_\MM(n).
    $$
\end{proof}

\begin{remark}
Note that the exact same proof of Theorem \ref{thm:deflate-guarantee-usable} goes through in the more general case where the matrix in question is not necessarily a projection, but any matrix close to a rank-deficient matrix $A$. In this case an extra $\sigma_r(T^*AT)$ term appears in the probability of success (see the guarantee given in the box for the Algorithm $\SPAN$ that appears in this appendix). 
\end{remark}
	\section{Alternate Proofs of Shattering and Davies' Conjecture} \label{sec:proofsforABBCS}

In this section we'll give an alternate and essentially different route to the smoothed analysis of eigenvalue gap and condition numbers in Theorem \ref{thm:smoothed}, as well as the proof of Davies' Conjecture in \cite{banks2019gaussian}, via results from \cite{abbcs}. We'll begin by recalling some notation from \cite{abbcs}, and direct the reader there for a more thorough treatment.

For any $n$ let $\P(\C^n)$ denote the projective space associated to $\C^n$, and given $A \in \C^{n \times n}$, $\lambda \in \C$ and $v \in \P(\C)$, define $A_{\lambda, v}: v^{\perp} \to v^{\perp}$ by
\[A_{\lambda, v} := P_v^\perp \, \circ\,  (A - \lambda) \vert_{v^\perp}\]
where $v^{\perp} = \{ x \in \C^n \mid \langle x, v \rangle = 0 \}$ and $P_{v^\perp} : \C^n \to v^\perp$ denotes the orthogonal projection. With this in hand, \cite{abbcs} defines the condition number of a triple $(A, \lambda, v)\in \C^{n\times n}\times \C\times \P(\C^n)$ as 
\[ \mu(A, \lambda, v) := \begin{cases}\Vert A \Vert_F \Vert A_{\lambda, v}^{-1} \Vert & \text{if } A_{\lambda, v} \text{ is invertible,}
\\ \infty & \text{otherwise.}\end{cases} \]
They similarly define the mean square condition number of a matrix as 
\[ \muav(A) := \left(\frac{1}{n} \sum_{j=1}^n \Vert A \Vert_F^2 \Vert \Vert A_{\lambda_j, v_j}^{-1} \Vert_F^2 \right)^\frac{1}{2},\]
where $(\lambda_j, v_j)$ are the eigenpairs of $A$. In particular, note that $\muav(A)<\infty$ only when $A$ has simple eigenvalues, and therefore $\muav(A)<\infty$ implies that $A$ is diagonalizable.

\subsection{Davies' conjecture}

To compare the notions of eigenvalue condition number and the condition number of a triple we recall the following theorem from \cite{abbcs}:

\begin{theorem}[Part of Proposition 2.7 of \cite{abbcs}] \label{thm:abbcs_eig_derivative}
Let $\mathcal{V}$ denote the \emph{solution variety} for the eigenpair problem, defined as 
\[ \mathcal{V} = \mathcal{V}_n := \{(A, \lambda, v) \in \mathbb{C}^{n \times n} \times \C  \times \mathbb{P}(\C) \mid (A - \lambda)v = 0 \},\]
and let $\Gamma : [0, 1] \to \mathcal{V}$, $\Gamma(t) = (A_t, \lambda_t, v_t)$ be a smooth curve such that $A_t$ lies in the unit sphere of $\mathbb{C}^{n \times n}$ for all $t$.   Then for all $t \in [0, 1]$, 
\[  |\dot{\lambda_t}| \le \sqrt{1 + \mu(A_t, \lambda_t, v_t)^2} \Vert \dot{A_t} \Vert. \]
\end{theorem}

Now recall that $\kappa(\lambda)$ has the following variational description (e.g. see Theorem 1 in \cite{greenbaum2020first}) for any a simple eigenpair $(\lambda, v)$ of $A$, in terms of the derivatives of smooth curves going through the point $(A, \lambda, v)$. Namely
$$\kappa(\lambda)= \sup_{\substack{\Gamma:[0, 1] \to \calV, \, \Gamma(0)=(A, \lambda, v)}} \frac{|\dot{\lambda}_0|}{\|\dot{A}_0\|}.$$
Hence, Theorem \ref{thm:abbcs_eig_derivative} implies
\begin{equation}
    \label{eq:boundkappawithmu}
    \kappa(\lambda) \leq \sqrt{1+\mu(A, \lambda, v)^2}. 
\end{equation}
It is then clear that $\muav(A)$ can also be used to upper bound $\kappa_V(A)$. In view of this, we remind the reader of the following result from \cite{abbcs}.  
\begin{theorem}[Theorem 2.14 of \cite{abbcs}]
\label{thm:abccsboundonkappaV}
Let $G_n \in \mathbb{C}^{n \times n}$ denote a complex Ginibre matrix with $\mathcal{N}(0, 1_\C/n)$ entries.  For any $A \in \C^{n \times n}$ and $\gamma > 0$, we have 
\[ \dE \left[\frac{\mu_{F, \mathsf{av}}(A + \gamma G_n)^2}{\Vert A+ \gamma G_n \Vert_F^2} \right] \le \frac{n^2 }{\gamma^2}.\]
\end{theorem}
We are now ready to prove the following result, which directly implies Davies' conjecture (for comparison, Theorem 1.1 of \cite{banks2019gaussian} is the same result but the exponent of $n$ is $3/2$ instead of $5/2$.)
\begin{proposition}
Suppose $A \in \C^{n \times n}$ and $\gamma \in (0,1)$.  Then there is a matrix $E \in \C^{n \times n}$ such that $\Vert E \Vert \le \gamma \Vert A \Vert$ 
and
\[ \kappa_V(A + E) \le C \frac{n^{5/2}}{\gamma}\]
where $C$ is an absolute constant.
\end{proposition}

\begin{proof} 
Let $\lambda_i, v_i$ denote the (random) eigenvalues and eigenvectors of $A + \gamma G_n$.  Let $B_r$ denote the event $\Vert A + \gamma G_n \Vert_F < r$.  Because $ \Vert G_n \Vert_F < 2\sqrt{n}$ with probability at least some absolute positive constant, for $ r = \Vert A \Vert + 2 \sqrt{n}$ the event $B_r$ holds with that probability as well. Now note that
\begin{align}
\dE \left[ \sum_i \kappa(\lambda_i)^2 \mid B_r \right]&\le \dE \left[ n + \sum_i \mu(A + \gamma G_n, \lambda_i, v_i)^2  \mid B_r \right] && \text{by (\ref{eq:boundkappawithmu})} \nonumber \\
&\le \dE\left[ n + n \mu_{F, \mathsf{av}}(A + \gamma G_n)^2 \mid B_r \right] \nonumber \\ \label{eq:boundabbcs}
&\le n + \frac{n^3 r^2}{\gamma^2 \P[B_r]},
\end{align}
where in the last line we use Theorem \ref{thm:abccsboundonkappaV} and
\[ \dE \left[ \left. \frac{\mu_{F, \mathsf{av}}(A + \gamma G_n)^2}{r^2} \, \right| B_r \right] \le \dE \left[ \left. \frac{\mu_{F, \mathsf{av}}(A + \gamma G_n)^2}{\Vert A+ \gamma G_n \Vert_F^2} \, \right| B_r \right] \le \frac{\dE \left[ \frac{\mu_{F, \mathsf{av}}(A + \gamma G_n)^2}{\Vert A+ \gamma G_n \Vert_F^2} \right]}{ \P[B_r]}.\]
Recalling the general inequality (see Lemma 3.1 \cite{banks2019gaussian}) $\kappa_V \le \sqrt{n \sum_{i=1}^n \kappa(\lambda_i)^2}$ we get
$$\dE[\kappa_V(A+\gamma G_n)^2| B_r] \leq n \dE \left[ \sum_i \kappa(\lambda_i)^2 \mid B_r \right].$$
So, when $\Vert A \Vert = 1$ and $\gamma < 1$,  if we set $ r = \Vert A \Vert + 2 \sqrt{n}$ as discussed above, the event $B_r$ occurs with positive probability, and by (\ref{eq:boundabbcs}) we know that $n \dE \left[ \sum \kappa(\lambda_i)^2 \mid B_r \right] \leq \frac{C n^5}{\gamma^2}$ for some constant $C$. It follows that there is some realization of $G_n$ for which $\kappa_V(A+\gamma G_n)^2 \leq \frac{C n^5}{\gamma^2}$, as we wanted to show. 
\end{proof}

\subsection{Smoothed analysis of $\gap$}

Let $M\in \C^{n\times n}$ be any matrix, and let $\lambda_1, \dots, \lambda_n$ be its eigenvalues. In what follows we will denote
$$\gap_i(M) := \min_{j\neq i} |\lambda_i -\lambda_j|.$$
We begin by comparing these quantities to the condition number of the corresponding triple. 

\begin{lemma}
\label{lem:gapandconditiontriple}
Let $M$ be a matrix with distinct eigenvalues and spectral decomposition $M=\sum_{i=1}^n \lambda_i v_i w_i^*$. Then, for every $i=1, \dots, n$ it holds that
$$\frac{\mu(M, \lambda_i, v_i)}{\|M\|_F} \geq \frac{1}{\gap_i(M)}. $$
\end{lemma}

\begin{proof}
First we show that $\Lambda(M_{\lambda_i, v_i}) = \Lambda(M-\lambda_i)\setminus \{0\}$. To see this, take any $j\neq i$ and note that
$$w_j^* P_{v_i^\perp} \, \circ \,  (M-\lambda_i) \vert_{v_i^\perp} = (\lambda_j-\lambda_i) w_i^* ,$$
and hence $\lambda_j-\lambda_i$ is an eigenvalue of $M_{\lambda_i, v_i}$. 

Now, using that the norm of a matrix is bigger than its spectral radius we get
\begin{align*}
 \|M_{\lambda_i, v_i}^{-1}\| & \geq \sup_{\lambda\in \Lambda(M_{\lambda_i, v_i})} \frac{1}{|\lambda|}
 \\ & = \frac{1}{\gap_i(M)} && \text{because } \Lambda(M_{\lambda_i, v_i}) = \Lambda(M-\lambda_i)\setminus \{0\}. 
\end{align*}
The claim then follows from the definition of $\mu(M, \lambda_i, v_i)$. 
\end{proof}

Using Theorem \ref{thm:abccsboundonkappaV} we get the following. 

\begin{proposition}
\label{prop:gapfromabbcs}
Let $A\in \C^{n\times n}$ be an arbitrary matrix and let $G_n$ be a normalized complex Ginibre matrix. Then for any $t, \gamma >0$
$$\P[\gap(A+\gamma G_n) < t \gamma ] \leq  n^3 t^2.$$
Thus, $\gap(A+\gamma G_n)=O(\gamma/n^{3/2})$ with probability bounded away from zero. 
\end{proposition}

\begin{proof}
Using Lemma \ref{lem:gapandconditiontriple} we get 
$$\frac{1}{\gap(A+\gamma G_n)^2} = \max_{i} \frac{1}{\gap_i(A+\gamma G_n)^2} \leq \max_i \frac{\mu(A+\gamma G_n, \lambda_i, v_i)^2}{\|A + \gamma G_n\|_F^2} \leq n \frac{\muav(A+\gamma G_n)^2}{\|A+\gamma G_n\|_F^2}.$$
Combining this with Theorem \ref{thm:abccsboundonkappaV} we obtain
$$ \dE\left[\frac{1}{\gap(A+\gamma G_n)^2} \right] \leq \frac{n^3}{\gamma^2}.$$
The proof is then concluded using Markov's inequality. 
\end{proof}

Remarkably, the $\gamma$ dependence in the bound of Proposition \ref{prop:gapfromabbcs} is optimal,  partially answering Question 2 in Section \ref{sec:conclusion} posed by us in a previous version of this paper. Also note that the bound is stronger than that from Corollary \ref{cor:mingapbound} obtained with our techniques. That said, and as discussed in Remark \ref{rem:comparisontoabbcs}, the results in \cite{abbcs} that were necessary for the proof of Proposition \ref{prop:gapfromabbcs} heavily exploit that the random perturbation has a complex Gaussian distribution, and it is not clear how to extend these result to other distributions.

With this in mind, we publicize the following conjecture of independent interest in random matrix theory, which was  communicated to us  by Vishesh Jain:

\begin{conjecture}
Let $K>0$  and let  $M_n$ be an $n \times n$ random matrix with independent complex entries (not necessarily centered), whose distributions are absolutely continuous with respect to the Lebesgue measure on $\C$, and have density upper bounded by $K$. Then
$$\P\Big[\gap(M_n)< \frac{t}{K}\Big] \leq \poly(n, t) \quad \forall t>0, $$
where $\poly(n, t)$ is a universal polynomial (i.e. its coefficients are independent of the distributions of the entries of $M_n$) in $n$ and $t$, which  is zero when $t=0$. 
\end{conjecture}

\end{document}